%% file: paper.tex
\documentclass[11pt,reqno]{amsart}
   \input{preamble}
     \input{notation}

         \title[]{Global Well-Posedness For Half-Wave Maps With $S^2$ and $\Hy$ Targets For Small Smooth Initial Data}
         \date{}
         \author[]{Yang Liu}
         \address{B\^atiment des Math\'ematiques \\ EPFL \\ Station 8 \\ 1015 Lausanne \\ Switzerland}
         \email{y.liu@epfl.ch}
        
\begin{document}
        \begin{abstract}
            We prove global well-posedness for the half-wave map with $S^2$ target for small $\dot{H}^{\frac{n}{2}} \times \dot{H}^{\frac{n}{2}-1}$ initial data. We also prove the global well-posedness  for the equation with $\Hy$ target for small smooth $\besov \times \dot{B}^{\frac{n}{2}-1}_{2,1}$ initial data. 
        \end{abstract}
        \maketitle
      \setcounter{tocdepth}{1}
    \tableofcontents
        \section{Introduction}
        Let $u: \R^{n+1}\rightarrow M \subseteq \R^3$, where $M=\Hy$ or $S^2$, be smooth and bounded with the property that $\nabla_{t,x}u(t,\cdot)\in L^r(\R^n)$ for some $r\in (1,\infty)$, and furthermore $\lim_{|x|\rightarrow +\infty}u(t, x) = Q$ for some fixed $Q\in M$, for each $t$. We define the operator $\Lah u = -\sum_{j=1}^n (-\triangle)^{-\frac12}\partial_j (\partial_j u)$, and we say that $u$ is a {\it{half-wave map}} if it satisfies

\begin{equation}
    \label{eq:halfwm}
\partial_t u = u\times_{\eta} \Lah u.
\end{equation}
We define $a\times_{\eta} b:= (\eta,1,1) (a \times b)$, where $\eta=\pm 1$ for $M=S^2,\ \Hy$ respectably. \par

The model is formally related to the Schrödinger maps equation 
\[
u_t = u \times \triangle u,    
\]
and the classical wave equation
\[
    \Box u=\partial_\alpha \partial^\alpha u=-u \partial_\alpha u^T \partial^\alpha u.
\]
    
Therefore, one way to consider the problem is to reformulate \eqref{eq:halfwm} into a nonlinear wave type equation as in Krieger and Sire \cite{krieger2017small}:
\begin{equation}
    \label{eq:halfwm-wave}
    \begin{split}
        \Box u = (\partial^2_t - \Delta) u =
    &(\nabla u \cdot \nabla u-\partial_t u \cdot \partial_t u) u \\
    &+\Pi_{u_{\perp}}\left[\ (u \cdot \Lah u )\ \Lah u \right]  \\
    &+\Pi_{u_{\perp}} \left[u \times \Lah \left(u \times \Lah u \right) - u \times \left(u \times (-\Delta) u \right)\right] 
    \end{split}
\end{equation} \par

The advantage of this reformulation is that we can use the well-established wave equation theory to solve the problem (e.g. \cite{2000mathTao}, \cite{sterbenz2005global}). We first prove the global well-posedness of \eqref{eq:halfwm-wave} with small critical initial data, then we will show that such a solution $u$ also solves the original half-wave map equation \eqref{eq:halfwm} in Theorem \ref{thm:solution}.

The global well-posedness of the Cauchy problem \eqref{eq:halfwm} with target $S^2$ for small $\besov\times \dot{B}^{\frac{n}{2}-1}_{2,1}$ initial data was established by Krieger and Sire \cite{krieger2017small} for $n\geq 5$. The result was later improved by Kiesenhofer and Krieger \cite{2019arXiv190412709K} to $n=4$.  \par
In this work, the author establishes the global well-posedness with $S^2$ target \eqref{eq:halfwm} for small $\dot{H}^{\frac{n}{2}} \times \dot{H}^{\frac{n}{2}-1}$ initial data and the global well-posedness problem for small smooth $\besov \times \dot{B}^{\frac{n}{2}-1}_{2,1}$ initial data with target $\Hy$. \par

\subsection{General Kähler Manifold}
 We propose a general formulation of the half-wave equation with a Kähler manifold target. Consider a Kähler manifold $(M, g, J, \omega)$, where $M$ 
is a smooth $m$-dimensional complex manifold. It also has a structure of a real $2m$-dimensional smooth manifold. $(M,g)$ is a smooth $2m$-dimensional Riemannian manifold embeds into $\R^k$, $k\geq 2m+1$, by the Nash embedding theorem. For a Kähler manifold, we have a compatible triple $(g, \omega, J)$, where $J$ is an integrable almost-complex structure of the complex manifold $M$ such that $J^2=-id$ and $\omega$ is a closed symplectic two-form defined by
$$\omega(X,Y):=g(JX,Y), \ \forall X,Y \in TM.$$ 

Moreover, a Kähler manifold is also Hermitian, hence we have $g(JX,JY)=g(X,Y)$, $\forall X,Y \in TM$. \par
The half-wave map $u$ defined on $\R^{n+1}$ with a Kähler manifold $M$ target satisfies
 \begin{equation}
    \label{eq:kahler}
     \partial_t u = J(u)\left[P_T(u) \Lah u \right].
 \end{equation}

 Here $P_T \in C^1 (\R^k, \mathcal{M}_{k}(\R))$ is the orthogonal projection onto $T_z M$ as in the work of Da Lio-Rivière \cite{DaLioFrancesca2016HaM} defined in the following:
 \begin{Definition}
    \label{tp}
    Let $P_T,P_N \in C^1(\mathbb{R}^k,M_{k}(\mathbb{R}))$ where
    $M_{k}(\mathbb{R})$ is the set of $k\times k$ matrices with entries taking values in
    $\mathbb{R}$. Let $P_T$  be the orthogonal projection of $T_z M$ and
    $P_N$ be the orthogonal projection onto $T_z^{\perp} M$ s.t.
    \begin{enumerate}[label=(\roman*)]
      \item $P_T \circ P_T =P_T$ \\
      \item $P_N \circ P_N =P_N$ \\
      \item$ P_T+P_N = 1_{k \times k}$ \\
      \item $\|\partial_z P_T \|_{L^{\infty}} < \infty$ \\
      \item $\langle P_T(z)U) \cdot (P_N(z)V \rangle=0 \quad \forall z\in \mathbb{R}^{k},\ \forall U,V\in T_z{\mathbb{R}^{k}}$
    \end{enumerate}
  \end{Definition}

We can view $S^2$ and $\Hy$ as Kähler manifolds with the almost complex structure $J(u):= u \times_{\eta}$. In this paper, we focus on the half-wave map equation with these two targets,
 \begin{equation*}
     \label{hwm:s}
     \partial_t u = u \times_{\eta} \Lah u.
 \end{equation*} 
\subsection{Critical $S^2$ Case}

We study the global well-posedness problem of the half-wave map with
smooth initial data $u[0]:=\big(u(0,\cdot),\partial_t u(0,\cdot) \big)=(u_0,u_1): \R^n \rightarrow S^2 \times TS^2$ s.t. 
$u_1 = u_0 \times \Lah u_0$, and 
\begin{equation*}
    \|u[0]\|_{\dot{H}^{\frac{n}{2}} \times \dot{H}^{\frac{n}{2}-1}} < \epsilon,
\end{equation*}
where $\epsilon  \ll 1$ is sufficiently small. \par
We improve the result from $\besov \subseteq  L^\infty$ to the slightly larger space $\dot{H}^{\frac{n}{2}}$ which does not embed into $L^\infty$. It is worth noting that both spaces share the same scaling invariance property.
As discussed above, the half-wave map can be reformulated into a nonlinear wave-type equation \eqref{eq:halfwm-wave}.  The global well-posedness problem for the wave map with critical initial data has been studied by Tao \cite{2000mathTao}, Klainerman-Rodnianski \cite{klainerman2001global} and
Shatah-Struwe \cite{shatah2002cauchy}. In particular, we adopt the frequency fold localization method and the approximately parallel transport argument from the work of Tao \cite{2000mathTao}. In Section \ref{sec:global}, we will show that all the additional nonlinear terms introduced in the reformulated half-wave equation \eqref{eq:halfwm-wave} are `` well-behaved ''  to apply such methods. \par
Our result is as follows.

\begin{theorem}
    \label{thm:s2}
    Let $n \geq 5$, $u: \R^{1+n} \rightarrow S^2$, such that $\lim_{|x|\rightarrow \infty} \ u(t,x) =Q$ for a fixed $Q\in S^2$, for all t. Let $u[0]:=\big(u(0,\cdot),\partial_t u(0,\cdot) \big)=(u_0,u_1): \R^n \rightarrow S^2 \times TS^2$ s.t. 
    $u_1 = u_0 \times \Lah u_0$ be a smooth pair of data satisfying \eqref{eq:halfwm} with $u(0,\cdot)$ is constant outside a compact subset of $\R^n$ (this ensure that $\Lah u(0,\cdot)$ is well defined ) and
    \begin{equation}
        \label{con:init_s2}
        \|u[0]\|_{\dot{H}^{\frac{n}{2}} \times \dot{H}^{\frac{n}{2}-1}} < \epsilon,
    \end{equation}
    where $\epsilon  \ll 1$ is sufficiently small. \par
     Then the Cauchy problem for \eqref{eq:halfwm} with initial data $u[0]$ admits a smooth global solution.
     
\end{theorem}
\begin{Remark}
   We restrict to $n\geq 5$ because we want to take advantage of the Strichartz estimate $L_t^2L_x^4$ for the nonlinearities, which is not available when $n=4$. In that case, we refer readers to techniques used in Kiesenhofer and Krieger \cite{2019arXiv190412709K}.
\end{Remark}

In order to prove this theorem, we first establish the local well-posedness of \eqref{eq:halfwm-wave} with more regular $\dot{H}^s \times \dot{H}^{s-1}, \ s>\frac{n}{2},$ initial data in Section \ref{sec:local}. Then we adopt the ideas from Tao \cite{T1} to extend this result to global well-posedness with small critical initial data in Section \ref{sec:global}. In the end, we show that the solution of \eqref{eq:halfwm-wave}  solves the original equation \eqref{eq:halfwm} to close the argument.

\subsection{Critical Besov $\Hy$ Case}
The non-compactness of $\Hy \subseteq \R^3$ presents difficulties when working with $\dot{H}^{\frac{n}{2}} \times \dot{H}^{\frac{n}{2}-1}$ initial data as we lose the $L^\infty$ bound. In the context of the wave equation with target $\Hy$, the global well-posedness problem has been solved by Krieger \cite{krieger2003global},\cite{krieger2004global}. Dealing with the non-local terms of the half-wave equation is more delicate. Instead, we extend the global well-posedness result of \cite{krieger2017small} and \cite{2019arXiv190412709K} for the $S^2$ target to the $\Hy$ target with the help of the intrinsic distance function of the hyperbolic plane. \par

The intrinsic hyperbolic distance on $\mathbb{H}^2$ is defined by
\begin{equation}
    \label{h2dist}
    d_H(p,q)= \text{arccosh}(-p \cdot_{\eta} q) = \ln (-p \cdot_{\eta} q+\sqrt{(p \cdot_{\eta} q)^2-1}),
\end{equation}
 for $p,q \in \mathbb{H}^2$. Note here that $p \cdot_\eta q \le -1$, $\forall p,q \in \Hy$. We also have the induced Riemannian metric $ds^2=-dx_0^2 +dx_1^2+dx_3^2$ on $\mathbb{H}^2$. \par
 For a given fixed point $Q\in \Hy$, we define the $L^p$ spaces taking values in $\Hy$ by
 $$L_Q^p(\R^{n+1};\Hy):=\left\{ f:\R^{1+n} \rightarrow \mathbb{H}^2 \big|\ d_H(f(x),Q)\in L^p \right\},$$
 as in the work of \cite{KorevaarNicholasJ1993Ssah} and \cite{JostJurgen1994Embm}. 

 We make similar definitions for the Sobolev spaces, Besov spaces, and others. 
Adopting the intrinsic distance function,  we have a similar theorem for the $\Hy$ targeted half-wave map:

 \begin{theorem}
    \label{thm:h2}
    Let $n\ge 4$, $u: \R^{1+n} \rightarrow \Hy \subseteq \R^3$ bounded such that $\lim_{|x|\rightarrow \infty} \ u(t,x) =Q$ for a fixed $Q\in \Hy$, for all t. Let $u[0]:=\big(u(0,\cdot),\partial_t u(0,\cdot) \big)=(u_0,u_1): \R^n \rightarrow \Hy \times T \Hy$ where
    $u_1 = u_0 \times_\eta \Lah u_0$ be a smooth pair of data satisfying \eqref{eq:halfwm} s.t. $u(0,\cdot)$ is constant outside a compact subset of $\R^n$ and
    \begin{equation}
        \label{con:init_hyper}
        \|u[0]\|_{\dot{B}^{\frac{n}{2}}_{2,1;Q} \times \dot{B}^{\frac{n}{2}-1}_{2,1;Q}} < \epsilon,
    \end{equation}
    where $\epsilon  \ll 1$ is sufficiently small. \
    
    Then \eqref{eq:halfwm} admits a global smooth solution. 
\end{theorem}

It is worth noting that if the target is a compact subset of $\Hy$, we can establish the global-well posedness for small $\dot{H}^{\frac{n}{2}} \times \dot{H}^{\frac{n}{2}-1}$ initial data as in the $S^2$ case . \par

\subsection{Physical Background} 
The half-wave map equation \eqref{eq:halfwm} admits a conserved energy
\[
E(t):=\int_{\R^n} |\Laq u|^2 dx. 
\] \par

This gives the a priori condition that $u(t)\in \dot{H}^{\frac{1}{2}}(\R^n)$ which implies that the \textit{half-wave map} is energy-critical in dimension $n=1$. In the work of Lenzmann and Schikorra \cite{LenzmannEnno2018Oehm}, the authors give a full classification of the travelling solitary waves for the energy-critical problem with target $S^2$ for $n=1$. However, there has not been much progress for more general global well-posedness of \eqref{eq:halfwm} in low dimensions $n\leq 3$. This is partly due to the lack of the Strichartz estimates used in the higher dimensional case for the null structure. We also point out that the conserved energy is closely related to the study of fractional harmonic maps by Da Lio and Rivi\`ere (see \cite{DaLioFrancesca2016HaM}, \cite{francesca2011three},\cite{da2011sub}).  \par

The one-dimensional energy-critical half-wave maps are of notable physics interest. It's intensively studied in the works of \cite{berntson2020multi},  \cite{zhou2015solitons},  \cite{gerard2018lax}, \cite{lenzmann2018short}, and \cite{enno2018energy}. The one-dimensional half-wave maps arise as a continuum limit of the discrete 
{\em Calogero-Moser (CM) spin system.} For the CM systems, we refer the reader to \cite{gibbons1984generalisation} and \cite{blom1999finding} in which the authors study the theory of completely integrable systems. In addition, the classical CM spin systems can be obtained by taking a suitable semiclassical limit of the quantum spin chains related to the well-known {\em Haldane-Shastry (HS) spin chains}, see e.g.\cite{haldane1988exact} \cite{shastry1988exact}, which are exactly solvable quantum models. \par

\noindent {\bf Acknowledgements:}
The author greatly thanks his supervisor Joachim Krieger for his guidance and 
fruitful discussion throughout the project. He also thanks Tobias Schmid for his generous help with hyperbolic geometry.

\section{Transform to Wave Equation Form}
\label{sec:transform}
We first reformulate the half-wave equation
\begin{equation*}
    u_t = u\times_{\eta} \Lah u
\end{equation*}
as a nonlinear wave equation involving additional non-local terms. We focus on the target $\Hy$ variant here. The $S^2$ case is in the work \cite{krieger2017small}. \par
By using the cross product rule of $\Hy$:
$$a \times_{\eta} b :=  \eta (a \times b),$$
$$
a \cdot_\eta b := (\eta a) \cdot b, 
$$
and 
$$(a\times_{\eta} b) \times_{\eta} c =( c\cdot_{\eta} b)\ a - (a \cdot_\eta b) \ b.$$ 
We compute
\begin{align*}
    \Box u = (\partial^2_t - \Delta) u =&\left(\Lah u \cdot_{\eta}\Lah u \right) \ u - \left(u \cdot_{\eta} \Lah u \right)\ \Lah u \nonumber \\
    &+ u \times_{\eta} \Lah \left(u \times_{\eta} \Lah u \right) - u \times_{\eta} \left(u \times_{\eta} (-\Delta) u \right)\\
    & -\left(\nabla u \cdot_{\eta} \nabla u\right) u.
\end{align*}

Additionally, we have the relation 
\begin{align*}
    &\left((\Lah u \cdot_{\eta}\Lah u ) \ u - \left(u \cdot_{\eta} \Lah u \right)\ \Lah u \right) \cdot_{\eta} u\\
    &\quad= - \left(\Lah u \cdot_{\eta} \Lah u \right) - \left(u \cdot_{\eta} \Lah u \right)^2 \\
    &\quad=\text{\footnotemark} - \|\Lah u \times_{\eta} u\|_{\eta}^2 \\
    &\quad= -(\partial_t u \cdot_{\eta} \partial_t u).
\end{align*}
\footnotetext{We write $\| x \|_{\eta}^2 := x \cdot_{\eta}x$.} 

Hence, we can deduce
\begin{equation}
    \label{eq:half-h2}
    \begin{split}
    \Box u = (\partial^2_t - \Delta) u =&\left(\Lah u \cdot_{\eta}\Lah u \right) \ u - \left(u \cdot_{\eta} \Lah u \right)\ \Lah u \nonumber \\
    &+ u \times_{\eta} \Lah \left(u \times_{\eta} \Lah u \right) - u \times_{\eta} \left(u \times_{\eta} (-\Delta) u \right) \nonumber \\
    & -\left(\nabla u \cdot_{\eta} \nabla u\right) u  \nonumber \\
    =&(\partial_t u \cdot_{\eta} \partial_t u - \nabla u \cdot_{\eta} \nabla u) u \nonumber \\
    &+\Pi_{u_{\perp}}\left[\ (u \cdot_{\eta} \Lah u )\ \Lah u \right]  \nonumber \\
    &+\Pi_{u_{\perp}} \left[u \times_{\eta} \Lah \left(u \times_{\eta} \Lah u \right) - u \times_{\eta} \left(u \times_{\eta} (-\Delta) u \right)\right] 
    \end{split}
\end{equation}
where $\Pi_{u_{\perp}}$ is the projection onto the orthonormal complement of $u$. Note here that $\nabla$ only represents the spatial derivatives. So we reformulate the half-wave map as a wave-type equation with two additional non-local terms. In the following sections, readers will see that these two additional non-local terms formally behave like $u \nabla u \cdot \nabla u$.
\section{Technical Preliminaries}
In this section, we introduce the function spaces and technical tools we are going to use.
Let $P_k$, $k\in \mathbb{Z}$, be standard Littlewood-Paley multipliers on $\R^n$, and $Q_j$, $j\in \Z$ be the multipliers which localize a space-time function $f(t, x)$ to dyadic distance $\sim 2^j$ from the light cone $\big|\tau\big| = \big|\xi\big|$ on the frequency domain. \par
For a smooth cutoff function $\chi\in C^\infty_0(\R_+)$ s.t.
\[
\sum_{k\in Z}\chi(\frac{x}{2^k}) = 1\,\forall x\in \R_+, 
\]
we set 
\[ \widehat{P_k f}(\xi) = \chi \Big(\frac{|\xi|}{2^k}\Big)\hat{f}(\xi),
\]
\[
\widehat{Q_j f}(\tau, \xi) = \chi\big(\frac{\big||\tau| - |\xi|\big|}{2^j}\big)\hat{f}(\tau, \xi). 
\]
where $\hat{f}(\tau, \xi):=\int_{\R^{n+1}} e^{-i(t\tau+x\cdot\xi)} f(t,x) \ dtdx$ is the space-time Fourier transform of $f$.

With these definitions, we define the homogenous hyperbolic sobolev space $X^{s,\theta}$, $s\geq \frac{n}{2}$ (see \cite{KM}, \cite{tao2006nonlinear}, \cite{d2005wave}) as 
\[
   \dot{X}^{s,\theta}(\R^{n+1}):=\big\{ u\in \mathcal{S}'(\R^{n+1}); \ \|u\|_{\dot{X}^{s,\theta}(\R^{n+1})}:= \sum_{j\in \Z} 2^{j\theta} \| \nabla_x^s Q_j u\|_{L_{t,x}^2(\R^{n+1})} < \infty \big\} ,
\]
where we choose $\theta$ such that 
$$s-\frac{n}{2}=\theta-\frac{1}{2}.$$
Furthermore, for the critical space, when $s=\frac{n}{2}$, we introduce the additional Besov type structure to the hyperbolic Sobolev space. We set
\begin{equation*}
    \begin{split}
        &\big\|u\big\|_{\dot{X}^{\frac{n}{2},\frac12,\infty}(\R^{n+1})}: = \sup_{j\in Z}2^{\frac{j}{2}}\big\|\nabla_x^{\frac{n}{2}}Q_j u\big\|_{L_{t,x}^2(\R^{n+1})}, \\
        &\big\|F\big\|_{\dot{X}^{\frac{n}{2}-1, -\frac12, 1}(\R^{n+1})}: = \sum_{j\in Z}2^{-\frac{j}{2}}\big\|\nabla_x^{\frac{n}{2}-1}Q_j F\big\|_{L_{t,x}^2(\R^{n+1})}.
    \end{split}
\end{equation*}

We also have the following bound for the localization in time operation on $X^{s,\theta}$ spaces(see \cite{d2005wave},\cite{tao2006nonlinear}, \cite{Burzio278468}).  \par 
\begin{Lemma}
   For $s\in \R$, $\theta \geq \frac{1}{2}$. Let $\varphi: \R \rightarrow \R_{+} \in C_0^\infty(\R)$ such that 
\begin{equation*}
    \varphi(t)=\begin{cases}
          1 \ \text{for} \ |t|\leq 1,\\
          0 \ \text{for} \ |t| \geq 2.
    \end{cases}
\end{equation*}
For a given $T>0$, we define a scaling function 
$$
\varphi_T:=\varphi \big(\frac{t}{T} \big),
$$ 
then we have 
\begin{equation}
    \label{eq:loc_xsb_es}
    \| \varphi_T u\|_{\dot{X}^{s,\theta}(\R^{n+1})} \lesssim \max\{1, T^{\frac{1}{2}-\theta}\} \ \| u\|_{\dot{X}^{s,\theta}(\R^{n+1})}.
\end{equation}
\end{Lemma}

Given the wave admissible condition $$\frac{1}{p} + \frac{n-1}{2q}\leq \frac{n-1}{4}, \ 2\leq p,q\leq \infty,\ \text{for } n\geq 4,$$
we define our solution spaces with the classical Strichartz norms. For the global critical case, we set
\begin{equation}\label{eq:Snorm}
    \begin{split}
        \big\|u\big\|_{S}&: = \sum_{k\in \Z}\sup_{(p,q)\,\text{admissible}}2^{(\frac1p + \frac{n}{q}-1)k}\big\|\nabla_{t,x}P_ku\big\|_{L_t^p L_x^q(\R^{n+1})} \\
        &\quad + \big\|\nabla_{t,x}P_k u\big\|_{\dot{X}^{\frac{n}{2}-1,\frac12,\infty}(\R^{n+1})}\\
        &=:\sum_{k\in \Z}\big\|P_k u\big\|_{S_k}.
    \end{split}
\end{equation} \par
For the local sub-critical case, when $s>\frac{n}{2}$, we set
\begin{equation}\label{eq:Ssnorm}
    \begin{split}
        \big\|u\big\|_{S_{\text{loc}}}&:= \sum_{k\in \Z}\sup_{(p,q)\,\text{admissible}}2^{(s-\frac{n}{2}+\frac1p + \frac{n}{q}-1)k}\big\|\nabla_{t,x}P_ku\big\|_{L_t^p L_x^q([0,T] \times \R^n)} \\
        &\quad + \big\|\nabla_{t,x}\varphi_T P_k u\big\|_{\dot{X}^{s-1,\theta}(\R^{n+1})}\\ 
    &=:\sum_{k\in \Z}\big\|P_k u\big\|_{S_{\text{loc}}}.
    \end{split}
    \end{equation}

Further, we define the spaces to control the nonlinearity respectably as
\begin{equation}\label{eq:Nnorm}
\big\|F\big\|_{N}: = \sum_{k\in \Z}\big\|P_kF\big\|_{L_t^1\dot{H}_x^{\frac{n}{2}-1} + \dot{X}^{\frac{n}{2}-1,-\frac12,1}(\R^{n+1})}
\end{equation}
\begin{equation}\label{eq:Nsnorm}
    \big\|F\big\|_{N_{\text{loc}}}: = \sum_{k\in \Z}\big\|\varphi_T P_kF\big\|_{L_t^1\dot{H}_x^{s-1} + \dot{X}^{s-1,\theta-1}(\R^{n+1}) }
    \end{equation}
From the works of wave maps \cite{Kri}, \cite{ 2000mathTao}, \cite{Tat3}, we have the Strichartz estimate,
\begin{Proposition}
    For $s \geq \frac{n}{2}$, we have
    \label{prop:linestimates}
    \begin{equation}\label{eq:energyinequality}
        \big\| u\big\|_S\lesssim \big\|u[0]\big\|_{\dot{H}^{s}\times \dot{H}^{s-1}} + \big\|\Box u\big\|_{N}. 
        \end{equation}
\end{Proposition}
Additionally, we introduce the following useful multilinear lemma to estimate the difference between non-local terms.

\begin{Lemma}[Lemma 3.1 in \cite{krieger2017small}]
    \label{lem:multilinestimates} For the following bilinear expression (where $\chi_{k_j}(\cdot)$ is the smooth cutoff to the annulus $|\xi|\sim 2^{k_j}$)
    \[
    F(u, v)(x): = \int_{\R^{n}}\int_{\R^{n}}m(\xi, \eta)e^{ix\cdot(\xi+\eta)}\chi_{k_1}(\xi)\hat{u}(\xi)\chi_{k_2}(\eta)\hat{v}(\eta)d\xi d\eta
    \]
    where the multiplier $m(\xi, \eta)$ is $C^\infty$ with respect to the coordinates on the support of $\chi_{k_1}(\xi)\cdot \chi_{k_2}(\eta)$, and satisfies point-wise bounds 
    \[
    \big|m(\xi, \eta)\big|\leq \gamma\lesssim 1,\, \big|(2^{k_1}\nabla_{\xi})^i(2^{k_2}\nabla)^jm(\xi, \eta)\big|\lesssim_{i,j} 1,\,\forall i,,j.  
    \]
    Then if $\big\|\cdot\big\|_{Z}, \big\|\cdot\big\|_{Y}, \big\|\cdot\big\|_{X}$ are translation invariant norms with the property that 
    \[
    \big\|u\cdot v\big\|_{Z}\leq \big\|u\big\|_{X}\cdot\big\|v\big\|_{Y}, 
    \]
    it follows that 
    \[
    \big\|F(u, v)\big\|_{Z}\lesssim\gamma^{(1-)}\big\|P_{k_1}u\big\|_{X}\big\|P_{k_2}v\big\|_{Y},
    \]
    where the implied constant only depends on the size of finitely many derivatives of $m$. 
    \end{Lemma}

    In particular, for $k_1,k_2 \lesssim 1$, we have
    \begin{equation}\label{singlefreq}
    \big\|P_0 \left( \Lah (u_{k_1}\cdot u_{k_2} ) - u_{k_1}\cdot  (-\triangle^{\frac12}u_{k_2})\right) \big\|_{Z}\lesssim2^{k_1} \big\| u_{k_1}\big\|_{X}\big\|u_{k_2}\big\|_{Y},
    \end{equation}
    and 
    \begin{equation}\label{singlefreq2}
    \big\|P_0\left(\Lah (u_{k_1}\cdot \Lah u)-u_{k_1}\cdot (\triangle u_{k_2}) \right)\big\|_{Z}\lesssim 2^{k_1+k_2} \big\| u_{k_1}\big\|_{X}\big\|u_{k_2}\big\|_{Y},
    \end{equation}
    by using the multipliers 
    $$m(\xi,\eta)=\chi_0(\xi +\eta)(|\xi +\eta|-|\eta|) \quad\text{and} \quad m(\xi,\eta)=\chi_0(\xi +\eta)|\eta|(|\xi +\eta|-|\eta|) \ \text{respectably}.$$ \par

    We will also take advantage of the geometric property of $S^2$. Since $u\cdot u = 1$, by the Littlewood-Paley decomposition, we deduce 
    \begin{equation}\label{eq:orthomicro}
    0 = u\cdot u - Q\cdot Q = \sum_{|k_1-k_2|\leq 10} u_{k_1}u_{k_2} + 2\sum_{k_1}u_{k_1}\cdot u_{<k_1-10}
    \end{equation}
    This implies that 
    \begin{equation}\label{eq:cancel}
    \sum_{|k_1-k_2|<10} (-\triangle)^{\frac{1}{2}}\big(u_{k_1}u_{k_2}\big) =- 2\sum_{k_1}(-\triangle)^{\frac{1}{2}}\big(u_{k_1}\cdot u_{<k_1-10}\big)
    \end{equation}

\section{Local Well-Posedness With $S^2$ target}
\label{sec:local}
We show that the sub-critical Cauchy problem of \eqref{eq:halfwm-wave} is locally well-posed for $H^s \times H^{s-1}(\R^n)$, $s>\frac{n}{2}$, initial data on the domain $[0,T] \times \R^n$ for $n \geq 5$. The local and global well-posedness problems of the sub-critical wave maps have been studied by Klainerman-Machedon \cite{klainerman1996estimates} for $n \geq 4$, Klainerman-Selberg \cite{klainerman1997remark} for $n =2,3$, and Keel-Tao \cite{KeelMarcus1998Lagw} for $n=1$. \par

Accounting on the higher regularity of initial data, we first prove a multilinear estimate proposition for the non-local terms of the \eqref{eq:halfwm-wave}. Then we use an iterative argument and the fixed point theorem to prove the local well-posedness of the sub-critical problem. In particular, we focus on the case when $s$ is close to $\frac{n}{2}$ i.e. $ 0<s-\frac{n}{2} \ll 1 $. For larger $s$, the result follows by the persistence of higher regularity argument.
\par
We first build the multilinear estimates:
\begin{Proposition}
    \label{prop:locmultilinear_s2}
    Let $n \geq 5$, $u: [0,T]\times \R^{n} \rightarrow S^2$, such that $\lim_{|x|\rightarrow \infty} \ u(t,x) =Q$ for a fixed $Q\in S^2$, for all t. For some $\beta, \sigma>0$, we have the following estimates.
    \begin{equation}
        \label{eq:locmultilin1_s2}
        \begin{split}
            &\big\|P_k \big[u( \partial_t u \cdot \partial_t u -\nabla u \cdot \nabla u )\big]\big\|_{N_{loc}} \\
    &\lesssim T^{\beta} (1+\big\|u\big\|_{S_{loc}})\big\|u\big\|_{S_{loc}}\big(\sum_{k_1\in Z}2^{-\sigma|k-k_1|}\big\|P_{k_1}u\big\|_{S_{loc}}\big).
        \end{split}
    \end{equation}

    Furthermore, for $\tilde{u}$ maps into a small neighborhood of $S^2$, we also have 
    \begin{equation}\label{eq:locmultilin2_s2}
        \begin{split}
            &\big\|P_k\big(\Pi_{\tilde{u}}\big((-\triangle)^{\frac{1}{2}}u\big)(u\cdot (-\triangle)^{\frac{1}{2}}u)\big)\big\|_{N_{loc}}\\
    &\lesssim T^{\beta}\prod_{v = u,\tilde{u}} (1+\big\|v\big\|_{S_{loc}})\big\|u\big\|_{S_{loc}}\big(\sum_{k_1\in Z}2^{-\sigma|k-k_1|}\big\|P_{k_1}u\big\|_{S_{loc}}\big)
        \end{split} 
    \end{equation}

    and

    \begin{equation}\label{eq:locmultilin3_s2}\begin{split}
    &\big\|P_k\big(\Pi_{\tilde{u}}\big[u \times (-\triangle)^{\frac{1}{2}}(u \times (-\triangle)^{\frac{1}{2}}u) - u \times(u \times (-\triangle)u)\big]\big)\big\|_{N_{loc}}\\
    &\lesssim T^{\beta}\prod_{v = u,\tilde{u}} (1+\big\|v\big\|_{S_{loc}})\big\|u\big\|_{S_{loc}}\big(\sum_{k_1\in Z}2^{-\sigma|k-k_1|}\big\|P_{k_1}u\big\|_{S_{loc}}\big).
    \end{split}\end{equation}

    Using the notation 
    \[
    \triangle_{1,2}F^{(j)} := F^{(1)} - F^{(2)},
    \]

    we have the difference estimates: 
    \begin{equation}\label{eq:locmultilin4_s2}\begin{split}
    &\big\|\triangle_{1,2}P_k\big[u^{(j)}(\nabla u^{(j)}\cdot\nabla u^{(j)} - \partial_t u^{(j)}\cdot\partial_t u^{(j)})\big]\big\|_{N_{loc}}\\&\lesssim T^{\beta}(1+\max_j\big\|u^{(j)}\big\|_{S_{loc}})(\max_j\big\|u^{(j)}\big\|_{S_{loc}})\big(\sum_{k_1\in Z}2^{-\sigma|k-k_1|}\big\|P_{k_1}u^{(1)} - P_ku^{(2)}\big\|_{S_{loc}}\big)\\
    & +T^{\beta} (1+\max_j\big\|u^{(j)}\big\|_{S_{loc}})(\big\|u^{(1)} - u^{(2)}\big\|_{S_{loc}})\big(\max_j\sum_{k_1\in Z}2^{-\sigma|k-k_1|}\big\|P_{k_1}u^{(j)}\big\|_{S_{loc}}\big),\\
    \end{split}\end{equation}

    and the two similarly estimates

    \begin{equation}\label{eq:locmultilin5_s2}\begin{split}
    &\big\|P_k\triangle_{1,2}\big(\Pi_{\tilde{u}^{(j)}_{}}\big((-\triangle)^{\frac{1}{2}}u^{(j)}\big)(u^{(j)}\cdot (-\triangle)^{\frac{1}{2}}u^{(j)})\big)\big\|_{N_{loc}}\\
    &\lesssim T^{\beta}\max_j\prod_{v = u^{(j)},\tilde{u}^{(j)}} (1+\big\|v\big\|_{S_{loc}})\big\|u^{(j)}\big\|_{S_{loc}}\big(\sum_{k_1\in Z}2^{-\sigma|k-k_1|}\big\|P_{k_1}u^{(1)} - P_{k_2}u^{(2)} \big\|_{S_{loc}}\big)\\
    & + T^{\beta}\max_j\prod_{v = u^{(j)},\tilde{u}^{(j)}} (1+\big\|v\big\|_{S_{loc}})\big\|u^{(1)}- u^{(2)}\big\|_{S_{loc}}\big(\max_j\sum_{k_1\in Z}2^{-\sigma|k-k_1|}\big\|P_{k_1}u^{(j)}\big\|_{S_{loc}}\big)\\
    & +  T^{\beta} \max_j(1+\big\|u^{(j)}\big\|_{S_{loc}})\big\|\tilde{u}^{(1)}- \tilde{u}^{(2)}\big\|_{S_{loc}}\big(\max_j\sum_{k_1\in Z}2^{-\sigma|k-k_1|}\big\|P_{k_1}u^{(j)}\big\|_{S_{loc}}\big),\\
    \end{split}
\end{equation}

 \begin{equation}
        \label{eq:locmultilin6}
        \begin{split}
        &\big\|P_k \triangle_{1,2} \big(\Pi_{\tilde{u}^{(j)}}\big[u^{(j)} \times (-\triangle)^{\frac{1}{2}}(u^{(j)} \times (-\triangle)^{\frac{1}{2}}u^{(j)}) - u ^{(j)}\times(u^{(j)} \times (-\triangle)u^{(j)})\big]\big)\big\|_{N_{loc}}\\
        &\lesssim T^{\beta}\max_j\prod_{v = u^{(j)},\tilde{u}^{(j)}} (1+\big\|v\big\|_{S_{loc}})\big\|u^{(j)}\big\|_{S_{loc}}\big(\sum_{k_1\in Z}2^{-\sigma|k-k_1|}\big\|P_{k_1}u^{(1)} - P_{k_2}u^{(2)} \big\|_{S_{loc}}\big)\\
    & + T^{\beta}\max_j\prod_{v = u^{(j)},\tilde{u}^{(j)}} (1+\big\|v\big\|_{S_{loc}})\big\|u^{(1)}- u^{(2)}\big\|_{S_{loc}}\big(\max_j\sum_{k_1\in Z}2^{-\sigma|k-k_1|}\big\|P_{k_1}u^{(j)}\big\|_{S_{loc}}\big)\\
    & +  T^{\beta}\max_j(1+\big\|u^{(j)}\big\|_{S_{loc}})\big\|\tilde{u}^{(1)}- \tilde{u}^{(2)}\big\|_{S_{loc}}\big(\max_j\sum_{k_1\in Z}2^{-\sigma|k-k_1|}\big\|P_{k_1}u^{(j)}\big\|_{S_{loc}}\big).\\
        \end{split}
    \end{equation}

    \end{Proposition}
The basic technics of the proof are analogous to the work of Krieger-Sire \cite{krieger2017small}. The key difference between our proposition and the previous work is that we gain a small $T$ factor for the estimates. This is due to the higher regularity of the initial data and the local in time settings. The benefit is that we can control the estimates by the $T$ factor instead of assuming the smallness of the initial data. The proof also shares similar technics to prove the global well-posedness result in Section \ref{sec:error}, hence we omit the detailed proof here. We refer the reader to Appendix \ref{app:proof} for detailed proof.  \par
With the multilinear estimates, we then have the local well-posedness result.
\begin{theorem}
\label{thm:local}    
 Assume the smooth initial data $u[0] \in \dot{H}^s \times \dot{H}^{s-1}(\R^n)$, for $s>\frac{n}{2}$. There exists $0<T<\infty$, depends on the $\dot{H}^s \times \dot{H}^{s-1}(\R^n)$ of the initial data $u[0]$. The Cauchy problem of \eqref{eq:halfwm-wave} admits a smooth solution on $[0,T]\times \R^n$ s.t.
$$
\|(u(t), \partial_t u(t))\|_{\dot{H}^s \times \dot{H}^{s-1}(\R^n)} \leq C \| u[0]\|_{\dot{H}^s \times \dot{H}^{s-1}(\R^n)},
$$
for all $0\leq t \leq T$.
\end{theorem}

We use the iterative method to build the solution for the half-wave equation \eqref{eq:halfwm-wave} introduced in Krieger-Sire \cite{krieger2017small}.

\begin{proof}
         
        
         First, we set $Q:=\lim_{|x|\rightarrow\infty}u_0(x)$ for $Q \in S^2$ as defined in Theorem \ref{thm:s2}. Let $u^{(1)}$ be the solution of the free wave map
         \begin{equation}
             \begin{cases}
                 &\Box u^{(1)} =0 \\
                 &u^{(1)}[0]:=\big(u(0,\cdot),\partial_t u(0,\cdot) \big) \in \dot{H}^s \times \dot{H}^{s-1}(\R^n).
             \end{cases}
         \end{equation}

        From the local well-posedness theorem of wave maps (see \cite{klainerman1996estimates}, \cite{klainerman1997remark} and \cite{Tat3}), we have $u^{(1)} \in C([0,T],\dot{H}^s \cap C^1([0,T], \dot{H}^{s-1}(\R^n))$. 
        
        
       For $j\geq 2$, we define $u^{(j)}$ inductively by
         \begin{equation}
            \label{eq:iterate}
            \begin{split}
        &(\partial_t^2 - \triangle) u^{(j)}\\&= u^{(j)}(\nabla u^{(j)}\cdot\nabla u^{(j)} - \partial_t u^{(j)}\cdot\partial_t u^{(j)})\\
        &+\Pi_{u_{\perp}^{(j)}}\big((-\triangle)^{\frac{1}{2}}u^{(j-1)}\big)(u^{(j-1)}\cdot (-\triangle)^{\frac{1}{2}}u^{(j-1)})\\
        &+\Pi_{u_{\perp}^{(j)}}\big[u^{(j-1)} \times (-\triangle)^{\frac{1}{2}}(u^{(j-1)} \times (-\triangle)^{\frac{1}{2}}u^{(j-1)}) - u^{(j-1)} \times(u^{(j-1)} \times (-\triangle)u^{(j-1)})\big]\\
        \end{split}
        \end{equation} 
        
        Since $u^{(j)}$ is also defined with the non-linear terms on the right-hand side of the equation, we need 
        to use a sub-iteration to find it explicitly. For each $u^{(j)}$, we set a sub-iteration scheme for $u^{(j,i)}$, $i\geq 0$ s.t. $u^{(j,0)}$ is the solution of the free wave map with initial data $u[0]$ as above. Then for $i \geq 1$, we let $u^{(j,i)}$ solves

         \begin{equation}
            \label{eq:iterate1}
            \begin{split}
        &(\partial_t^2 - \triangle) u^{(j,i)}\\
        &= u^{(j,i-1)}(\nabla u^{(j,i-1)}\cdot\nabla u^{(j,i-1)} - \partial_t u^{(j,i-1)}\cdot\partial_t u^{(j,i-1)})\\
        &+\Pi_{u_{\perp}^{(j,i-1)}}\big((-\triangle)^{\frac{1}{2}}u^{(j-1)}\big)(u^{(j-1)}\cdot (-\triangle)^{\frac{1}{2}}u^{(j-1)})\\
        &+\Pi_{u_{\perp}^{(j,i-1)}}\big[u^{(j-1)} \times (-\triangle)^{\frac{1}{2}}(u^{(j-1)} \times (-\triangle)^{\frac{1}{2}}u^{(j-1)}) - u^{(j-1)} \times(u^{(j-1)} \times (-\triangle)u^{(j-1)})\big].\\
        \end{split}
        \end{equation} 

        Within the sub-iteration of $u^{(j,i)}$, by the multi-linear estimate in Proposition \ref{prop:locmultilinear_s2}, we have 
        $$u^{(j,i-1)} \in S_{loc} \Rightarrow u^{(j,i)} \in S_{loc}.$$
        
        Using the fixed point argument, we conclude that the equation \eqref{eq:iterate1} has a unique solution in $S_{loc}$.
         Next, we show the sequence $ \{u^{(j,i)} \}_{i\geq 0}$ is actually converging to the solution $u^{(j)}$. For $T$ sufficiently small enough, we apply Proposition \ref{prop:linestimates} inductively to \eqref{eq:iterate1} to conclude that $ \{u^{(j,i)} \}_{i\geq 0}$ is a bounded sequence for all $i \in \mathbb{N}$. By the difference estimates \eqref{eq:locmultilin3_s2}, \eqref{eq:locmultilin4_s2}, \eqref{eq:locmultilin5_s2} and \eqref{eq:locmultilin6} in Proposition \ref{prop:locmultilinear_s2}, we know that $u^{(j,i)}$ is a Cauchy sequence. Thus we have 

        $$u^{(j,i)} \rightarrow u^{(j)} \ \text{in } S_{loc},$$
         
        We use this argument to the original iteration scheme \eqref{eq:iterate} to conclude
        
        $$
        u^{(j)} \rightarrow u \ \text{in } S_{loc}.
        $$

       Hence, we find the local solution of \eqref{eq:halfwm-wave}.
        Lastly, we need to check all $u^{j}$ maps into $S^2$. From the above, we only know that $u^{(j)}$ is closed to $S^2$ w.r.t to $L^\infty$ norm. The result is clear for $u^{(j)}, j = 0,1$. With the iteration scheme, we have 
        \begin{align*}
            &\Box(u^{(j)}\cdot u^{(j)} -1 ) = 2(u^{(j)}\cdot u^{(j)} - 1)(\nabla u^{(j)}\cdot\nabla u^{(j)} - \partial_t u^{(j)} \cdot\partial_t u^{(j)}), \\
            &(u^{(j)}\cdot u^{(j)} - 1)[0] = 0,
        \end{align*}
        for all $j\geq 2$. This inductively implies that $u^{(j)}$ maps into $S^2$ for all $j$. \par
        
        If our initial data $u[0]$ is smooth, we can also conclude the solution is smooth, by differentiating the equation \eqref{eq:iterate} and \eqref{eq:iterate1} and bootstrapping. 
        
\end{proof}
\section{Global Well-Posedness With $S^2$ target}
\label{sec:global}
For the Cauchy problem of critical half-wave equation with small $\dot{H}^{\frac{n}{2}} \times \dot{H}^{\frac{n}{2}-1}(\R^n)$ initial data, we first study the wave type equation \eqref{eq:halfwm-wave} by adopting the strategies in the work of Tao \cite{2000mathTao}. Then we show that the solution for the wave type equation \eqref{eq:halfwm-wave} is actually solve the original equation \eqref{eq:halfwm} to conclude the Theorem \ref{thm:s2}. \par
For the wave type half-wave map \eqref{eq:halfwm-wave}, we have 
\begin{Proposition}
    \label{prop:waveform}
    Let $n \geq 5$, $u: \R^{1+n} \rightarrow S^2$, such that $\lim_{|x|\rightarrow \infty} \ u(t,x) =Q$ for a fixed $Q\in S^2$, for all t. Let $u[0]=(u(0,\cdot), u_t(0,\cdot))=(u_0,u_1):\R^n \rightarrow S^2 \times TS^2$ be a smooth pair of data as in Theorem \ref{thm:s2} satisfies the wave type half-wave map \eqref{eq:halfwm-wave} such that $u[0] \in \dot{H}^s \times \dot{H}^{s-1}(\R^n)$ for $s>\frac{n}{2}$ and its $\dot{H}^{\frac{n}{2}} \times \dot{H}^{\frac{n}{2}-1}(\R^n)$ is sufficiently small s.t.
    \begin{equation}
        \label{con:wave_init_s2}
        \|u[0]\|_{\dot{H}^{\frac{n}{2}} \times \dot{H}^{\frac{n}{2}-1}} < \epsilon, \ \epsilon  \ll 1.  
    \end{equation}
     Then the Cauchy problem for \eqref{eq:halfwm-wave} with initial data $u[0]$ has a global smooth solution s.t. 
    \begin{equation}
        \label{wave_bound}
    \big\| u[t] \big\|_{L^\infty_t \big(\dot H^{\frac{n}{2}}_x \times \dot H^{\frac{n}{2}-1}_x \big)} \lesssim \big\| u[0] \big\|_{\dot H^{\frac{n}{2}}_x \times \dot H^{\frac{n}{2}-1}_x}.
    \end{equation}
\end{Proposition}

Firstly, we introduce the frequency fold localization concept from Tao \cite{2000mathTao}. We define the notion of frequency fold for small $\dot{H}^{\frac{n}{2}}$ norm with respect to its $\ell^2$ summation property. For a fixed $0 < \sigma < \frac{1}{4}$ only depends on dimension $n$ and the small constant $0 < \varepsilon \ll 1$, we define

\begin{Definition}\label{envelope-def}  A \emph{frequency envelope} is a sequence $c = \{c_k\}_{k \in \Z} \subseteq \R$ such that
\begin{equation}
\label{l2-ass}
\|c\|_{l^2} \lesssim \varepsilon,
\end{equation}
and
\begin{equation}
    \label{local}
2^{-\sigma |k-k'|}  \lesssim \frac{c_k}{c_{k'}} \lesssim 2^{\sigma |k-k'|}
\end{equation}
for all $k, k' \in \Z$.  Moreover, we say $c_k \sim c_k'$ if $k = k' + O(1)$. \par
\end{Definition}

For a frequency envelope $c$ and a pair of functions $(f,g)$ on $\R^n$, we say that $(f,g)$ lies underneath the envelope $c$ if
$$ \| P_k f \|_{\dot H^{\frac{n}{2}}} + \| P_k g \|_{\dot H^{\frac{n}{2}-1}} \leq c_k \quad \forall k \in \Z.$$

One can easily verify that for a pair of function $(f,g)$ lies underneath an envelope $c$ if and only if
\begin{equation}
    \label{init}
\| (f,g) \|_{\dot H^{n/2} \times \dot H^{n/2-1}} \lesssim \varepsilon.
\end{equation}

With the definition of frequency envelope, we now define \textit{error} terms, essentially, they are the functions with small $L_t^1L_x^2([0,T]\times \R^n)$ norm (without further specify,  the Strichartz norms are all defined on $[0,T]\times \R^n$ in the following). 
\begin{Definition}
    For a function $u$ defined on $[0,T] \times \R^{n}$, we call it error if 
    $$\|u\|_{L_t^1 L_x^2} \lesssim  C_0^3 c_0 \epsilon$$
\end{Definition}
 For the half-wave equation \eqref{eq:halfwm-wave}, we will show that the additional non-local terms of the equation are \textit{errors}. Specifically, we will have
\begin{equation}  
    \label{eq:hwm-linear}
    \Box u_0 = -2 u_{\leq -10} \nabla_{t,x} u_{\leq -10} \nabla_{t,x} u_0 + \textit{error}.
\end{equation}
This process is called linearization in the work of Tao \cite{2000mathTao}. \par
Instead of proving the global well-posedness Proposition \ref{prop:waveform} directly, we first prove a local well-posedness result of $\eqref{eq:halfwm-wave}$ on $[0,T]\times \R^n$, $ 0<T<\infty$ for critical initial data. Then we show that the local solution can be extended globally.

\begin{Proposition}
    \label{main-prop}  Let $0 < T < \infty$, $c$ be a frequency envelope $\|c\|_{\ell^2} \lesssim \epsilon$ for a sufficiently small $0<\varepsilon \ll 1$. Let $u$ be the map in Proposition \ref{prop:waveform} which solves the equation \eqref{eq:halfwm-wave} on $[0,T] \times \R^n$ for initial data $u[0]$ lies underneath the envelope $c$.  We have
    \begin{equation}
        \label{control}
    \| P_k u \|_{S_k([0,T] \times \R^n)} \leq C_0 c_k
    \end{equation}
    for all $k \in \Z$, and a constant $C_0(n) \gg 1$ depends only on $n$. Hence $u(t)$ lies underneath the frequency fold $c$ for all $t\in [0,T]$.
    \end{Proposition}
Since $C_0,\ c_k$ are constants independent of $T$, we can iteratively extend the local solution of Proposition~\ref{main-prop} to prove the Proposition \ref{prop:waveform}. \par

For initial data $u[0] \in H^s \times H^{s-1}$, $\frac{n}{2}<s<s+\sigma$, for a $0<\sigma <\frac{1}{4}$, we have a local smooth local solution of \eqref{eq:halfwm-wave} $u$ on $[0,T]\times \R^n$ by Theorem \ref{thm:local}. Further assume the smallness of the $\dot{H}^{\frac{n}{2}} \times \dot{H}^{\frac{n}{2}-1}$ of $u[0]$ as in Proposition \ref{main-prop}, we have
\begin{align*}
    \|P_k u\|_{L_t^\infty H^s\times H^{s-1}([0,T]\times \R^n)} &\lesssim  2^{(\frac{n}{2}-s)k} \| P_k u\|_{S([0,T]\times \R^n)} \\
    &\leq C_0 2^{(\frac{n}{2}-s)k} c_k.
\end{align*}
For the frequency fold $c=(c_k)_{k\in \Z}$ such that
$$
c_k:= \sum_{k' \in \Z} 2^{-\sigma|k-k'|} \| P_k u[0] \|_{\dot{H}^{\frac{n}{2}} \times \dot{H}^{\frac{n}{2}-1}}.
$$ 
We can further deduce that
\begin{align*}
    \|P_k u\|_{L_t^\infty H^s\times H^{s-1}([0,T]\times \R^n)}
    &\lesssim C_0 \sum_{k_1} 2^{(s-\frac{n}{2}-\sigma)|k-k_1|} \|P_k u[0]\|_{\dot{H}^{s}\times \dot{H}^{s-1}}.
\end{align*} 
Thus we conclude that 
\begin{equation*}
    \|u\|_{L_t^\infty H^s\times H^{s-1}([0,T]\times \R^n)} \lesssim \|u[0]\|_{L_t^\infty H^s\times H^{s-1}([0,T]\times \R^n)}.
\end{equation*}

So we extend the local solution $u$ globally to prove Proposition \ref{prop:waveform}. Our task reduces to show Proposition \ref{main-prop}. \par
\begin{proof}[Proof of Proposition \ref{main-prop} ]
 We use a continuity argument to show that the \eqref{control} hold on $[0,T]\times \R^n$. We define the solution time interval s.t.
\begin{equation}
    I:=\left\{ 0\leq t \leq T; \big\|P_k u\|_{S_k([0,T]\times \R^n)} \leq C_0 c_k \right\}.
\end{equation}

If we prove the interval $I=[0,T]$, then we prove the Proposition \ref{main-prop}. Firstly, $I$ is a non-empty set. We know the initial data lies underneath the frequency fold $c$, i.e.
 
$$\big\|P_k u[0]\|_{S_k} \leq C_0 c_k.$$ 

Thus $0\in I$, and $I \neq \emptyset$.

Next, we show that $I$ is both an open and a closed subset of $[0,T]$. This implies that $I$ is the whole interval $[0,T]$.  \par

Note that the map
$$ t \rightarrow \big\|P_k u \big\|_{S_k([0,t]\times \R^n)}$$

is continuous on $[0,T]$, so $I$ is a closed set.  \par

The argument for $I$ is open is more complicated. We first introduce a larger set:
\begin{equation}
    I':=\left\{ 0\leq t' \leq T; \big\|P_k u\|_{S_k([0,T]\times \R^n)} \leq 2C_0 c_k \right\}.
\end{equation}

The difference between $I$ and $I'$ is just the bound is enlarged by a factor of 2. Essentially, we want to show that $\big\|P_k u\|_{S_k([0,T]\times \R^n)}$ can not exceed $C_0c_k$ on $[0,T]\times \R^n$. \par
In order to show $I$ is an open set, we first argue that  $I \subseteq \mathring{I'}$ (the interior of $I'$). Then we show that $I' \subseteq I$. This guarantees $I$ is an open set.  \par

We have a priori bound
$$\|P_k u\|_{S_k([0,t]\times \R^n)} \lesssim 2^{-{|k|}},$$ 

uniformly for $0\leq t \leq T$ and $\forall k \in \mathbb{Z}$. \par
For a $t_1 \in I$ s.t. $0<t_1< T$, 
we can find $t_1< t_2 \leq T$ s.t.
\begin{equation}
    \|P_k u\|_{S_k([0,t_2]\times \R^n)} \leq 2C_0c_k, \ \forall k \in \mathbb{Z},
\end{equation}

by the continuity of map $ t \rightarrow \big\|P_k u \big\|_{S_k([0,t]\times \R^n)}$. So we conclude $I \subseteq \mathring{I'}$. \par
We only left to show $I'\subseteq I$. This argument translates into the following proposition.

\begin{Proposition}
    \label{reduced}  Let $0 < T < \infty$, $c$ be a frequency envelope s.t. $\|c\|_{\ell^2} \lesssim \epsilon$ for a sufficiently small $0<\epsilon \ll 1$. Let $u$ be the map defined in Theorem~\ref{thm:s2} which solves the wave type half-wave map \eqref{eq:halfwm-wave} on $[0,T] \times \R^n$ for the initial data $u[0]$ lies underneath the envelope $c$. s.t.
    \begin{equation}
        \|P_k u\|_{S([0,T]\times \R^n)}\leq 2C_0c_k \quad \forall k\in \Z.
    \end{equation}
    Then, we have 
    \begin{equation}
        \label{eq:prop4.3}
        \|P_k u\|_{S([0,T]\times \R^n)}\leq C_0c_k \quad \forall k \in \Z.
    \end{equation} 
\end{Proposition}

After we prove the Proposition \ref{reduced}, we close the argument that $I=[0,T]$. Hence we proved Proposition \ref{main-prop}.

\end{proof}
\subsection{Proof of Proposition \ref{reduced}}
\label{sec:error}
Instead of proving the \eqref{eq:prop4.3} holds for all $k \in \mathbb{Z}$, it's enough to show that the result holds for $k=0$ because of the scaling invariance of the $S$ norm. To see this, we consider a scaling solution 
\begin{equation*}
    \bar{u}(t,x):= u(\frac{t}{2^k},\frac{x}{2^k}), \quad \forall(t,x) \in [0,2^k T]\times \R^n
\end{equation*}

and the corresponding frequency fold 
$$
\bar{c}=(\bar{c}_{k'})_{k' \in \mathbb{Z}}:=(c_{k+k'})_{k'\in \mathbb{Z}}.
$$

By the definition of $\ell^2$ norm, we know that $\|\bar{c}\|_{\ell^2}=\|c\|_{\ell^2}$. The scaled initial data $\bar{u}[0]$ is also in the frequency fold $\bar{c}$. Then we verify that
\begin{align*}
    \|P_{k'} \bar{u}\|_{S([0,2^k T]\times \R^n)} =\|P_{k'+k} u\|_{S([0,T]\times \R^n)} \leq 2C_0c_{k+k'}= 2C_0\bar{c}_{k'}, \ \forall k' \in \mathbb{Z}.
\end{align*}

 So we have $\bar{u}$ satisfies the assumption of Proposition \ref{reduced} in $[0,2^k T]$. We also have 
 \begin{equation*}
     \|P_k u\|_{S([0,T]\times \R^n)} =\|P_0 \bar{u}\|_{S([0,2^k T] \times \R^n)} \leq C_0 \bar{c}_0 =C_0 c_k.
 \end{equation*}

Hence we can just verify \eqref{eq:prop4.3} for $k=0$. For the simplicity of notations, we will just write $u$ instead of $\bar{u}$ and $u_k:=P_k u$ in the following. \par
Let's apply the projection $P_0$ to the \eqref{eq:halfwm-wave}, we derive
    \begin{align}
        \Box u_0 =& (\partial^2_t - \Delta) u_0 \nonumber \\
        =&P_0 \left[(\nabla u \cdot \nabla u-\partial_t u \cdot \partial_t u) u \right] \label{error1} \\
        &+P_0 \left[\Pi_{\tilde{u}_{\perp}}(\Lah u)\ \left( u \cdot \Lah u \right)\right]  \label{error2} \\
        &+P_0 \left[\Pi_{\tilde{u}_{\perp}} \left(u \times \Lah \left(u \times \Lah u \right) - u \times \left(u \times (-\Delta) u \right) \right) \right] \label{error3}.
    \end{align} \par
We change $\Pi_{u_{\perp}}$ to $\Pi_{\tilde{u}_{\perp}}$ where $\tilde{u}$ maps into a small neighborhood of $S^2$ in the sense that $\|\tilde{u}\|_{L_{t,x}^\infty} \simeq 1$ with $\|\tilde{u}\|_S \lesssim 1$ as the projection only concerns about vector direction.  \par



From the Proposition 4.2 in \cite{2000mathTao} and the multilinear estimate of \cite{krieger2017small},  we can directly write the wave term \eqref{error1} as
\begin{align*}
    P_0\big[(\nabla u \cdot \nabla u-\partial_t u \cdot \partial_t u) u\big] = 2 \left(\nabla u_{\leq -10} \cdot \nabla u_0 -\partial_t  u_{\leq -10} \cdot \partial _t u_0 \right)  \ u_{\leq -10}+ \textit{error}
\end{align*}

Next, we show both the \eqref{error2} and \eqref{error3} terms are \textit{error}. \par
For the \eqref{error2}, we write
\begin{align}  
\Pi_{\tilde{u}_{\perp}}\big((-\triangle)^{\frac{1}{2}}u\big)(u\cdot (-\triangle)^{\frac{1}{2}}u) &= \sum_{|k_1-k_2|\leq 10}\Pi_{\tilde{u}_{\perp}}\big((-\triangle)^{\frac{1}{2}}u\big)(u_{k_1}\cdot (-\triangle)^{\frac{1}{2}}u_{k_2})\label{eq:miltilin2decomp1}\\
&+ \sum_{k_1}\Pi_{\tilde{u}_{\perp}}\big((-\triangle)^{\frac{1}{2}}u\big)(u_{k_1}\cdot (-\triangle)^{\frac{1}{2}}u_{<k_1-10})\label{eq:miltilin2decomp2}\\
&+\sum_{k_2}\Pi_{\tilde{u}_{\perp}}\big((-\triangle)^{\frac{1}{2}}u\big)(u_{<k_2-10}\cdot (-\triangle)^{\frac{1}{2}}u_{k_2})\label{eq:miltilin2decomp3},
\end{align}

and considers the three terms separately. \par

First, we introduce a lemma to deal with the projection term,
\begin{Lemma}
    \label{lem:multilin2}
     For $u, \tilde{u}$ as in the Proposition \ref{reduced}. Then we have the following estimates for the projector term:
     \begin{equation}\label{eq:piu}
        \big\|P_0\big(\Pi_{\tilde{u}^{\perp}}\big((-\triangle)^{\frac12}u\big)\big\|_{S} \lesssim C_0\epsilon,
        \end{equation}
         and by re-scaling, we also have 
        \begin{equation}\label{eq:piu1}
        \big\|P_k\big(\Pi_{\tilde{u}^{\perp}}\big((-\triangle)^{\frac12}u\big)\big\|_{S}\lesssim 2^k C_0\epsilon.
        \end{equation} 
\end{Lemma}
\begin{proof}
  We first write
    \begin{align*}
        P_0\big(\Pi_{\tilde{u}^{\perp}}\big((\Lah u\big) = P_0\Lah u - P_0\big((\frac{\tilde{u}}{\|\tilde{u}\|}\cdot(\Lah u)\frac{\tilde{u}}{\|\tilde{u}\|}\big).
        \end{align*}

        For the term involve $\tilde{u}$, we split it into

        \begin{align*}
        P_0\big((\frac{\tilde{u}}{\|\tilde{u}\|}\cdot(\Lah u)\frac{\tilde{u}}{\|\tilde{u}\|}\big) &= P_0\big((\frac{\tilde{u}}{\|\tilde{u}\|}\cdot(\Lah u_{>k+10})\frac{\tilde{u}}{\|\tilde{u}\|}\big)\\
        & + P_0\big((\frac{\tilde{u}}{\|\tilde{u}\|}\cdot(\Lah u_{[k-10,k+10]})\frac{\tilde{u}}{\|\tilde{u}\|}\big)\\
        &+ P_0\big((\frac{\tilde{u}}{\|\tilde{u}\|}\cdot(\Lah u_{<k-10})\frac{\tilde{u}}{\|\tilde{u}\|}\big)\\
        \end{align*}

    Since $\tilde{u}$ maps into a small neighborhood of $S^2$ , we know that $\big\|\tilde{u}\big\|_{L_{t,x}^\infty} \sim 1$. For $k_1>10$, we have 
    $$\|\tilde{u}_{>k_1-a}\|_{L_t^{\infty}L_x^2} \lesssim 2^{-\frac{n}{2}k_1}, \  \text{for} \ a=5,10.$$
    Using the Bernstein inequality and Hölder's inequality, we have 
    \begin{align*}
    \big\|P_0\left([\frac{\tilde{u}}{\|\tilde{u}\|}]_{>k_1-10}\cdot\Lah u_{k_1}\right)\big\|_{L_t^2L_x^\infty} 
      &\lesssim \big\|P_0\left(\tilde{u}_{>k_1-10}\cdot\Lah u_{k_1}\right)\big\|_{L_t^2L_x^2} \\
      &\le  \ \left\| \tilde{u}_{>k_1-10} \right\|_{L_t^{\infty} L_x^2} \| \Lah u_{k_1} \|_{L_t^2 L_x^{\infty}}\\
      &\lesssim 2^{-\frac{nk_1}{2}}\| \Lah u_{k_1} \|_{L_t^2 L_x^{\infty}}.
    \end{align*}

      Thus we can estimate the first term as 
      \begin{align*}
        &\big\|P_0\big((\frac{\tilde{u}}{\|\tilde{u}\|}\cdot\Lah u_{>10})\frac{\tilde{u}}{\|\tilde{u}\|}\big)\big\|_{L_t^2L_x^\infty\cap L_t^\infty L_x^2} \nonumber \\
        &\leq \sum_{k_1>10}\big\|P_0\big(([\frac{\tilde{u}}{\|\tilde{u}\|}]_{>k_1-10}\cdot\Lah u_{k_1})\frac{\tilde{u}}{\|\tilde{u}\|}\big)\big\|_{L_t^2L_x^\infty\cap L_t^\infty L_x^2} \nonumber \\
        &+\sum_{k_1>10}\big\|P_0\big(([\frac{\tilde{u}}{\|\tilde{u}\|}]_{\leq k_1-10}\cdot\Lah u_{k_1})[\frac{\tilde{u}}{\|\tilde{u}\|}]_{>k_1-5}\big)\big\|_{L_t^2L_x^\infty\cap L_t^\infty L_x^2} \nonumber\\
        &\lesssim  \sum_{k_1>10}\sum_{a=5,10}\big\|[\frac{\tilde{u}}{\|\tilde{u}\|}]_{>k_1-a}\big\|_{L_t^\infty L_x^2}\big\|\Lah u_{k_1}\big\|_{L_t^2L_x^\infty\cap L_t^\infty L_x^2} \\
        &\lesssim \sum_{k_1>10}2^{(1-\frac{n}{2})k_1}\big\|u_{k_1}\big\|_{S}\\
        &\lesssim C_0\epsilon 
        \end{align*}    

    Note here, we use the Cauchy-Schwartz inequality $\sum_{k_1>10}2^{(1-\frac{n}{2})k_1} c_k \lesssim \|c\|_{\ell^2} \lesssim \epsilon$ for the last inequality.

    Similarly, for the second term on the right, we have 
\begin{align*}
\big\|P_0\big((\frac{\tilde{u}}{\|\tilde{u}\|}\cdot(-\triangle)^{\frac12}u_{[-10,10]})\frac{\tilde{u}}{\|\tilde{u}\|}\big)\big\|_{L_t^2L_x^\infty\cap L_t^\infty L_x^2}&\lesssim \big\|(-\triangle)^{\frac12}u_{[-10,10]}\big\|_{L_t^2L_x^\infty\cap L_t^\infty L_x^2}\\
&\lesssim \sum_{k_1\in [-10,10]} 2^{\frac{|k_1|}{2}}\big\|u_{k_1}\big\|_{S} \\
&\lesssim C_0\epsilon 
\end{align*}

Finally, for the last term, we deduce
\begin{align*}
&\big\|P_0\big((\frac{\tilde{u}}{\|\tilde{u}\|}\cdot(-\triangle)^{\frac12}u_{<-10})\frac{\tilde{u}}{\|\tilde{u}\|}\big)\big\|_{L_t^2L_x^\infty\cap L_t^\infty L_x^2}\\
&\lesssim \big\|(-\triangle)^{\frac12}u_{<k-10}\big\|_{L_t^2 L_x^\infty\cap L_{t,x}^{\infty}}\big\|[\frac{\tilde{u}}{\|\tilde{u}\|}]_{>-10}\big\|_{L_t^\infty L_x^2}\\
&\lesssim  \sum_{k_1<-10}2^{-\frac{n}{2}k_1}2^{-\frac{|k_1|}{2}}\big\|u_{k_1}\big\|_{S} \\ 
&\lesssim C_0\epsilon 
\end{align*}

Thus we can conclude that 
\begin{equation}
    \big\|P_0\big((\frac{\tilde{u}}{\|\tilde{u}\|}\cdot(\Lah u)\frac{\tilde{u}}{\|\tilde{u}\|}\big) \big\|_{L_t^2L_x^\infty\cap L_t^\infty L_x^2}  \lesssim C_0\epsilon
\end{equation}

Interpolating between $L_t^2 L_x^\infty$ and $L_t^\infty L_x^2$ norms within the wave admissible range, we obtain that 
\begin{equation*}
    \big\|P_0\big((\frac{\tilde{u}}{\|\tilde{u}\|}\cdot(\Lah u)\frac{\tilde{u}}{\|\tilde{u}\|}\big) \big\|_{S}  \lesssim C_0\epsilon
\end{equation*}

 Scaling the frequency to $\sim 2^k$, we have \eqref{eq:piu1}
\end{proof}
With Lemma \ref{lem:multilin2}, it's easy to see that
\begin{align*}
    &\big\|\Pi_{\tilde{u}^{\perp}}\big((-\triangle)^{\frac12}u \big)\big\|_{L_t^\infty L_x^2+ L_{t,x}^\infty}\lesssim C_0\epsilon, \\
    &\big\|P_{[-20,20]}\big(\Pi_{\tilde{u}_{\perp}}\big((-\triangle)^{\frac{1}{2}}u\big)\big)\big\|_{L_t^\infty L_x^2}\lesssim C_0\epsilon,\\
    &\big\|P_{<10}\Pi_{\tilde{u}_{\perp}}\big((-\triangle)^{\frac{1}{2}}u\big)\big\|_{L_t^2 L_x^\infty}\lesssim C_0\epsilon.
\end{align*}

 We first estimate \eqref{eq:miltilin2decomp1} on the right-hand side. We further split it into
    \begin{align*}
    &\big\|P_0\big[ \sum_{|k_1-k_2|\leq 10}\Pi_{\tilde{u}_{\perp}}\big((-\triangle)^{\frac{1}{2}}u\big)(u_{k_1}\cdot (-\triangle)^{\frac{1}{2}}u_{k_2})\big]\big\|_{L_t^1 L_x^2}\\
    &\lesssim \sum_{\substack{|k_1-k_2|\leq 10\\ k_1<-20}}\big\|P_{[-20,20]}\big(\Pi_{\tilde{u}_{\perp}}\big((-\triangle)^{\frac{1}{2}}u\big)\big)\big\|_{L_t^\infty L_x^2}\big\|u_{k_1}\big\|_{L_t^2 L_x^\infty}\big\|(-\triangle)^{\frac{1}{2}}u_{k_2}\big\|_{L_t^2 L_x^\infty}\\
    & +  \sum_{\substack{|k_1-k_2|\leq 10\\ k_1\geq -20}}\big\|\Pi_{\tilde{u}_{\perp}}\big((-\triangle)^{\frac{1}{2}}u\big)\big\|_{L_t^\infty L_x^2+L_{t,x}^\infty}\big\|u_{k_1}\big\|_{L_t^2 L_x^4}\big\|(-\triangle)^{\frac{1}{2}}u_{k_2}\big\|_{L_t^2 L_x^4}.\\
    \end{align*}

For the first part, we have
        \begin{align*}
        &\sum_{\substack{|k_1-k_2|\leq 10\\ k_1<-20}}\big\|P_{[-20,20]}\big(\Pi_{\tilde{u}_{\perp}}\big((-\triangle)^{\frac{1}{2}}u\big)\big)\big\|_{L_t^\infty L_x^2}\big\|u_{k_1}\big\|_{L_t^2 L_x^\infty}\big\|(-\triangle)^{\frac{1}{2}}u_{k_2}\big\|_{L_t^2 L_x^\infty}\\
        &\lesssim C_0\epsilon \ \sum_{\substack{|k_1-k_2|\leq 10\\ k_1<-20}}2^{\frac{k_2-k_1}{2}} \|u_{k_1}\|_S \ \|u_{k_2}\|_S\\
        &\lesssim  C^3_0 \epsilon^2
        \end{align*}
     
Further, we also have
        \begin{align*}
            &\sum_{\substack{|k_1-k_2|\leq 10\\ k_1\geq -20}}\big\|\Pi_{\tilde{u}_{\perp}}\big((-\triangle)^{\frac{1}{2}}u\big)\big\|_{L_t^\infty L_x^2+L_{t,x}^\infty}\big\|u_{k_1}\big\|_{L_t^2 L_x^\infty}\big\|(-\triangle)^{\frac{1}{2}}u_{k_2}\big\|_{L_t^2 L_x^\infty}\\
            &\lesssim  C_0 \epsilon \ \big(\sum_{\substack{|k_1-k_2|\leq 10\\ k_1\geq -20}}2^{\frac{k_2-k_1}{2}}\big\|u_{k_1}\big\|_{S_{k_1}}\big\|u_{k_2}\big\|_{S_{k_2}}\big)\\
            &\lesssim  C_0^3\epsilon^2
            \end{align*}

        Hence we conclude that \eqref{eq:miltilin2decomp1} is an \textit{error}.  \par
        
        Next, for \eqref{eq:miltilin2decomp2}, we split it into three parts:
        \begin{align*}
        &\sum_{k_1}\Pi_{\tilde{u}_{\perp}}\big((-\triangle)^{\frac{1}{2}}u\big)(u_{k_1}\cdot (-\triangle)^{\frac{1}{2}}u_{<k_1-10})\\&= \sum_{k_1\geq 5}\Pi_{\tilde{u}_{\perp}}\big((-\triangle)^{\frac{1}{2}}u\big)(u_{k_1}\cdot (-\triangle)^{\frac{1}{2}}u_{<k_1-10})\\
        & +  \sum_{k_1\in[-5,5]}\Pi_{\tilde{u}_{\perp}}\big((-\triangle)^{\frac{1}{2}}u\big)(u_{k_1}\cdot (-\triangle)^{\frac{1}{2}}u_{<k_1-10})\\
        & +  \sum_{k_1<-5}\Pi_{\tilde{u}_{\perp}}\big((-\triangle)^{\frac{1}{2}}u\big)(u_{k_1}\cdot (-\triangle)^{\frac{1}{2}}u_{<k_1-10})\\
        \end{align*}

        For the first sum $k_1\geq 5$, we have 
        \begin{align*}
        &\big\|P_0\big( \sum_{k_1\geq 5}\Pi_{\tilde{u}_{\perp}}\big((-\triangle)^{\frac{1}{2}}u\big)(u_{k_1}\cdot (-\triangle)^{\frac{1}{2}}u_{<k_1-10})\big)\big\|_{L_t^1 L_x^2}\\
        &\lesssim  \sum_{k_1\geq 5}\big\|P_{[k_1 - 5,k_1+5]}\big(\Pi_{\tilde{u}_{\perp}}\big((-\triangle)^{\frac{1}{2}}u\big)\big)\big\|_{L_t^\infty L_x^2}\big\|u_{k_1}\big\|_{L_t^2 L_x^\infty}\big\|(-\triangle)^{\frac{1}{2}}u_{<k_1-10}\big\|_{L_t^2 L_x^\infty}\\
        &\lesssim C_0 \epsilon \sum_{\substack{k_1\geq 5 \\ k_2<k_1-10}} 2^{\frac{k_2-k_1}{2}} \|u_{k_1}\|_S \|u_{k_2}\|_S\\
        &\lesssim C_0^3 \epsilon^2
        \end{align*}
    
        Similarly, for the second term, we have 
        \begin{align*}
        &\big\|P_0\big[ \sum_{k_1\in[-5,5]}\Pi_{\tilde{u}_{\perp}}\big((-\triangle)^{\frac{1}{2}}u\big)(u_{k_1}\cdot (-\triangle)^{\frac{1}{2}}u_{<k_1-10})\big]\big\|_{L_t^1 L_x^2}\\
        &\lesssim \sum_{k_1\in[-5,5]}\big\|P_{<10}\Pi_{\tilde{u}_{\perp}}\big((-\triangle)^{\frac{1}{2}}u\big)\big\|_{L_t^2 L_x^\infty}\big\|u_{k_1}\big\|_{L_t^\infty L_x^2}\big\|(-\triangle)^{\frac{1}{2}}u_{<k_1-10}\big\|_{L_t^2 L_x^\infty} \\
        &\lesssim C_0 \epsilon \ \sum_{\substack{k_1 \in[-5,5]  \\ k_2<k_1-10}} \ 2^{-\frac{5k_1}{2}} 2^{\frac{k_2}{2}}\ \|u_{k_1}\|_S \|u_{k_2}\|_S\\
        &\lesssim C^3_0c_0 \epsilon,
        \end{align*}
        
        Finally, we have
        \begin{align*}
        &P_0\big(\sum_{k_1<-5}\Pi_{\tilde{u}_{\perp}}\big((-\triangle)^{\frac{1}{2}}u\big)(u_{k_1}\cdot (-\triangle)^{\frac{1}{2}}u_{<k_1-10})\big)\\
        & = P_0\big(\sum_{k_1<-5}P_{[-2,2]}\big(\Pi_{\tilde{u}_{\perp}}\big((-\triangle)^{\frac{1}{2}}u\big)\big)(u_{k_1}\cdot (-\triangle)^{\frac{1}{2}}u_{<k_1-10})\big)\\
        &\lesssim \big\| P_{[-2,2]}\big(\Pi_{\tilde{u}_{\perp}}\big((-\triangle)^{\frac{1}{2}}u\big)\big) \big\|_{L_t^\infty L_x^2} \sum_{k_1<-5} \| u_{k_1} \|_{L_t^2L_x^\infty} \|\Lah u_{<k_1-10} \|_{L_t^2L_x^\infty} \\
        &\lesssim C_0 \epsilon \sum_{\substack{k_1\leq -5 5 \\ k_2<k_1-10}} 2^{\frac{k_2-k_1}{2}} \|u_{k_1}\|_S \|u_{k_2}\|_S \\
        &\lesssim C^3_0\epsilon^2
        \end{align*}

    Hence, we showed that \eqref{eq:miltilin2decomp2} is an \textit{error}. \par
    
    Lastly, for the third term \eqref{eq:miltilin2decomp3}. We first use the multilinear estimate Lemma~\ref{lem:multilinestimates} to show that the difference
    \begin{align*}
        &\sum_{k_2}\big\| \Pi_{\tilde{u}_{\perp}}\big((-\triangle)^{\frac{1}{2}}u\big)(u_{<k_2-10}\cdot (-\triangle)^{\frac{1}{2}}u_{k_2}) - \Pi_{\tilde{u}_{\perp}}\big((-\triangle)^{\frac{1}{2}}u\big)(-\triangle)^{\frac{1}{2}}(u_{<k_2-10}\cdot u_{k_2}) \big\|_{L_t^1 L_x^2} \\
        &\lesssim C_0 \epsilon \ \sum_{k_2}\sum_{k_3<k_2-10} 2^{k_3} \|u_{k_3}\|_{L_t^2L_x^\infty} \|u_{k_2}\|_{L_t^2L_x^\infty} \\
        &\lesssim C_0 \epsilon \sum_{k_2} \sum_{k_3<k_2-10} 2^{\frac{k_3-k_2}{2}} \|u_{k_2}\|_S \|u_{k_3}\|_S \\
        &\lesssim C^3_0 \epsilon^2,
    \end{align*}
    is an \textit{error} term. Thus it suffices to bound
    \begin{align*}
    &\sum_{k_2}\Pi_{\tilde{u}_{\perp}}\big((-\triangle)^{\frac{1}{2}}u\big)(-\triangle)^{\frac{1}{2}}(u_{<k_2-10}\cdot u_{k_2})= -\sum_{|k_3 - k_4|<10}\frac12\Pi_{\tilde{u}_{\perp}}\big((-\triangle)^{\frac{1}{2}}u\big)(-\triangle)^{\frac{1}{2}}(u_{k_3}\cdot u_{k_4})
    \end{align*}
    where we used \eqref{eq:orthomicro}. Again, we treat the cases for $k_3<-20$ and $k_3 \geq -20$ separately. For the low-frequency case, we have 
    \begin{align*}
    &\big\|P_0\big[\sum_{\substack{|k_3 - k_4|<10\\ k_3<-20}}\frac12\Pi_{\tilde{u}_{\perp}}\big((-\triangle)^{\frac{1}{2}}u\big)(-\triangle)^{\frac{1}{2}}(u_{k_3}\cdot u_{k_4})\big]\big\|_{L_t^1 L_x^2}\\
    &\lesssim \sum_{\substack{|k_3 - k_4|<10\\ k_3<-20}}\big\|P_{[-10,10]}\big[\Pi_{\tilde{u}_{\perp}}\big((-\triangle)^{\frac{1}{2}}u\big)\big]\big\|_{L_t^\infty L_x^2}\big\|(-\triangle)^{\frac{1}{2}}(u_{k_3}\cdot u_{k_4})\big\|_{L_t^1 L_x^\infty} \\
    &\lesssim \big\|P_{[-10,10]}\big[\Pi_{\tilde{u}_{\perp}}\big((-\triangle)^{\frac{1}{2}}u\big)\big]\big\|_{L_t^\infty L_x^2} \ 
    \sum_{\substack{|k_3 - k_4|<10\\ k_3<-20}}\big\|(-\triangle)^{\frac{1}{2}}(u_{k_3}\cdot u_{k_4})\big\|_{L_t^1 L_x^\infty} \\
    &\lesssim C_0\epsilon \sum_{\substack{|k_3 - k_4|<10\\ k_3<-20}}2^{k_3}\big\|u_{k_3}\big\|_{L_t^2L_x^\infty}\big\|u_{k_4}\big\|_{L_t^2L_x^\infty}\\
    &\lesssim C_0^3 \epsilon^2, 
    \end{align*}

   For $k_3>-20$, we estimate it as above to conclude 
    \begin{align*}
        &\big\|P_0\big[\sum_{\substack{|k_3 - k_4|<10\\ k_3>-20}}\Pi_{\tilde{u}_{\perp}}\big((-\triangle)^{\frac{1}{2}}u\big)(-\triangle)^{\frac{1}{2}}(u_{k_3}\cdot u_{k_4})\big]\big\|_{L_t^1 L_x^2}\\
        &\lesssim \sum_{\substack{|k_3 - k_4|<10\\ k_3>-20}}\big\|\big[\Pi_{\tilde{u}_{\perp}}\big((-\triangle)^{\frac{1}{2}}u\big)\big]\big\|_{L_t^2L_x^\infty+L_t^\infty L_x^\infty}\big\|(-\triangle)^{\frac{1}{2}}(u_{k_3}\cdot u_{k_4})\big\|_{L_t^1 L_x^2} \\
        &\lesssim  C_0 \epsilon \sum_{\substack{|k_3 - k_4|<10\\ k_3>-20}} \big\|(-\triangle)^{\frac{1}{2}}(u_{k_3}\cdot u_{k_4})\big\|_{L_t^1 L_x^2} \\
        &\lesssim  C_0\epsilon \sum_{\substack{|k_3 - k_4|<10\\ k_3>-20}} 2^{k_3+k_4} \big\|u_{k_3}\big\|_{L_t^2L_x^4}\big\|u_{k_4}\big\|_{L_t^2L_x^4}\\ 
        &\lesssim C_0\epsilon \sum_{\substack{|k_3 - k_4|<10\\ k_3>-20}} 2^{-\frac{3(k_3+k_4)}{4}} \big\|u_{k_3}\big\|_{S}\big\|u_{k_4}\big\|_{S}\\ 
        &\lesssim  C^3_0 \epsilon^2.
        \end{align*}
        Combining all the estimates above,
        we conclude that \eqref{eq:miltilin2decomp2}
        \begin{align*}
            P_0\big[\Pi_{\tilde{u}_{\perp}}\big((-\triangle)^{\frac{1}{2}}u\big)(u\cdot (-\triangle)^{\frac{1}{2}}u) \big]=\textit{error}
        \end{align*}

Lastly, we consider the most complicated term \eqref{error3}. As in the Lemma \ref{lem:multilin2}, we know that for a function $F$ s.t.
 $$\big\|P_k F \big\|_{L_t^1 \dot{H}^{\frac{n}{2}-1}} \lesssim 2^{-\sigma|k|} C_0^3c_0\epsilon,$$ 
 for some $\sigma>1$, we have
\begin{equation*}
    \big\|P_0[\Pi_{\tilde{u}^\perp} F] \big\|_{L_t^1 L_x^2} \lesssim \sum_{k} 2^{-|k|} \big\|P_k F \big\|_{L_t^1 \dot{H}_x^{\frac{n}{2}-1}} \lesssim C_0^3 c_0 \epsilon.
\end{equation*}

Therefore, it's sufficient to show that 
\begin{equation}
   \big\|\sum_{k_1,k_2 \in \Z}\ P_0 \big[ u \times \Lah \left(u_{k_1} \times \Lah u_{k_2} \right) - u \times \left(u_{k_1} \times (-\Delta) u_{k_2} \right) \big] \big\|_{L_t^1 L_x^2} \lesssim C^3_0 c_0 \epsilon.   
\end{equation}

We consider different cases in the following. For $k_1>10$, we consider 
\begin{align*}
    &\big\|\sum_{\substack{k_1>10 \\ k_2 \in \Z}}P_0\big[u \times (-\triangle)^{\frac{1}{2}}(u_{k_1} \times (-\triangle)^{\frac{1}{2}}u_{k_2}) - u \times(u_{k_1} \times (-\triangle)u_{k_2})\big]\big\|_{L_t^1 L_x^2}\\
    &=\big\|\sum_{\substack{|k_1 - k_2|<5\\k_1>10}}P_0\big[u \times (-\triangle)^{\frac{1}{2}}(u_{k_1} \times (-\triangle)^{\frac{1}{2}}u_{k_2}) - u \times(u_{k_1} \times (-\triangle)u_{k_2})\big]\big\|_{L_t^1 L_x^2}\\
    &+\big\|\sum_{\substack{k_1>10 \\ k_2<k_1-5}  }P_0\big[u \times (-\triangle)^{\frac{1}{2}}(u_{k_1} \times (-\triangle)^{\frac{1}{2}}u_{k_2}) - u \times(u_{k_1} \times (-\triangle)u_{k_2})\big]\big\|_{L_t^1 L_x^2} \\
    &+\big\|\sum_{\substack{k_1>10 \\ k_2>k_1+5}  }P_0\big[u \times (-\triangle)^{\frac{1}{2}}(u_{k_1} \times (-\triangle)^{\frac{1}{2}}u_{k_2}) - u \times(u_{k_1} \times (-\triangle)u_{k_2})\big]\big\|_{L_t^1 L_x^2}
\end{align*}

For the first term, we use Lemma \ref{lem:multilinestimates}, in particular, \eqref{singlefreq} and \eqref{singlefreq2}. We have
\begin{align*}
&\big\|\sum_{\substack{|k_1 - k_2|<5\\k_1>10}}P_0\big[u \times (-\triangle)^{\frac{1}{2}}(u_{k_1} \times (-\triangle)^{\frac{1}{2}}u_{k_2}) - u \times(u_{k_1} \times (-\triangle)u_{k_2})\big]\big\|_{L_t^1 L_x^2}\\
&\lesssim \|u\|_{L_{t,x}^\infty} \ \sum_{\substack{|k_1 - k_2|<5\\k_1>10}} 2^{k_1+k_2}\big\|P_{k_1}u\big\|_{L_t^2 L_x^4}\big\|u_{k_2}\big\|_{L_t^2 L_x^4} \\
&\lesssim \sum_{\substack{|k_1 - k_2|<5\\k_1>10}}2^{-\frac {3}{4}(k_1+k_2)}\big\|u_{k_1}\big\|_{S} \big\|u_{k_2}\big\|_{S}\\
&\lesssim C^2_0\varepsilon^2.
\end{align*}

Then, for the second term, we have
\begin{align*}
&\big\|\sum_{\substack{k_1>10 \\ k_2<k_1-5}  }P_0\big[u \times (-\triangle)^{\frac{1}{2}}(u_{k_1} \times (-\triangle)^{\frac{1}{2}}u_{k_2}) - u \times(u_{k_1} \times (-\triangle)u_{k_2})\big]\big\|_{L_t^1 L_x^2}\\
    &\lesssim \sum_{\substack{k_1>10 \\ k_2<k_1-5}} \|u_{k_1}\|_{L_t^2L_x^\infty}\  2^{k_1+k_2} \|u_{k_1}\|_{L_t^\infty L_x^2} \ \|u_{k_2}\|_{L_t^2 L_x^\infty} \\
    &\lesssim  \sum_{k_1>10} 2^{-\frac{3}{2}k_1} \|u_{k_1}\|^2_S \ \sum_{k_2<k_1-5} 2^{\frac{k_2-k_1}{2}}\|u_{k_2}\|_S \\
    &\lesssim C_0^3 \epsilon^3
 \end{align*}

The third case, when $k_1>10, \ k_2>k_1+5$, can be symmetrically estimated as 
\begin{align*}
    &\big\|\sum_{\substack{k_1>10 \\ k_2>k_1+5}} P_0\big[u \times (-\triangle)^{\frac{1}{2}}(u_{k_1} \times (-\triangle)^{\frac{1}{2}}u_{k_2}) - u \times(u_{k_1} \times (-\triangle)u_{k_2})\big]\big\|_{L_t^1 L_x^2}\\
        &\lesssim \sum_{\substack{k_1>10 \\ k_2>k_1+5}} \|u_{k_2}\|_{L_t^2L_x^\infty}\  2^{k_1+k_2} \|u_{k_1}\|_{L_{t,x}^\infty} \ \|u_{k_2}\|_{L_t^2 L_x^4} \\
        &\lesssim \sum_{\substack{k_1>10 \\ k_2>k_1+5}} 2^{k_1-k_2}\ 2^{-\frac{3}{2}k_2} \|u_{k_2}\|^2_S \\
        &\lesssim C_0^2 \epsilon^3
     \end{align*}

Hence, we conclude that when $k_1>10$, \eqref{error3} is an \textit{error}. Moreover, the above estimates can be used to get the bound for $k_2>10$ as well. Therefore, we conclude that  \eqref{error3} is an an \textit{error} whence $\max\{k_1, k_2\}>10$. \par

Next, we consider the case for $k_1,k_2<-10$. By Lemma~\ref{lem:multilinestimates}, we have
\begin{align*}
&P_0\big[ u \times (-\triangle)^{\frac12}(u_{k_1} \times (-\triangle)^{\frac12}u_{k_2})\big] =
P_0\big[ u \times ((-\triangle)^{\frac12}u_{k_1} \times (-\triangle)^{\frac12}u_{k_2})\big] +\textit{error} \\
&P_0\big[ u \times \left(u_{k_1} \times (-\Delta) u_{k_2} \right)    \big] =
P_0\big[ u \times ((-\triangle)^{\frac12}u_{k_1} \times (-\triangle)^{\frac12}u_{k_2})\big] +\textit{error}.
\end{align*}

Therefore, we can conclude that
\begin{align*}
    &\big\| \sum_{k_1,k_2<-10} P_0\big[ P_{[-5,5]}u \times \Lah u_{k_1} \times \Lah u_{k_2}\big] \big\|_{L_t^1 L_x^2}  \\
    &\lesssim \sum_{k_1,k_2<-10} \|P_{[-5,5]}u\|_{L_t^{\infty} L_x^2} \|\Lah u_{k_1}\|_{L_t^2 L_x^\infty} \|\Lah u_{k_2}\|_{L_t^2 L_x^\infty} \\
    &\lesssim  C_0 c_0 \sum_{k_1,k_2<-10}\  2^{\frac{k_1+k_2}{2}} \|u_{k_1}\|_S \|u_{k_2}\|_S \\
    &\lesssim C_0^2 c_0 \epsilon^2
\end{align*}

Thus, we know in this case \eqref{error3} is an \textit{error} as well. \par

Finally, we consider the scenario when one frequency is intermediate in $[-10,10]$ and the other one is small. This is the most delicate case.\par
{\it{(i): $k_1\in [-10,10], k_2<10$.}} Firstly, we have
\begin{align*}
\big\|\sum_{\substack{k_1 \in[-10,10] \\ k_2<10}} 
 P_0\big[ u \times(u_{k_1} \times (-\triangle)u_{k_2})\big]\big\|_{L_t^1 L_x^2} 
&\lesssim \sum_{\substack{k_1 \in[-10,10] \\ k_2<10}} \ \big\|u_{k_1} \big\|_{L_t^2 L_x^4}\big\| (-\triangle)u_{k_2})\big\|_{L_t^2 L_x^4}\\
&\lesssim \sum_{\substack{k_1 \in[-10,10] \\ k_2<10}}\ \ 2^{\frac{k_2}{4}}\big\|u_{k_2}\big\|_{S}\big\|u_{k_1}\big\|_{S} \\
&\lesssim C_0^2 c_0 \epsilon.
\end{align*}

As before, Lemma~\ref{lem:multilinestimates} gives us
\[
P_0\big[ u \times (-\triangle)^{\frac12}(u_{k_1} \times (-\triangle)^{\frac12}u_{k_2})\big] =
P_0\big[ u \times ((-\triangle)^{\frac12}u_{k_1} \times (-\triangle)^{\frac12}u_{k_2})\big] +\textit{error}.
\]

Then we do a further decomposition:
\begin{align*}
    &\big\|P_0\big[ u \times ((-\triangle)^{\frac12}u_{k_1} \times (-\triangle)^{\frac12}u_{k_2})\big]\big\|_{L_t^1 L_x^2}\\
    &=\big\|P_0\big[ u_{\geq k_2-10} \times ((-\triangle)^{\frac12}u_{k_1} \times (-\triangle)^{\frac12}u_{k_2})\big]\big\|_{L_t^1 L_x^2}\\
    &+\big\|P_0\big[ u_{\leq k_2-10} \times ((-\triangle)^{\frac12}u_{k_1} \times (-\triangle)^{\frac12}u_{k_2})\big]\big\|_{L_t^1 L_x^2}.
\end{align*}

When the first factor $u$'s frequency is larger than $2^{k_2-10}$, we have
\begin{align*}
&\big\|\sum_{\substack{k_1 \in[-10,10] \\ k_2<10}} \ P_0\big[ u_{\geq k_2-10} \times ((-\triangle)^{\frac12}u_{k_1} \times (-\triangle)^{\frac12}u_{k_2})\big]\big\|_{L_t^1 L_x^2}\\
&\lesssim \sum_{\substack{k_1 \in[-10,10] \\ k_2<10}}\big\|u_{\geq k_2-10}\big\|_{L_t^2L_x^\infty}\big\|(-\triangle)^{\frac12}u_{k_1}\big\|_{L_t^\infty L_x^2}\big\| (-\triangle)^{\frac12}u_{k_2}\big\|_{L_t^2 L_x^\infty}\\
&\lesssim \sum_{\substack{k_1 \in[-10,10] \\ k_2<10}} 2^{-\frac{3}{2}k_1} \|u_{k_1}\|_S \sum_{k_3>k_2-10} 2^{\frac{k_2-k_3}{2}} \ \|u_{k_3}\|_S \|u_{k_2}\|_S \\
&\lesssim C_0^3c_0\epsilon^2.
\end{align*}

Move the the second part, by the cross product rule $$a \times (b \times c)= b(a\cdot c)-c(a\cdot b),$$ we have  
\begin{align}
&P_0\big[ u_{<k_2-10} \times ((-\triangle)^{\frac12}u_{k_1} \times (-\triangle)^{\frac12}u_{k_2})\big] \nonumber \\
& = P_0\big[(-\triangle)^{\frac12}u_{k_1}( u_{<k_2-10} \cdot (-\triangle)^{\frac12}u_{k_2})\big] \label{eq:interlow1}\\
 &- P_0\big[(-\triangle)^{\frac12}u_{k_2}(u_{<k_2-10} \cdot(-\triangle)^{\frac12}u_{k_1})\big] \label{eq:interlow2}
\end{align}

Again, by the Lemma \ref{lem:multilinestimates}, we have
\begin{align*}
    u_{<k_2-10} \cdot \Lah u_{k_2} &= \Lah (u_{<k_2-10} \cdot P_{k_2} u_{[k_2-10,k_2+10]}) +\textit{error} \\
&= \Lah P_{k_2}(u_{<k_2-10} \cdot u_{[k_2-10,k_2+10]}) +\textit{error}
\end{align*}

Using the geometric property property \eqref{eq:orthomicro} of the $S^2$, we have
\begin{equation}\label{eq:orthomicrolocalized}
0=P_{k}(u \cdot u) = 2 P_{k} (u_{[k-10,k+10]}\cdot u_{<k-10}) + P_{k}(u_{>k+10} \cdot u_{>k+10}).
\end{equation}

Therefore we can replace
\begin{equation}
    \label{error3_local}
    P_{k_2}(u_{<{k_2}-10} \cdot u_{[{k_2}-10,{k_2}+10]}) = -\frac{1}{2} P_{k_2}(u_{>{k_2}+10} \cdot u_{>{k_2}+10}).
\end{equation}
 
Hence, we estimate \eqref{eq:interlow1} as

\begin{align*}
    &\big\| \sum_{k_1\in [-10,10]} P_0\big[(-\triangle)^{\frac12}u_{k_1} \sum_{k_2<10} ( u_{<k_2-10} \cdot (-\triangle)^{\frac12} u_{k_2}) \big] \big\|_{L_t^1 L_x^2} \\
    &\simeq \sum_{k_1\in [-10,10]} \big\| P_0\big[\Lah u_{k_1} \sum_{k_2<10} \Lah P_{k_2} (u_{>k_2+10} \cdot u_{>k_2+10}) \|_{L_t^1 L_x^2}  \\
    &\lesssim \sum_{k_1\in [-10,10]} \|\Lah u_{k_1} \|_{L_t^\infty L_x^2} \sum_{k_2<10} 2^{k_2} \|u_{>k_2+10} \|^2_{L_t^2 L_x^\infty} \\
    &\lesssim \sum_{k_1\in [-10,10]} 2^{-\frac{5k_1}{2}}\|u_{k_1}\|_S \ \sum_{k_2<10} \|u_{>k_2+10}\|^2_S \\
    &\lesssim C_0^3c_0\epsilon^2
\end{align*}

We estimate the term \eqref{eq:interlow2} similarly. So we conclude in the case {{(i)}}, \eqref{error3} is an error term. \par

{\it{(ii)}}: $k_2\in [-10,10], k_1<10$. As shown in the case {\it{(i)}}, we further split the cases with respect to the frequency of the first term. For $u_{\geq k_1-10}$, we treat it as above. Thus, we are only left to show the case for $u_{< k_1-10}$.

Using Lemma~\ref{lem:multilinestimates} again, we have 
\begin{align*}
&\big\|\sum_{\substack{k_1<10 \\ k_2 \in [-10,10]}} P_0\big[u_{<k_1-10} \times (-\triangle)^{\frac{1}{2}}(u_{k_1} \times (-\triangle)^{\frac{1}{2}}u_{k_2}) \\
&\quad - (-\triangle)^{\frac{1}{2}}\big(u_{<k_1-10} \times (u_{k_1} \times (-\triangle)^{\frac{1}{2}}u_{k_2})\big)\big]\big\|_{L_t^1 L_x^2}\\
& \lesssim\sum_{\substack{k_1<10 \\ k_2 \in [-10,10]}} \big\|(-\triangle)^{\frac{1}{2}}u_{<k_1-10}\big\|_{L_t^2 L_x^\infty}\big\|u_{k_1}\big\|_{L_t^2 L_x^\infty}\big\|(-\triangle)^{\frac{1}{2}}u_{k_2}\big\|_{L_t^\infty L_x^2} \\
 &\lesssim C_0c_0\  \sum_{\substack{k_1<10 \\ l-k_1:=a \leq -10}} 2^{\frac{a}{2}} \|u_{a+k_1}\|_S \|u_{k_1}\|_S \\ 
 &\lesssim C_0^3c_0 \epsilon^2.
\end{align*}

So we can consider $P_0 \big[(-\triangle)^{\frac{1}{2}}\big(u_{<k_1-10} \times (u_{k_1} \times (-\triangle)^{\frac{1}{2}}u_{k_2})\big)\big]$ instead. We reformulate \eqref{error3} as 

\begin{equation}\label{eq:thefourterms}\begin{split}
&P_0\big[(-\triangle)^{\frac{1}{2}}\big(u_{<k_1-10} \times (u_{k_1} \times (-\triangle)^{\frac{1}{2}}u_{k_2})\big) - u_{<k_1-10} \times(u_{k_1} \times (-\triangle)u_{k_2})\big]\\
& = P_0 \big[ (-\triangle)^{\frac{1}{2}}\big( u_{k_1}(u_{<k_1-10}\cdot (-\triangle)^{\frac{1}{2}}u_{k_2}) -  (-\triangle)^{\frac{1}{2}}u_{k_2}(u_{<k_1-10}\cdot u_{k_1}) \big) \big]\\
& - P_0[\big(u_{k_1}( u_{<k_1-10}\cdot (-\triangle)u_{k_2}) - (-\triangle)u_{k_2}(u_{<k_1-10} \cdot u_{k_1})\big)] \\
&= P_0 \big[ (-\triangle)^{\frac{1}{2}}\big( u_{k_1}(u_{<k_1-10}\cdot (-\triangle)^{\frac{1}{2}}u_{k_2}) -  \big(u_{k_1}( u_{<k_1-10}\cdot (-\triangle)u_{k_2} \big) \big]\\
&+P_0\big[ (-\triangle)u_{k_2}(u_{<k_1-10} \cdot u_{k_1})\big) -\Lah\big( \Lah u_{k_2} (u_{<k_1-10}\cdot u_{k_1}) \big) \big].
\end{split}\end{equation}

Using the multi-linear estimates again, we know the differences 
\begin{align*}
     u_{<k_1-10} \cdot \Lah u_{k_2} - \Lah(u_{<k_1-10} \cdot u_{k_2})  \sim \textit{error},
 \end{align*}

 and

 \begin{align*}
      u_{<k_1-10} \cdot (-\triangle) u_{k_2}-  (-\triangle) (u_{<k_1-10} \cdot u_{k_2}) \sim \textit{error}. 
 \end{align*}

Additionally, by putting $u_{k_1}$ into $L_t^2 L_x^\infty$, we also know that
\begin{align*}
&P_0\big[(-\triangle)^{\frac{1}{2}}\big(u_{k_1}(u_{<k_1-10}\cdot (-\triangle)^{\frac{1}{2}}u_{k_2})\big) \big]-  P_0\big[u_{k_1}( u_{<k_1-10}\cdot (-\triangle)u_{k_2})\big]\\
& = P_0 \big[(-\triangle)^{\frac{1}{2}}\big(u_{k_1}(-\triangle)^{\frac{1}{2}}(u_{<k_2-10}\cdot u_{k_2}) \big] - P_0\big[u_{k_1}(-\triangle)( u_{<k_2-10}\cdot u_{k_2})\big] +\textit{error}.
\end{align*}

Then we use the geometric property of $S^2$ to replace the product $u_{<k_2-10}\cdot u_{k_2}$. We conclude that
\begin{align*}
    &\sum_{\substack{k_1<10 \\k_2\in[-10,10]}} \big\| P_0 \big[(-\triangle)^{\frac{1}{2}}\big(u_{k_1}(-\triangle)^{\frac{1}{2}}(u_{<k_2-10}\cdot u_{k_2}) \big] - P_0\big[u_{k_1}(-\triangle)( u_{<k_2-10}\cdot u_{k_2})\big]\big\|_{L_t^1 L_x^2} \\
&\lesssim \sum_{\substack{k_1<10 \\k_2\in[-10,10]}} \big\|(-\triangle)^{\frac{1}{2}}u_{k_1}\big\|_{L_t^2 L_x^\infty} \ 2^{k_2}\ \big\|u_{>k_2+10}\big\|_{L_t^2 L_x^\infty}\big\|u_{>k_2+10}\big\|_{L_t^\infty L_x^2} \\
&\lesssim \sum_{k_1<10} \ 2^{\frac{k_1}{2}} \|u_{k_1}\|_S \ \sum_{\substack{k_2\in[-10,10] \\k_3>k_2+10}} \  2^{-2k_3} \|u_{k_3}\|^2_S \\
&\lesssim  C^3_0 \epsilon^3.
\end{align*}

Then it only remains to bound the difference 
\begin{equation*}
    P_0\big[ (-\triangle)u_{k_2}(u_{<k_1-10} \cdot u_{k_1})\big) -\Lah\big( \Lah u_{k_2} (u_{<k_1-10}\cdot u_{k_1}) \big) \big].
\end{equation*}
As before, we replace  
\begin{equation*}
    P_{k_1}(u_{<k_1-10} \cdot u_{k_1}) = -\frac{1}{2} P_{k_1}(u_{>k_1+10} \cdot u_{>k_1+10}).
\end{equation*}
Then we deduce
\begin{align*}
    &\sum_{\substack{k_1 <10 \\k_2\in [-10,10]} } \| P_0\big[ (-\triangle)u_{k_2}P_{k_1}(u_{>k_1+10} \cdot u_{>k_1+10} ) \\
    &-\Lah \big( (-\triangle)^{\frac{1}{2}}u_{k_2} P_{k_1}(u_{>k_1+10} \cdot u_{>k_1+10} ) \big)\big]\|_{L_t^1 L_x^2} \\
    &\lesssim \sum_{\substack{k_1 <10 \\k_2\in [-10,10]} } 2^{k_1+k_2} \| P_{k_1}(u_{>k_1+10} \cdot u_{>k_1+10} ) \|_{L_t^1 L_x^\infty} \| u_{k_2}\|_{L_t^\infty L_x^2} \\
    &\lesssim  \sum_{\substack{k_1 <10 \\k_2\in [-10,10]} }  2^{-\frac{3}{2}k_2} \|u_{k_2}\|_S\ \sum_{k_3>k_1+10} 2^{k_1-k_3}\|u_{k_3}\|^2_{S} \\
    &\lesssim C_0^3 c_0 \epsilon^2.
    \end{align*}

Hence we proved that \eqref{error3} is an \textit{error} term.\par
Overall, we showed that all the additional nonlinearity terms introduced by the half-wave equation are \textit{error} terms. Thus, we obtain the linearized equation:
\begin{equation}  
    \label{eq:hwm-linear}
    \Box u_0 = -2 u_{\leq -10} (\nabla_{t,x} u_{\leq -10} \cdot \nabla_{t,x} u_0) + \textit{error} 
\end{equation}

With the linearized equation \eqref{eq:hwm-linear}, we are able to directly apply the result from Tao's approximately parallel transport method to renormalize the equation. We construct an approximately orthogonal $3 \times 3$ invertible matrix $U$ to transform $u_0= U w$. Then we can recast \eqref{eq:hwm-linear} as
\begin{equation}
    \label{eq:w}
    \Box w= \textit{error}
\end{equation}

By the construction of $U$, we will have
$$\|u_0\|_S \lesssim \|w\|_S.$$

Thus it suffices to show that 
$$
\|w\|_S \lesssim C_0c_0
$$

to conclude the Proposition \ref{reduced}. \par
We further extract an antisymmetric structure from \eqref{eq:hwm-linear} to have a more canceling structure.
 From now on, we view $u$ as a column vector of $\R^3$. We rewrite \eqref{eq:hwm-linear} as 
$$
\Box u_0 = -2 u_{\leq 10} \  \partial_\alpha u_{\leq -10}^T \ \partial^\alpha u_0 + \textit{error}.
$$

We know that
\begin{equation*}
    \label{2infty-bound}
    \|\partial_\alpha u_{\leq -10}\|_{L_t^2 L_x^\infty} \lesssim C_0\varepsilon,
\end{equation*}

and 

\begin{equation*}
    \label{sph1}
    \| u_{\leq -10}^T \partial_\alpha u_0\|_{L_t^2 L_x^2} \lesssim  C_0^2 c_0 \varepsilon.
\end{equation*}

Apply H\"older inequality, we have
$$
\partial_\alpha u_{\leq -10} u_{\leq -10}^T \partial^\alpha u_0 = \textit{error}.$$

Thus we can write our equation as 
\begin{equation*}
\Box u_0 = -2 \left( u_{\leq -10} \ \partial_\alpha u^T_{\leq -10}\ \partial_\alpha u_0 - \partial^\alpha u_{\leq -10} \ u^T_{\leq -10} \ \partial^\alpha u_0 \right)+ \textit{error}.
\end{equation*}

We define an anti-symmetric $3 \times3$ matrix
$$ A_\alpha := \partial_\alpha u_{\leq -10} u_{\leq -10}^T - u_{\leq -10} \partial_\alpha u_{\leq -10}^T.$$

 Therefore, our equation formally become 
\begin{equation}
    \label{eq:cancel}
    \Box u_0= 2A_\alpha \partial^\alpha u_0 +\textit{error}.
\end{equation}

In the following, we briefly state the approximate parallel transport of Tao \cite{2000mathTao}.
We construct the real $3 \times 3$-valued matrix field U by induction. Let $U_0:=I$ be the identity matrix. For $k \geq 1$, we set 
\begin{equation}
    \label{ack-accurate}
U_k := (u_k u_{<k}^T) - u_{<k} u_k^T) U_{<k}
\end{equation}

and

$$ U_{<k} := I + \sum_{-M < k' < k} U_{k'},$$
where $M$ is a large integer depending on $T$ to be chosen later. Then we define
$$ U := I + \sum_{-M < k \leq -10} U_k.$$

With this inductive construction, we have the following proposition: 

\begin{Proposition}[Proposition 6.1 in \cite{2000mathTao}]
    \label{u-est}
With the same assumption of Proposition \ref{reduced}, in which $\varepsilon$ is sufficiently small depending on $C_0$. We have the following properties of the invertible matrix $U$:
\begin{enumerate}[label=(\roman*)]
    \item 
        \label{u-invert}
    $\| U^T U - I \|_{L^\infty_t L^\infty_x}, \ 
    \| \partial_t(U^T U - I) \|_{L^\infty_t L^\infty_x}
     \lesssim C_0^2 \varepsilon.$
    \item 
        \label{u-infty}
    $\| U \|_{L^\infty_t L^\infty_x}, \ \| U^{-1} \|_{L^\infty_t L^\infty_x} \lesssim 1.$
    \item 
        \label{1infty}
    $\| \partial_\alpha U - A_\alpha U \|_{L^1_t L^\infty_x} \lesssim C_0^2 \varepsilon.$
    \item 
        \label{u-energy}
    $\| \partial_\alpha U \|_{L^\infty_t L^\infty_x} \lesssim C_0^2 \varepsilon.$
    \item 
        \label{u2}
    $\| \partial_\alpha U \|_{L^2_t L^\infty_x} \lesssim C_0^2 \varepsilon.$
    \item 
        \label{u-dd}
    $\| \Box U \|_{L^2_t L^{n-1}_x} \lesssim C_0^2 \varepsilon$
\end{enumerate}
\end{Proposition}

Since $U$ is invertible, we set $w:= U^{-1} u_0$, which is smooth by the assumption. The equation \eqref{eq:cancel} becomes
\begin{equation}
    \label{eq:w}
\Box w = -2 U^{-1} (\partial_\alpha U - A_\alpha U)\  \partial^\alpha w + 2 U^{-1} \  A_\alpha \ \partial^\alpha U \  U^{-1} u_0 - U^{-1}\  \Box U \  U^{-1} u_0 + \textit{error}.
\end{equation}

Now we show that all the terms on the right of \eqref{eq:w} are \textit{error} terms by using the Proposition \ref{u-est}. For the first term, we write
$$ \partial^\alpha w = U^{-1} \partial^\alpha u_0 + U^{-1} (\partial^\alpha U) U^{-1}u_0.$$

Then we have 
\begin{align*}
    &\|U^{-1} (\partial_\alpha U - A_\alpha U) \partial^\alpha w \|_{L_t^1 L_x^2} \\
    &\lesssim \|U^{-1}\|_{L_{t,x}^\infty} \ \| \partial_\alpha U - A_\alpha U\|_{L_t^1 L_x^\infty} \ \|\partial^\alpha w\|_{L_t^\infty L_x^2} \\
    &\lesssim C_0^2 \epsilon \left( \|U^{-1} \partial^\alpha u_0\|_{L_t^\infty L_x^2} + \|U^{-1} \partial^\alpha U U^{-1}u_0 \|_{L_t^\infty L_x^2} \right) \\
    &\lesssim C_0^2 \epsilon \left(\|U^{-1}\|_{L_{t,x}^\infty} \|\partial^\alpha u_0 \|_{L_t^\infty L_x^2} + \|U^{-1}\|_{L_{t,x}^2} \|\partial_\alpha U\|_{L_{t,x}^\infty} \|u_0\|_{L_t^\infty L_x^2} \right) \\
    &\lesssim C_0^2 \epsilon (C_0 c_0+C_0^2 c_0 \epsilon) \\
    &\lesssim C_0^3c_0\epsilon 
\end{align*}

For the second term on the right, 
\begin{align*}
    &\|U^{-1} \  A_\alpha \ \partial^\alpha U \  U^{-1} u_0\|_{L_t^1L_x^2}  \\
    &\lesssim \|U^{-1}\|^2_{L_{t,x}^\infty} \|\partial_\alpha U\|_{L_t^2 L_x^\infty} \|A_{\alpha}\|_{L_t^2 L_x^\infty} \| u_0 \|_{L^\infty_t L^2_x} \\
    &\lesssim C_0^2 c_0^2 \epsilon
\end{align*}

Notice here, we use the estimate
$$ \| A_\alpha \|_{L^2_t L^\infty_x} \lesssim C_0 \varepsilon$$ \par

Lastly, we deduce that
\begin{equation*}
    \|U^{-1} (\Box U) U^{-1} u_0\|_{L_t^1 L_x^2} \lesssim \|\Box U\|_{L_t^2L_x^4} \|u_0\|_{L_t^2 L_x^4} \lesssim C_0^2 c_0 \epsilon.
\end{equation*}

Therefore, we showed that the right-hand side of \eqref{eq:w} is \textit{error}. Hence we conclude
\begin{equation}
    \Box w= \textit{error}.
\end{equation} \par

By the estimates \ref{u-infty}, \ref{u-energy}, we have 
$$ \| u_0 \|_{S} \lesssim \| w \|_{S}.$$ \par
So it suffices to show that
\begin{equation}
\label{w-targ}
\|w \|_{S} \le C_0 c_0.
\end{equation} \par
Although $u_0$ is supported on the frequency annulus $|\xi| \sim 1$, $w$ is not supported on the frequency annulus $|\xi| \sim 1$ because of the multiplication of $U$. By the construction of $U$, we can see that $U$ is supported on the ball $|\xi|\leq 2^{-7}$. So it suffices to consider $w$ on the frequency annulus $2^{-10} \leq |\xi| \leq 2^{10}$.  \par
We use the localized version of the standard Strichartz estimate Proposition \ref{prop:linestimates} to estimate $\|w\|_S$. We have
\begin{equation}
    \|P_k w\|_S \leq \| P_k w(0)\|_{\dot{H}^\frac{n}{2}} + \| P_k \partial_t w(0)\|_{\dot{H}^{\frac{n}{2}-1}} +2^{(\frac{n}{2}-1)k} \|\Box P_k w\|_{L_t^1 L_x^2}, 
\end{equation}
 \par
The initial condition 
\begin{equation*}
    \begin{cases}
        &w(0)=U^{-1}(0), \\
&\partial_t w(0)=U^{-1}(0) \partial_t u(0)- U^{-1}(0) \partial_t U(0) U^{-1}(0) u(0)
    \end{cases}
\end{equation*}

can be easily bounded by 
$$
\|w(0)\|_{L_x^2}+\|\partial_t w(0)\|_{L_x^2} \lesssim \|u[0]\|_{L_x^2} \lesssim C_0 c_0.
$$

By choosing $\epsilon$ small enough s.t. $C_0^2 \epsilon \ll 1$, we have 
\begin{align*}
    \|P_{-10\leq k \leq 10} w\|_S \lesssim \sum_{-10<k<10} \big(2^{\frac{nk}{2}}+2^{(\frac{n}{2}-1)k} \big) C_0c_0+ 2^{(\frac{n}{2}-1)k} C_0^3 \epsilon c_0 \lesssim C_0c_0.
\end{align*}

Therefore, we conclude the Proposition \ref{reduced}. \par

\subsection{Solution Of The Half-Wave Map Equation}
\label{sec:sol}

From Proposition \ref{prop:waveform}, we obtain the global solution for \eqref{eq:halfwm-wave}. We show such a solution also solves the original \eqref{eq:halfwm}. 

\begin{theorem}
\label{thm:solution}
For a smooth initial data $u[0]:\R^{1+n} \rightarrow S^2 \times TS^2$ as in the Theorem~\ref{thm:s2}. For the global smooth solution $u: \R^{1+n} \rightarrow S^2 $ of \eqref{eq:halfwm-wave} obtained from  Proposition \ref{prop:waveform}. We know that $u$ also solves \eqref{eq:halfwm}. 
\end{theorem}

\begin{proof}

In order to show the solution obtains by Proposition \ref{prop:waveform} also solves \eqref{eq:halfwm}, we introduce
\[
X: = u_t - u\times (-\triangle)^{\frac12}u,
\]

and the energy-type functional 
\[
\tilde{E}(t): = \frac12\int_{\R^n}\big|(-\triangle)^{\frac{n-3}{4}}X(t,\cdot)\big|^2\,dx
\]

as in  \cite{krieger2017small} and \cite{2019arXiv190412709K}. \par
Firstly, since the initial data $u[0]\in C_0^\infty$, by a standard energy estimate of the wave equation for \eqref{eq:halfwm-wave}, we know that
$\nabla_{t,x} u(t,\cdot) \in H^{\frac{n-3}{2}}$, $\forall t \in \R$. Therefore $\tilde{E}(t)$ is well defined. 

Because the solution $u$ is smooth, thus the energy-type functional is differentiable with respect to $t$. We obtain

\begin{align*}
    \partial_tX = -X\times  (-\triangle)^{\frac12}u  - u\times(-\triangle)^{\frac12}X - u\big(X\cdot(u\times  (-\triangle)^{\frac12}u + u_t)\big),
    \end{align*}
and 
\begin{equation}
    \label{eq:technical10}
    \begin{split}
\frac{d}{dt}\tilde{E}(t) =& -\int_{\R^n}(-\triangle)^{\frac{n-3}{4}}\big(X\times  (-\triangle)^{\frac12}u +  u\times(-\triangle)^{\frac12}X\big)\cdot (-\triangle)^{\frac{n-3}{4}}X\,dx\\
& - \int_{\R^n}(-\triangle)^{\frac{n-3}{4}}\big(u\big(X\cdot(u\times  (-\triangle)^{\frac12}u + u_t)\big)\big)\cdot(-\triangle)^{\frac{n-3}{4}}X\,dx.\\
\end{split}
\end{equation}


By the multilinear estimate of Lemma~\ref{lem:multilinestimates} and Sobolev inequality $\dot{H}^{\frac{n-3}{2}} \hookrightarrow L^{\frac{2n}{3}}$, we have
\begin{align*}
    &\big\|(-\triangle)^{\frac{n-3}{4}}\big(X\times  (-\triangle)^{\frac12}u +  u\times(-\triangle)^{\frac12}X\big) - u\times(-\triangle)^{\frac{n-1}{4} }X \big\|_{L_x^2}\\
    &\lesssim \big\|(-\triangle)^{\frac{n-3}{4}}X\big\|_{L_x^2} \big\| (-\triangle)^{\frac12}u\big\|_{L_x^\infty} + \big\|X\big\|_{L_x^{\frac{2n}{3}}} \big\|(-\triangle)^{\frac{n-1}{4} }u\big\|_{L_x^{\frac{2n}{n-3}}} \\
     & + \big\|(-\triangle)^{\frac{n-3}{4}}u\big\|_{L_x^{\frac{2n}{n-5}}} \big\|(-\triangle)^{\frac12}X\big\|_{L_x^{\frac{2n}{5}}}\\
    &\lesssim C(\|u\|) \  \big\|(-\triangle)^{\frac{n-3}{4}}X\big\|_{L_x^2}^2.
    \end{align*} 
    where $C(\|u\|)$ is a constant depending on the various Strichartz norms of $u$.
   
    Similarly, we also have
    \begin{align*}
        \big\|u\times(-\triangle)^{\frac{n-1}{4}}X  - (-\triangle)^{\frac14}\big(u\times(-\triangle)^{\frac{n-2}{4}}X\big)\big\|_{L_x^2}
         \lesssim \|\nabla_x u\|_{L_x^\infty}\|(-\triangle)^{\frac{n-3}{4}}X\|_{L_x^2}.
    \end{align*}

On account of the half-wave equation \eqref{eq:halfwm}, we have 
$$
\int_{\R^n}(-\triangle)^{\frac14}\big(u\times(-\triangle)^{\frac{n-2}{4}}X\big)\cdot (-\triangle)^{\frac{n-3}{4}}X\ dx =0.
$$

 Hence, for the first term of \eqref{eq:technical10}, we have
    \begin{align*}
        &\int_{\R^n}(-\triangle)^{\frac{n-3}{4}}\big(X\times  (-\triangle)^{\frac12}u +  u\times(-\triangle)^{\frac12}X\big)\cdot (-\triangle)^{\frac{n-3}{4}}X\,dx \\
        &\lesssim \int_{\R^n}\left((-\triangle)^{\frac{n-3}{4}}\big(X\times  (-\triangle)^{\frac12}u +  u\times(-\triangle)^{\frac12}X \big) -u\times(-\triangle)^{\frac{n-1}{4}}X \right)\cdot (-\triangle)^{\frac{n-3}{4}}X\,dx \\
        & +\|\nabla_x u\|_{L_x^\infty}\|(-\triangle)^{\frac{n-3}{4}}X\|^2_{L_x^2} \\
        &\lesssim C(\|u\|) \|(-\triangle)^{\frac{n-3}{4}}X\|^2_{L_x^2}
        \end{align*}

        For the second term of \eqref{eq:technical10}, we estimate it similarly. We deduce that 
        \begin{align*}
            &\big\|(-\triangle)^{\frac{n-3}{4}}\big(u\big(X\cdot(u\times  (-\triangle)^{\frac12}u + u_t)\big)\big\|_{L_x^2}\\
            &\lesssim \big\|(-\triangle)^{\frac{n-3}{4}} X \big\|_{L_x^2} \ \|\Lah u\|_{L_x^2} \ \|u\|_{L_x^\infty} + \big\|(-\triangle)^{\frac{n-3}{4}} X \big\|_{L_x^2} \|\partial_t u\|_{L_x^\infty}\\
            &\lesssim C(\|u\|) \ \big\|(-\triangle)^{\frac{n-3}{4}} X \big\|_{L_x^2}.
        \end{align*}

        We conclude the second term of  \eqref{eq:technical10} by 
        \begin{align*}
        &\big| \int_{\R^n}(-\triangle)^{\frac{n}{4} - \frac{3}{4}}\big(u\big(X\cdot(u\times  (-\triangle)^{\frac12}u + u_t)\big)\big)\cdot(-\triangle)^{\frac{n}{4} - \frac{3}{4}}X\,dx\big|\\
        &\lesssim C(\|u\|)\big\|(-\triangle)^{\frac{n-3}{4}}X\big\|_{L_x^2}^2.
        \end{align*}

        Combining these two estimates, we know \eqref{eq:technical10} satisfies
        \[
        \frac{d}{dt}\tilde{E}(t)\leq C(u)\ \tilde{E}(t).
        \]

        Since $\tilde{E}(0) = 0$, we have $\tilde{E}(t) = 0$ for all $t$. It further implies that $X \equiv 0$. So $u$ solves the original half-wave equation \eqref{eq:halfwm}. Hence we proved the Theorem \ref{thm:s2}.

    \end{proof}

\section{Global Well-Posedness With $\Hy$ target}

 We now extend our result to the $\Hy$ target. Comparing the $S^2$ target discussed above, the obstacles here are that $\Hy$ is a non-compact manifold in $\R^3$ and its Lorentzian geometry structure. Instead of consider $\dot{H}^{\frac{n}{2}}$ initial data, we consider the critical $\besov$ initial data which embeds into $L^\infty$. Thus the solution stays in a compact subset of $\Hy$. We introduce the following theorem for the hyperbolic Lorentzian geometry to reduce the geometric constraint of $\Hy$.

\begin{theorem}[Theorem 2.61 in \cite{BahouriHajer2011Faan}]
    \label{thm:funcaction}
    Let $f$ be a smooth function on $\R$, which vanishes at 0. For a Besov space $\dot{B}^s_{p,q}$
    where $s>0$ and $p,q \in [1,\infty]$ satisfies
    $$ s < \frac{n}{p}, \ \text{or} \ s=\frac{n}{p}\ \text{and}\ q=1.$$
    For any real-valued function $u \in \dot{B}^{s}_{p,q} \cap L^{\infty}$, the function $f\circ u$ belongs to the same space i.e.
    \begin{equation}
        \|f \circ u\|_{\dot{B}^{s}_{p,q}} \le C(f', \|u\|_{L^{\infty}}) \ \|u\|_{\dot{B}^{s}_{p,q}}
    \end{equation}
\end{theorem}

For the distance function of $\Hy$
$$d_H(p,q)= \text{arccosh}(-p \cdot_{\eta} q) = \ln (-p \cdot_{\eta} q+\sqrt{(p \cdot_{\eta} q)^2-1}),$$
we can smoothly extend $\text{arccosh}(x):= \ln (x + \sqrt{x^2-1})$ defined on $x\in [1,\infty)$ to a function vanish at 0 in order to meet the criteria of Theorem \ref{thm:funcaction}. 
Therefore, we know that for $-u \cdot_\eta Q \in \besov$, we have $d_H(u,Q) \in \besov$. As $Q$ is a priori fixed point, by the triangle inequality, it suffices to show
$$u^i\in \besov \ \text{for} i=0,1,2$$

Note here, we view $u=(u^1,u^2,u^3)$ as a $\R^3$ vector without the hyperbolic plane constraint.  \par
With this relaxation from the hyperbolic constraint, the proof Theorem \ref{thm:h2} is similar to the works of \cite{krieger2017small} and \cite{2019arXiv190412709K} and also the $S^2$ case above.

Firstly, we reformulate \eqref{eq:halfwm} into 
\begin{align}
    \label{eq:waveh2}
    \begin{split}
    \Box u = (\partial^2_t - \Delta) u =
    &u \times_{\eta} \Lah \left(u \times_{\eta} \Lah u \right) - u \times_{\eta} \left(u \times_{\eta} (-\Delta) u \right) -\left(\nabla u \cdot_{\eta} \nabla u\right) u   \\
    =&(\partial_t u \cdot_{\eta} \partial_t u - \nabla u \cdot_{\eta} \nabla u) u  \\
    &+\Pi_{u_{\perp}}\left[\ (u \cdot_{\eta} \Lah u )\ \Lah u \right]   \\
    &+\Pi_{u_{\perp}} \left[u \times_{\eta} \Lah \left(u \times_{\eta} \Lah u \right) - u \times_{\eta} \left(u \times_{\eta} (-\Delta) u \right)\right] 
    \end{split}  
\end{align}
as in Section \ref{sec:transform}. Then we have a multilinear estimate analogous to \cite{2019arXiv190412709K}.
\begin{Proposition}
    \label{prop:multilinear_h2}
     For $u$ as in Theorem \ref{thm:h2}, and some $\sigma>0$, we have the bounds 
    \begin{equation}
        \label{eq:multilin1}
    \big\|P_k\big[u( \partial_t u \cdot_\eta \partial_t u -\nabla u \cdot_\eta \nabla u )\big]\big\|_{N}\lesssim (1+\big\|u\big\|_{S})\big\|u\big\|_{S}\big(\sum_{k_1\in Z}2^{-\sigma|k-k_1|}\big\|P_{k_1}u\big\|_{S_{k_1}}\big)
    \end{equation}
    Furthermore, for $\tilde{u}$ maps into a small neighborhood of $\Hy$, we also have 
    \begin{equation}\label{eq:multilin2}
    \big\|P_k\big(\Pi_{\tilde{u}}\big((-\triangle)^{\frac{1}{2}}u\big)(u\cdot_\eta (-\triangle)^{\frac{1}{2}}u)\big)\big\|_{N}\lesssim \prod_{v = u,\tilde{u}} (1+\big\|v\big\|_{S})\big\|u\big\|_{S}\big(\sum_{k_1\in Z}2^{-\sigma|k-k_1|}\big\|P_{k_1}u\big\|_{S_{k_1}}\big)
    \end{equation}
    and
    \begin{equation}\label{eq:multilin3}\begin{split}
    &\big\|P_k\big(\Pi_{\tilde{u}}\big[u \times_\eta (-\triangle)^{\frac{1}{2}}(u \times_\eta (-\triangle)^{\frac{1}{2}}u) - u \times_\eta(u \times_\eta (-\triangle)u)\big]\big)\big\|_{N}\\
    &\lesssim \prod_{v = u,\tilde{u}} (1+\big\|v\big\|_{S})\big\|u\big\|_{S}\big(\sum_{k_1\in Z}2^{-\sigma|k-k_1|}\big\|P_{k_1}u\big\|_{S_{k_1}}\big).
    \end{split}\end{equation}
    Using the notation 
    \[
    \triangle_{1,2}F^{(j)} := F^{(1)} - F^{(2)},
    \]
    we have the difference estimates: 
    \begin{equation}\label{eq:multilin4}\begin{split}
    &\big\|\triangle_{1,2}P_k\big[u^{(j)}(\nabla u^{(j)}\cdot_\eta\nabla u^{(j)} - \partial_t u^{(j)}\cdot_\eta\partial_t u^{(j)})\big]\big\|_{N}\\&\lesssim (1+\max_j\big\|u^{(j)}\big\|_{S})(\max_j\big\|u^{(j)}\big\|_{S})\big(\sum_{k_1\in Z}2^{-\sigma|k-k_1|}\big\|P_{k_1}u^{(1)} - P_ku^{(2)}\big\|_{S_{k_1}}\big)\\
    & + (1+\max_j\big\|u^{(j)}\big\|_{S})(\big\|u^{(1)} - u^{(2)}\big\|_{S})\big(\max_j\sum_{k_1\in Z}2^{-\sigma|k-k_1|}\big\|P_{k_1}u^{(j)}\big\|_{S_{k_1}}\big),\\
    \end{split}\end{equation}
    and the two similarly estimates
    \begin{equation}\label{eq:multilin5}\begin{split}
    &\big\|P_k\triangle_{1,2}\big(\Pi_{\tilde{u}^{(j)}_{}}\big((-\triangle)^{\frac{1}{2}}u^{(j)}\big)(u^{(j)}\cdot_\eta (-\triangle)^{\frac{1}{2}}u^{(j)})\big)\big\|_{N}\\
    &\lesssim \max_j\prod_{v = u^{(j)},\tilde{u}^{(j)}} (1+\big\|v\big\|_{S})\big\|u^{(j)}\big\|_{S}\big(\sum_{k_1\in Z}2^{-\sigma|k-k_1|}\big\|P_{k_1}u^{(1)} - P_{k_2}u^{(2)} \big\|_{S_{k_1}}\big)\\
    & + \max_j\prod_{v = u^{(j)},\tilde{u}^{(j)}} (1+\big\|v\big\|_{S})\big\|u^{(1)}- u^{(2)}\big\|_{S}\big(\max_j\sum_{k_1\in Z}2^{-\sigma|k-k_1|}\big\|P_{k_1}u^{(j)}\big\|_{S_{k_1}}\big)\\
    & +  \max_j(1+\big\|u^{(j)}\big\|_{S})\big\|\tilde{u}^{(1)}- \tilde{u}^{(2)}\big\|_{S}\big(\max_j\sum_{k_1\in Z}2^{-\sigma|k-k_1|}\big\|P_{k_1}u^{(j)}\big\|_{S_{k_1}}\big),\\
    \end{split}
\end{equation}

 \begin{equation}
        \label{eq:multilin6}
        \begin{split}
        &\big\|P_k \triangle_{1,2} \big(\Pi_{\tilde{u}^{(j)}}\big[u^{(j)} \times_\eta (-\triangle)^{\frac{1}{2}}(u^{(j)} \times_\eta (-\triangle)^{\frac{1}{2}}u^{(j)}) - u ^{(j)}\times_\eta(u^{(j)} \times_\eta (-\triangle)u^{(j)})\big]\big)\big\|_{N}\\
        &\lesssim \max_j\prod_{v = u^{(j)},\tilde{u}^{(j)}} (1+\big\|v\big\|_{S})\big\|u^{(j)}\big\|_{S}\big(\sum_{k_1\in Z}2^{-\sigma|k-k_1|}\big\|P_{k_1}u^{(1)} - P_{k_2}u^{(2)} \big\|_{S_{k_1}}\big)\\
    & + \max_j\prod_{v = u^{(j)},\tilde{u}^{(j)}} (1+\big\|v\big\|_{S})\big\|u^{(1)}- u^{(2)}\big\|_{S}\big(\max_j\sum_{k_1\in Z}2^{-\sigma|k-k_1|}\big\|P_{k_1}u^{(j)}\big\|_{S_{k_1}}\big)\\
    & +  \max_j(1+\big\|u^{(j)}\big\|_{S})\big\|\tilde{u}^{(1)}- \tilde{u}^{(2)}\big\|_{S}\big(\max_j\sum_{k_1\in Z}2^{-\sigma|k-k_1|}\big\|P_{k_1}u^{(j)}\big\|_{S_{k_1}}\big).\\
        \end{split}
    \end{equation}

    \end{Proposition}
The proof of this Proposition is analogous to the $S^2$ case in Kiesenhofer-Krieger\cite{2019arXiv190412709K} and the proof in Section \ref{sec:error} above, so we omit the proof here. \par
With the multilinear estimate, we can then simply follow the iteration arguments to show there is a smooth solution to the \eqref{eq:waveh2} and then prove the solution solve the original \eqref{eq:halfwm} as in Section \ref{sec:sol}  to conclude Theorem \ref{thm:h2}.

 \clearpage
\appendix
\section{Proof of Proposition \ref{prop:locmultilinear_s2}}
\label{app:proof}
In this section, without further specification, all the Strichartz norms are defined on $[0,T]\times \R^n$. \par

{\it{Proof of \eqref{eq:locmultilin1_s2}}}:
\par The first term is the standard wave equation term. We refer the reader to the 
    works by Klainerman-Machedon \cite{klainerman1996smoothing}, \cite{klainerman1996estimates}, Klainerman-Selberg \cite{klainerman1997remark} and Tataru \cite{Tat3} \cite{tataru2005rough} for more details.
    For the completeness of this work, we also include the estimate here. We use the Littlewood-Paley theory to localize the derivative terms to frequency $2^{k_1}$ and $2^{k_2}$, and consider the different cases separately. 
\par

{\it{(1): high high interactions $\max\{k_1, k_2\}>k+10$.}}
We further split this case relative to the difference between $k_1,k_2$.
\begin{equation*}
    \begin{split}
    &P_k[u\nabla_{t,x}u_{k_1}\nabla_{t,x}u_{k_2}] \\
    &=\sum_{\substack{|k_1-k_2| \leq 10 \\ k_1>k+10}}P_k[u\nabla_{t,x}u_{k_1}\nabla_{t,x}u_{k_2}]
     +\sum_{\substack{k_2 > k_1+10 \\ k_2>k+10}}P_k[u\nabla_{t,x}u_{k_1}\nabla_{t,x}u_{k_2}] \\
    &\quad +\sum_{\substack{k_1 > k_2+10 \\ k_1>k+10}}P_k[u\nabla_{t,x}u_{k_1}\nabla_{t,x}u_{k_2}].
    \end{split}
\end{equation*}

For $\gamma>0$, by the H\"older inequality, we have 
$$\|u\|_{L_t^2([0,T])} \leq T^{\frac{1}{2}-\frac{1}{\gamma}} \|u\|_{L_t^{2+\gamma}([0,T])}.$$

 For a wave admissible pair $(p,q)$, $(p+\gamma,q)$ is wave admissable as well. We can further assume that $s-\frac{n}{2} \leq \frac{\gamma}{4(2+\gamma)}$ which will insure the sum over the frequencies of $k_1, k_2$ converge to a constant. Hence we get
\begin{align*}
&2^{k(s-1)} \big\|\sum_{\substack{k_1 = k_2+O(1) \\ k_1>k+10}} P_k[u\nabla_{t,x}u_{k_1}\nabla_{t,x}u_{k_2}]\big\|_{L_t^1 L_x^2} \\
&\lesssim 2^{k(s-1)} T^{\frac{\gamma}{2(2+\gamma)}} \big\|\nabla_{t,x}u_{k_1}\big\|_{L_t^{2+\gamma} L_x^4}\big\|\nabla_{t,x}u_{k_2}\big\|_{L_t^2 L_x^4}\\
&\lesssim T^{\frac{\gamma}{2(2+\gamma)}} \sum_{k_1>k+10} 2^{(\frac{n}{2}-s+\frac{\gamma}{2(2+\gamma)})k_1} 2^{(s-1)(k-k_1)} \|u_{k_1}\|_{S_{loc}}^2 \\
&\lesssim T^{\frac{\gamma}{2(2+\gamma)}} \|u\|_{S_{loc}}^2.
\end{align*}

For $k_2>k_1+10$, we have

\begin{align*}
&2^{k(s-1)} \big\|\sum_{\substack{k_2 > k_1+10 \\ k_2>k+10}}P_k[u\nabla_{t,x}u_{k_1}\nabla_{t,x}u_{k_2}]\big\|_{L_t^1 L_x^2} \\
 &\simeq 2^{k(s-1)}\big\|\sum_{\substack{k_2 > k_1+10 \\ k_2>k+10}}P_k[u_{k_2} \nabla_{t,x}u_{k_1}\nabla_{t,x}u_{k_2}]\big\|_{L_t^1 L_x^2}\\
&\lesssim 2^{k(s-1)} \sum_{\substack{k_2 > k_1+10 \\ k_2>k+10}} \big\|u_{k_2}\big\|_{L_t^{2+\gamma} L_x^4}\big\|\nabla_{t,x}u_{k_2}\big\|_{L_t^2 L_x^4}\big\|\nabla_{t,x}u_{k_1}\big\|_{L_t^\infty L_x^2+L_{t,x}^\infty}\\
&\lesssim  T^{\frac{\gamma}{2(2+\gamma)}} \|u\|_{S_{loc}} \sum_{k_2>k+10} 2^{(k-k_2)(s-1)} 2^{(-\frac{4+\gamma}{2(2+\gamma)}+\frac{n}{2}-s)k_2} \|u_{k_2}\|_{S_{loc}}^2 \\
&\lesssim T^{\frac{\gamma}{2(2+\gamma)}} \|u\|_{S_{loc}}^3.
\end{align*}

The case $k_1>k_2+10$ is symmetrical, we omit here.  
\\

{\it{(2): high low interactions $\max\{k_1, k_2\}<k-10$.}} 
This is similar to the estimate we showed above. Firstly, we localize the frequency of the first factor $u$ to $\sim 2^k$, and estimate it by the $L_t^\infty L_x^2$ norm. Then we put the rest two terms $\nabla_{t,x}u_{k_j}$, $j = 1,2$ into $L_t^2 L_x^\infty$ and  $L_t^{2+\gamma} L_x^\infty$ respectably to conclude the result. For instance, 
\begin{align*}
    &2^{k(s-1)}\big\|\sum_{\substack{k_1 < k-10 \\ |k_1-k_2|<10}}P_k[u_{k} \nabla_{t,x}u_{k_1}\nabla_{t,x}u_{k_2}]\big\|_{L_t^1 L_x^2}\\
    &\lesssim 2^{k(s-1)} \| u_k\|_{L_t^\infty L_x^2} \sum_{\substack{k_1 < k-10 \\ |k_1-k_2|<10}} \|\nabla u_{k_1}\|_{L_t^2 L_x^\infty} \| \nabla u_{k_2}\|_{L_t^{2+\gamma} L_x^\infty} \\
    &\lesssim T^{\frac{\gamma}{2(2+\gamma)}} \|u\|_{S_{loc}} \sum_{k_1 < k-10} 2^{k_1-k} 2^{(\frac{\gamma}{2(2+\gamma)}+n-2s)k_1} \|u_{k_1}\|_{S_{loc}}^2 \\
    &\lesssim  T^{\frac{\gamma}{2(2+\gamma)}} \|u\|_{S_{loc}}^3.
\end{align*}

The rest cases follow similarly. \\

{\it{(3): low high interactions $\max\{k_1, k_2\}\in [k-10,k+10]$.}} We only need to consider the case when $k_1<k_2-10$, other cases will be similar to {\it{(1)}}. We separate it into
\begin{align}
    &2^{k(s-1)} \big\|\sum_{\substack{k_2 \in [k-10,k+10] \\ k_1<k_2-10}}  P_k[u\nabla_{t,x}u_{k_1}\nabla_{t,x}u_{k_2}]\big\|_{L_t^1 L_x^2}  \nonumber \\
    &=2^{k(s-1)} \big\|\sum_{\substack{k_2 \in [k-10,k+10] \\ k_1<k_2-10}}  P_k[u_{\geq k_1-10}\nabla_{t,x}u_{k_1}\nabla_{t,x}u_{k_2}]\big\|_{L_t^1 L_x^2}  \label{eq:loc1.3.1}\\
    &+ 2^{k(s-1)} \big\|\sum_{\substack{k_2 \in [k-10,k+10] \\ k_1<k_2-10}}  P_k[u_{< k_1-10}\nabla_{t,x}u_{k_1}\nabla_{t,x}u_{k_2}]\big\|_{L_t^1 L_x^2} \label{eq:loc1.3.2}
\end{align}

For \eqref{eq:loc1.3.1}, we can easily derive
\begin{align*}
&2^{k(s-1)} \big\|\sum_{\substack{k_2 \in [k-10,k+10] \\ k_1<k_2-10}}  P_k[u_{\geq k_1-10}\nabla_{t,x}u_{k_1}\nabla_{t,x}u_{k_2}]\big\|_{L_t^1 L_x^2} \\
&\lesssim 2^{k(s-1)} T^{\frac{\gamma}{2(2+\gamma)}} \sum_{\substack{k_2 \in [k-10,k+10] \\ k_1<k_2-10}} \big\|u_{\geq k_1-10}\big\|_{L_t^2 L_x^\infty}\big\|\nabla_{t,x}u_{k_1}\big\|_{L_t^{2+\gamma} L_x^\infty}\big\|\nabla_{t,x}u_{k_2}\big\|_{L_t^\infty L_x^2}\\
&\lesssim T^{\frac{\gamma}{2(2+\gamma)}} \|u_k\|_{S_{loc}} \sum_{\substack{k_3 >k_1-10 \\ k_1<k_2-10}} 2^{\frac{k_1-k_3}{2}} 2^{\frac{\gamma}{2(2+\gamma)}k_1} 2^{(\frac{n}{2}-s)(k_1+k_3)} \|u_{k_1}\|_{S_{loc}} \|u_{k_3}\|_{S_{loc}} \\
&\lesssim T^{\frac{\gamma}{2(2+\gamma)}} \big\|u\big\|^3_{S_{loc}}.
\end{align*}

For the second term \eqref{eq:loc1.3.2}, we need to consider the hyperbolic Sobolev space norm.
We divide the case by the time-space projector $Q_j$ as 
\begin{align}
    &P_k[u_{< k_1-10}\nabla_{t,x}u_{k_1}\nabla_{t,x}u_{k_2}] \nonumber \\
    &=  P_k[Q_{\geq k_1-10}u_{<k_1-10}\nabla_{t,x}u_{k_1}\nabla_{t,x}u_{k_2}] \label{eq:q1.2.1}\\
    &+P_k[Q_{<k_1-10}u_{<k_1-10}\partial_{\alpha}u_{k_1}\partial^{\alpha}u_{k_2}]. \label{eq:q1.2.2}
\end{align}

We estimate the first term \eqref{eq:q1.2.1} as above. Then we left the most delicate term \eqref{eq:q1.2.2}, which we need to utilize the wave equation's null structure. Here we further split into
\begin{align}
    &P_k[Q_{<k_1-10}u_{<k_1-10}\partial_{\alpha}u_{k_1}\partial^{\alpha}u_{k_2}] \nonumber \\
    &=P_k[Q_{<k_1-10}u_{<k_1-10}\partial_{\alpha}u_{k_1}\partial^{\alpha}Q_{>k_1-10}u_{k_2}]  \label{eq:null1.2.1}\\
    &+P_k[Q_{<k_1-10}u_{<k_1-10}\partial_{\alpha}Q_{\geq k_1+10}u_{k_1}\partial^{\alpha}Q_{<k_1-10}u_{k_2}]\label{eq:null1.2.2}\\
    &+P_k[Q_{<k_1-10}u_{<k_1-10}\partial_{\alpha}Q_{<k_1+10}u_{k_1}\partial^{\alpha}Q_{<k_1-10}u_{k_2}]\label{eq:null1.2.3}.
\end{align}
    
The first term \eqref{eq:null1.2.1} can be treated as usual,
\begin{align*}
&2^{k(s-1)}\big\|\sum_{\substack{k_2 \in [k-10,k+10] \\ k_1<k_2-10}}P_k[Q_{<k_1-10}u_{<k_1-10}\partial_{\alpha}u_{k_1}\partial^{\alpha}Q_{>k_1-10}u_{k_2}]\big\|_{L_t^1 L_x^2} \\
&\lesssim 2^{k(s-1)} \sum_{\substack{k_2 \in [k-10,k+10] \\ k_1<k_2-10}}\big\|\partial_{\alpha}u_{k_1}\big\|_{L_t^{2+\gamma} L_x^\infty}\big\|\partial^{\alpha}Q_{>k_1-10}u_{k_2}\big\|_{L_{t,x}^2}\\
&\lesssim T^{\frac{\gamma}{2(2+\gamma)}}\sum_{\substack{j >k_1-10 \\ k_1<k_2-10}} 2^{j(\frac{n}{2}-s-\frac{1}{2})} 2^{(1-\frac{1}{2+\gamma}+\frac{n}{2}-s)k_1} \|u_{k_1}\|_{S_{loc}} \|u_{k_2}\|_{S_{loc}} \\
&\lesssim T^{\frac{\gamma}{2(2+\gamma)}} \big\|u\big\|^2_{S_{loc}}.
\end{align*}

For the second term \eqref{eq:null1.2.2}, we use the $\dot{X}^{s-1,\theta-1}$ norm to estimate it. In particular, we rely on \eqref{eq:loc_xsb_es} to deal with the localization factor $\varphi_T$. Thus for some $\beta>0$,  we have

\begin{align*}
& \sum_{\substack{k_2 \in [k-10,k+10] \\ k_1<k_2-10}}\big\|\varphi_T P_k[Q_{<k_1-10}u_{<k_1-10}\partial_{\alpha}Q_{\geq k_1+10}u_{k_1}\partial^{\alpha}Q_{<k_1-10}u_{k_2}]\big\|_{\dot{X}^{s-1,\theta-1}}\\
&\lesssim T^{\frac{1}{2}-\theta} \ \sum_{\substack{k_2 \in [k-10,k+10] \\ k_1<k_2-10}}\big\| P_k[Q_{<k_1-10}u_{<k_1-10}\partial_{\alpha}Q_{\geq k_1+10}u_{k_1}\partial^{\alpha}Q_{<k_1-10}u_{k_2}]\big\|_{\dot{X}^{s-1,\theta-1}}\\
&\lesssim T^{\beta} 2^{k(s-1)} \sum_{\substack{k_2 \in [k-10,k+10] \\ k_1<k_2-10}} \sum_{j\geq k_1+10}2^{(s-\frac{n}{2}-\frac{1}{2})j}\big\|Q_j\nabla_{t,x}u_{k_1}\big\|_{L_t^{2+\gamma} L_x^\infty}\big\|\partial^{\alpha}Q_{<k_1-10}u_{k_2}\big\|_{L_t^\infty L_x^2}\\
&\lesssim T^{\beta} \|u\|_{S_{loc}}\sum_{\substack{k_1<k_2-10 \\j\geq k_1+10}} 2^{(\frac{\gamma}{2(2+\gamma)}+n-2s)k_1} \|u_{k_1}\|_{S_{loc}} \\
&\lesssim T^{\beta} \|u\|_{S_{loc}}^2.
\end{align*}

Finally, for \eqref{eq:null1.2.3}, we need to use the null structure to isolate difficult parts further. We write 
\begin{align}
&2P_k[Q_{<k_1-10}u_{<k_1-10}\partial_{\alpha}Q_{<k_1+10}u_{k_1}\partial^{\alpha}Q_{<k_1-10}u_{k_2}] \nonumber\\
&= P_k[Q_{<k_1-10}u_{<k_1-10}\Box\big(Q_{<k_1+10}u_{k_1}Q_{<k_1-10}u_{k_2}\big)] \label{eq:xb1}\\
& - P_k[Q_{<k_1-10}u_{<k_1-10}\Box Q_{<k_1+10}u_{k_1}Q_{<k_1-10}u_{k_2}]\label{eq:xb2}\\
& - P_k[Q_{<k_1-10}u_{<k_1-10}Q_{<k_1+10}u_{k_1}\Box Q_{<k_1-10}u_{k_2}]\label{eq:xb3}.
\end{align}

Then we consider each one of them individually. For the \eqref{eq:xb1}, we have 
\begin{align}
&P_k[Q_{<k_1-10}u_{<k_1-10}\Box\big(Q_{<k_1+10}u_{k_1}Q_{<k_1-10}u_{k_2}\big)] \nonumber \\
& = \Box P_k[Q_{<k_1-10}u_{<k_1-10}\big(Q_{<k_1+10}u_{k_1}Q_{<k_1-10}u_{k_2}\big)]
\label{eq:1.3q1}\\
& - P_k[\nabla_{t,x}Q_{<k_1-10}u_{<k_1-10}\nabla_{t,x}\big(Q_{<k_1+10}u_{k_1}Q_{<k_1-10}u_{k_2}\big)] \label{eq:1.3q2}\\
& - P_k[\nabla_{t,x}^2Q_{<k_1-10}u_{<k_1-10}\big(Q_{<k_1+10}u_{k_1}Q_{<k_1-10}u_{k_2}\big)] \label{eq:1.3q3}
\end{align}

We treat the most difficult \eqref{eq:1.3q1} as following

\begin{align*}
&\sum_{\substack{k_2 \in [k-10,k+10] \\ k_1<k_2-10}}\big\| \varphi_T  \Box P_k[Q_{<k_1-10}u_{<k_1-10}\big(Q_{<k_1+10}u_{k_1}Q_{<k_1-10}u_{k_2}\big)]\big\|_{\dot{X}^{s-1,\theta-1}}\\
&\lesssim T^{\frac{n}{2}-s} \sum_{\substack{k_2 \in [k-10,k+10] \\ k_1<k_2-10}}\big\| P_k[Q_{<k_1-10}u_{<k_1-10}\big(Q_{<k_1+10}u_{k_1}Q_{<k_1-10}u_{k_2}\big)]\big\|_{\dot{X}^{s,\theta}} \\
&\lesssim T^{\frac{n}{2}-s} 2^{sk} \sum_{\substack{k_2 \in [k-10,k+10] \\ k_1<k_2-10}}  \sum_{j<k_1+20} 2^{(s-\frac{n}{2}+\frac{1}{2})j} \big\|P_kQ_j[Q_{<k_1-10}u_{<k_1-10} \\
&\big(Q_{<k_1+10}u_{k_1}Q_{<k_1-10}u_{k_2}\big)]\big\|_{L_{t,x}^2}\\
&\lesssim 2^{sk} T^{\beta} \sum_{\substack{k_2 \in [k-10,k+10] \\ k_1<k_2-10}} \sum_{j<k_1+20}2^{(s-\frac{n}{2}+\frac{1}{2})j} \big\|Q_{<k_1+10}u_{k_1}\big\|_{L_t^{2+\gamma} L_x^\infty} \big\|Q_{<k_1-10}u_{k_2}\big\|_{L_t^\infty L_x^2}\\
&\lesssim  T^{\beta} \|u\|_{S_{loc}} \sum_{ \substack{k_1<k_2-10 \\ j<k_1+20}} 2^{(s-\frac{n}{2}+\frac{1}{2})j} 2^{(-\frac{1}{2+\gamma} +\frac{n}{2}-s)k_1 }\big\|u_{k_1}\big\|_{S_{loc}} \\ 
&\lesssim T^{\beta} \big\|u\big\|^2_{S_{loc}}. 
\end{align*}

The rest two terms \eqref{eq:1.3q2} and \eqref{eq:1.3q3} can be estimated by the standard Strichartz norm as before. \par
Next, we move to the \eqref{eq:xb2}. We have
\begin{align*}
&2^{k(s-1)} \sum_{\substack{k_2 \in [k-10,k+10] \\ k_1<k_2-10}} \big\|P_k[Q_{<k_1-10}u_{<k_1-10}\Box Q_{<k_1+10}u_{k_1}Q_{<k_1-10}u_{k_2}]\big\|_{L_t^1 L_x^2}\\
&\lesssim 2^{k(s-1)} \sum_{\substack{k_2 \in [k-10,k+10] \\ k_1<k_2-10}}\big\|\Box Q_{<k_1+10}u_{k_1}\big\|_{L_t^{2+\gamma} L_x^4}\big\|Q_{<k_1-10}u_{k_2}\big\|_{L_t^2 L_x^4}\\
&\lesssim T^{\frac{\gamma}{2(2+\gamma)}} \|u\|_{S_{loc}} \sum_{k_1<k-10} 2^{\frac{3}{2}(k_1-k)} 2^{(\frac{\gamma}{2(2+\gamma)}+\frac{n}{2}-s)k_1} \|u_{k_1}\|_{S_{loc}} \\
&\lesssim T^{\frac{\gamma}{2(2+\gamma)}} \|u\|_{S_{loc}}^2
\end{align*}

Lastly, for \eqref{eq:xb3}, we get 
\begin{align*}
&2^{k(s-1)} \sum_{\substack{k_2 \in [k-10,k+10] \\ k_1<k_2-10}} \big\| P_k[Q_{<k_1-10}u_{<k_1-10}Q_{<k_1+10}u_{k_1}\Box Q_{<k_1-10}u_{k_2}]\big\|_{L_t^1 L_x^2}\\
&\lesssim 2^{k(s-1)} \sum_{\substack{k_2 \in [k-10,k+10] \\ k_1<k_2-10}} \big\|Q_{<k_1+10}u_{k_1}\big\|_{L_t^{2+\gamma} L_x^\infty}\big\|\Box Q_{<k_1-10}u_{k_2}\big\|_{L_{t,x}^2} \\
&\lesssim  T^{\frac{\gamma}{2(2+\gamma)}} \sum_{k_1<k-10} 2^{(-\frac{1}{2+\gamma}+\frac{n}{2}-s)k_1} \|u_{k_1}\| \sum_{j<k_1-10} 2^{(\frac{1}{2}+\frac{n}{2}-s)j} \|u_{k}\|_{S_{loc}} \\
&\lesssim T^{\frac{\gamma}{2(2+\gamma)}} \|u\|_{S_{loc}}.
\end{align*}

Hence we conclude the standard wave term. \par
{\it{Proof of \eqref{eq:locmultilin2_s2}}}: \par
 In the following, we will only use the Strichartz norm of the $N_{loc}$ space to estimate the nonlinearity terms. 

First, we introduce a lemma to deal with the projector $\Pi_{\tilde{u}^{\perp}}$.

\begin{Lemma}
    \label{lem:locmultilin2}
     For $u, \tilde{u}$ as in the Proposition \ref{prop:locmultilinear_s2}. We have the following:
     \begin{equation}\label{eq:locpiu}
        \big\|P_0\big(\Pi_{\tilde{u}^{\perp}}\big((-\triangle)^{\frac12}u\big)\big\|_{S_{loc}} \lesssim (1+\|\tilde{u}\|_{S_{loc}})\sum_{k_1\in \Z} 2^{\sigma |k_1|} \|u_{k_1}\|_{S_{loc}},
        \end{equation}
         and by re-scaling, we also have 
        \begin{equation}\label{eq:locpiu1}
        \big\|P_k\big(\Pi_{\tilde{u}^{\perp}}\big((-\triangle)^{\frac12}u\big)\big\|_{S_{loc}}\lesssim 2^k 2^{(\frac{n}{2}-s)k} (1+\|\tilde{u}\|_{S_{loc}})\sum_{k_1\in \Z} 2^{\sigma |k_1-k|} \|u_{k_1}\|_{S_{loc}}.
        \end{equation} 
        for a suitable $\sigma>0$. 
\end{Lemma}
\begin{proof}
  We first write
    \begin{align*}
        P_k\big(\Pi_{\tilde{u}^{\perp}}\big((\Lah u\big) = P_k\Lah u - P_k\big((\frac{\tilde{u}}{\|\tilde{u}\|}\cdot(\Lah u)\frac{\tilde{u}}{\|\tilde{u}\|}\big).
        \end{align*}

        For the term involve $\tilde{u}$, we split it into
        \begin{align*}
        P_k\big((\frac{\tilde{u}}{\|\tilde{u}\|}\cdot(\Lah u)\frac{\tilde{u}}{\|\tilde{u}\|}\big) &= P_k\big((\frac{\tilde{u}}{\|\tilde{u}\|}\cdot(\Lah u_{>k+10})\frac{\tilde{u}}{\|\tilde{u}\|}\big)\\
        & + P_k\big((\frac{\tilde{u}}{\|\tilde{u}\|}\cdot(\Lah u_{[k-10,k+10]})\frac{\tilde{u}}{\|\tilde{u}\|}\big)\\
        &+ P_k\big((\frac{\tilde{u}}{\|\tilde{u}\|}\cdot(\Lah u_{<k-10})\frac{\tilde{u}}{\|\tilde{u}\|}\big)\\
        \end{align*}
    
    Since $\tilde{u}$ maps into a small neighborhood of $S^2$ , we have $\big\|\tilde{u}\big\|_{L_{t,x}^\infty} \sim 1$. For $k_1>10$, we know that for $a=5,10$,
    \begin{equation*}
        \|\tilde{u}_{>k_1-a}\|_{L_t^{\infty}L_x^2} \lesssim 2^{-\frac{n}{2}k_1}.
    \end{equation*}

    Using the Bernstein inequality and Hölder's inequality, we have 
    \begin{align*}
    \big\|P_0\left([\frac{\tilde{u}}{\|\tilde{u}\|}]_{>k_1-10}\cdot\Lah u_{k_1}\right)\big\|_{L_t^2L_x^\infty} 
      &\lesssim \big\|P_0\left(\tilde{u}_{>k_1-10}\cdot\Lah u_{k_1}\right)\big\|_{L_t^2L_x^2} \\
      &\le  \ \left\| \tilde{u}_{>k_1-10} \right\|_{L_t^{\infty} L_x^2} \| \Lah u_{k_1} \|_{L_t^2 L_x^{\infty}}\\
      &\lesssim 2^{-\frac{nk_1}{2}}\| \Lah u_{k_1} \|_{L_t^2 L_x^{\infty}}.
    \end{align*}

      Thus we can estimate the first term as 
      \begin{align*}
        &\big\|P_0\big((\frac{\tilde{u}}{\|\tilde{u}\|}\cdot\Lah u_{>10})\frac{\tilde{u}}{\|\tilde{u}\|}\big)\big\|_{L_t^2L_x^\infty\cap L_t^\infty L_x^2} \nonumber \\
        &\leq \sum_{k_1>10}\big\|P_0\big(([\frac{\tilde{u}}{\|\tilde{u}\|}]_{>k_1-10}\cdot\Lah u_{k_1})\frac{\tilde{u}}{\|\tilde{u}\|}\big)\big\|_{L_t^2L_x^\infty\cap L_t^\infty L_x^2} \nonumber \\
        &+\sum_{k_1>10}\big\|P_0\big(([\frac{\tilde{u}}{\|\tilde{u}\|}]_{\leq k_1-10}\cdot\Lah u_{k_1})[\frac{\tilde{u}}{\|\tilde{u}\|}]_{>k_1-5}\big)\big\|_{L_t^2L_x^\infty\cap L_t^\infty L_x^2} \nonumber\\
        &\lesssim  \sum_{k_1>10}\sum_{a=5,10}\big\|[\frac{\tilde{u}}{\|\tilde{u}\|}]_{>k_1-a}\big\|_{L_t^\infty L_x^2}\big\|\Lah u_{k_1}\big\|_{L_t^2L_x^\infty\cap L_t^\infty L_x^2} \\
        &\lesssim \sum_{k_1>10}2^{-\frac{3k_1}{2}}\big\|u_{k_1}\big\|_{S_{loc}}
        \end{align*}    
    \footnotetext{We use the Cauchy-Schwartz inequality s.t. $\sum_{k_1>10}2^{(1-n)k_1} c_k \lesssim \|c\|_{\ell^2} \lesssim \epsilon$}

    Similarly, for the second term on the right, we have 
\begin{align*}
\big\|P_0\big((\frac{\tilde{u}}{\|\tilde{u}\|}\cdot(-\triangle)^{\frac12}u_{[-10,10]})\frac{\tilde{u}}{\|\tilde{u}\|}\big)\big\|_{L_t^2L_x^\infty\cap L_t^\infty L_x^2}&\lesssim \big\|(-\triangle)^{\frac12}u_{[-10,10]}\big\|_{L_t^2L_x^\infty\cap L_t^\infty L_x^2}\\
&\lesssim \sum_{k_1\in [-10,10]} 2^{\frac{|k_1|}{2}}\big\|u_{k_1}\big\|_{S_{loc}} 
\end{align*}

Finally, for the last term, we deduce
\begin{align*}
&\big\|P_0\big((\frac{\tilde{u}}{\|\tilde{u}\|}\cdot(-\triangle)^{\frac12}u_{<-10})\frac{\tilde{u}}{\|\tilde{u}\|}\big)\big\|_{L_t^2L_x^\infty\cap L_t^\infty L_x^2}\\
&\lesssim \big\|(-\triangle)^{\frac12}u_{<k-10}\big\|_{L_t^2 L_x^\infty\cap L_{t,x}^{\infty}}\big\|[\frac{\tilde{u}}{\|\tilde{u}\|}]_{>-10}\big\|_{L_t^\infty L_x^2}\\
&\lesssim  \sum_{k_1<-10}2^{-\frac{|k_1|}{2}}\big\|u_{k_1}\big\|_{S_{loc}} 
\end{align*}

Thus we can conclude that 
\begin{equation}
    \big\|P_0\big((\frac{\tilde{u}}{\|\tilde{u}\|}\cdot(\Lah u)\frac{\tilde{u}}{\|\tilde{u}\|}\big) \big\|_{L_t^2L_x^\infty\cap L_t^\infty L_x^2}  \lesssim C_0\epsilon
\end{equation}

Interpolating between $L_t^2 L_x^\infty$ and $L_t^\infty L_x^2$ norms within the wave admissible range, we obtain that 
\begin{equation*}
    \big\|P_0\big((\frac{\tilde{u}}{\|\tilde{u}\|}\cdot(\Lah u)\frac{\tilde{u}}{\|\tilde{u}\|}\big) \big\|_{S_{loc}}  \lesssim  (1+\|\tilde{u}\|_S)\sum_{k_1\in \Z} 2^{\sigma_1 |k_1|} \|u_{k_1}\|_S
\end{equation*}

 Scaling the frequency to $P_k$, we will have \eqref{eq:locpiu1}
\end{proof}

With Lemma \ref{lem:locmultilin2}, it's easy to see that
\begin{align*}
    &\big\|\Pi_{\tilde{u}^{\perp}}\big((-\triangle)^{\frac12}u \big)\big\|_{L_t^\infty L_x^2+ L_{t,x}^\infty}\lesssim \sum _{k_3\in Z}2^{-\sigma|k-k_3|}\big\|P_{k_3}u\big\|_{S_{loc}}(1+\big\|\tilde{u}\big\|_{S_{loc}}), \\
    &\big\|P_{[k-20,k+20]}\big(\Pi_{\tilde{u}_{\perp}}\big((-\triangle)^{\frac{1}{2}}u\big)\big)\big\|_{L_t^\infty L_x^2}\lesssim 2^{k(1-s)} \sum _{k_3\in Z}2^{-\sigma|k-k_3|}\big\|P_{k_3}u\big\|_{S_{loc}}(1+\big\|\tilde{u}\big\|_{S_{loc}}) ,\\
    &\big\|P_{<k+10}\Pi_{\tilde{u}_{\perp}}\big((-\triangle)^{\frac{1}{2}}u\big)\big\|_{L_t^2 L_x^\infty}\lesssim \sum _{k_3\in Z}2^{-\sigma|k-k_3|}\big\|P_{k_3}u\big\|_{S_{loc}}(1+\big\|\tilde{u}\big\|_{S_{loc}}).
\end{align*}

Now we start to estimate the second non-local term \eqref{eq:locmultilin2_s2}. Write 
\begin{align}
\Pi_{\tilde{u}_{\perp}}\big((-\triangle)^{\frac{1}{2}}u\big)(u\cdot (-\triangle)^{\frac{1}{2}}u) &= \sum_{|k_1-k_2|\leq 10}\Pi_{\tilde{u}_{\perp}}\big((-\triangle)^{\frac{1}{2}}u\big)(u_{k_1}\cdot (-\triangle)^{\frac{1}{2}}u_{k_2}) \label{local2.1}\\
&+ \sum_{k_1}\Pi_{\tilde{u}_{\perp}}\big((-\triangle)^{\frac{1}{2}}u\big)(u_{k_1}\cdot (-\triangle)^{\frac{1}{2}}u_{<k_1-10}) \label{local2.2}\\
&+\sum_{k_2}\Pi_{\tilde{u}_{\perp}}\big((-\triangle)^{\frac{1}{2}}u\big)(u_{<k_2-10}\cdot (-\triangle)^{\frac{1}{2}}u_{k_2}) \label{local2.3}
\end{align}

For the first term of \eqref{local2.1}, we have
\begin{align*}
&2^{k(s-1)} \big\|P_k\big[ \sum_{|k_1-k_2|\leq 10}\Pi_{\tilde{u}_{\perp}}\big((-\triangle)^{\frac{1}{2}}u\big)(u_{k_1}\cdot (-\triangle)^{\frac{1}{2}}u_{k_2})\big]\big\|_{L_t^1 L_x^2} \nonumber \\
&\lesssim 2^{k(s-1)} \sum_{\substack{|k_1-k_2|\leq 10\\ k_1<k-20}}\big\|P_{[k-20,k+20]}\big(\Pi_{\tilde{u}_{\perp}}\big((-\triangle)^{\frac{1}{2}}u\big)\big)\big\|_{L_t^\infty L_x^2}\big\|u_{k_1}\big\|_{L_t^2 L_x^\infty}\big\|(-\triangle)^{\frac{1}{2}}u_{k_2}\big\|_{L_t^2 L_x^\infty} \\
& +  2^{k(s-1)} \sum_{\substack{|k_1-k_2|\leq 10\\ k_1\geq k-20}}\big\|\Pi_{\tilde{u}_{\perp}}\big((-\triangle)^{\frac{1}{2}}u\big)\big\|_{L_t^\infty L_x^2+L_{t,x}^\infty}\big\|u_{k_1}\big\|_{L_t^2 L_x^4}\big\|(-\triangle)^{\frac{1}{2}}u_{k_2}\big\|_{L_t^2 L_x^4}. 
\end{align*}

For the first part, we deduce that
\begin{align*}
    &2^{k(s-1)}\sum_{\substack{|k_1-k_2|\leq 10\\ k_1<k-20}}\big\|P_{[k-20,k+20]}\big(\Pi_{\tilde{u}_{\perp}}\big((-\triangle)^{\frac{1}{2}}u\big)\big)\big\|_{L_t^\infty L_x^2}\big\|u_{k_1}\big\|_{L_t^2 L_x^\infty} \big\|(-\triangle)^{\frac{1}{2}}u_{k_2}\big\|_{L_t^2 L_x^\infty}\\
    &\lesssim T^{\frac{\gamma}{2(2+\gamma)}} \ \ \sum_{\substack{|k_1-k_2|\leq 10\\ k_1<k-20}}\big\|u_{k_1}\big\|_{L_t^{2+\gamma} L_x^\infty} \big\|(-\triangle)^{\frac{1}{2}}u_{k_2}\big\|_{L_t^2 L_x^\infty} \\
    &\quad \big( \sum _{k_3\in Z}2^{-|k-k_3|}\big\|P_{k_3}u\big\|_{S_{loc}}(1+\big\|\tilde{u}\big\|_{S_{loc}}) \big) \\
    &\lesssim T^{\frac{\gamma}{2(2+\gamma)}} \ \sum_{\substack{|k_1-k_2|\leq 10\\ k_1<k-20}} 2^{(\frac{n}{2}-s)(k_1+k_2)} \ 2^{\frac{k_2}{2}-\frac{k_1}{2+\gamma}} \ \big\|u_{k_1}\big\|_{S_{loc}} \big\|u_{k_2}\big\|_{S_{loc}} \\
    &\quad  \big( \sum _{k_3\in Z}2^{-|k-k_3|}\big\|P_{k_3}u\big\|_{S_{loc}}(1+\big\|\tilde{u}\big\|_{S_{loc}}) \big) \\
    &\lesssim T^{\frac{\gamma}{2(2+\gamma)}} \  \  \|u\|_{S_{loc}}^2 \ \big( \sum _{k_3\in Z}2^{-|k-k_3|}\big\|P_{k_3}u\big\|_{S_{loc}}(1+\big\|\tilde{u}\big\|_{S_{loc}}) \big)
    \end{align*}

The second part is easier to treat. We have
\begin{align*}
&  2^{k(s-1)} \sum_{\substack{|k_1-k_2|\leq 10\\ k_1\geq k-20}}\big\|\Pi_{\tilde{u}_{\perp}}\big((-\triangle)^{\frac{1}{2}}u\big)\big\|_{L_t^\infty L_x^2+L_{t,x}^\infty}\big\|u_{k_1}\big\|_{L_t^2 L_x^4}\big\|(-\triangle)^{\frac{1}{2}}u_{k_2}\big\|_{L_t^2 L_x^4}\\
&\lesssim  2^{k(s-1)}T^{\frac{\gamma}{2(2+\gamma)}} \sum_{\substack{|k_1-k_2|\leq 10\\ k_1\geq k-20}}\big\|u_{k_1}\big\|_{L_t^{2+\gamma} L_x^4}\big\|(-\triangle)^{\frac{1}{2}}u_{k_2}\big\|_{L_t^2 L_x^4} \\
&\quad \big( \sum_{k_3\in Z}2^{-\sigma|k_3|}\big\|P_{k_3}u\big\|_{S_{loc}}(1+\big\|\tilde{u}\big\|_{S_{loc}}) \big)\\
&\lesssim 2^{k(s-1)} T^{\frac{\gamma}{2(2+\gamma)}} \big(\sum_{\substack{|k_1-k_2|\leq 10\\ k_1\geq k-20}} 2^{(n-2s)k_1} \ 2^{-(\frac{n}{2}-\frac{\gamma}{2(2+\gamma)})k_1} \ \big\|u_{k_1}\big\|_{S_{loc}}\big\|u_{k_2}\big\|_{S_{loc}}\big) \\
 &\quad \big( \sum_{k_3\in Z}2^{-\sigma|k_3|}\big\|P_{k_3}u\big\|_{S_{loc}}(1+\big\|\tilde{u}\big\|_{S_{loc}}) \big)\\
&\lesssim  T^{\frac{\gamma}{2(2+\gamma)}} \ \big(\sum _{k_3\in Z}2^{-\sigma|k_3|}\big\|P_{k_3}u\big\|_{S_{loc}}\big)\big\|u\big\|_{S_{loc}}^2(1+\big\|\tilde{u}\big\|_{S_{loc}}).
\end{align*}
This concludes the required bound for the first term \eqref{local2.1} on the right-hand side. \par

Now we move to estimate the second term \eqref{local2.2}. For a given $k\in \Z$, we split it as a sum of three terms:
\begin{align}
&\sum_{k_1}\Pi_{\tilde{u}_{\perp}}\big((-\triangle)^{\frac{1}{2}}u\big)(u_{k_1}\cdot (-\triangle)^{\frac{1}{2}}u_{<k_1-10}) \nonumber\\
&= \sum_{k_1\geq k+5}\Pi_{\tilde{u}_{\perp}}\big((-\triangle)^{\frac{1}{2}}u\big)(u_{k_1}\cdot (-\triangle)^{\frac{1}{2}}u_{<k_1-10}) \label{eq:locproj21}\\
& +  \sum_{k_1\in[k-5,k+5]}\Pi_{\tilde{u}_{\perp}}\big((-\triangle)^{\frac{1}{2}}u\big)(u_{k_1}\cdot (-\triangle)^{\frac{1}{2}}u_{<k_1-10}) \label{eq:locproj22}\\
& +  \sum_{k_1<k-5}\Pi_{\tilde{u}_{\perp}}\big((-\triangle)^{\frac{1}{2}}u\big)(u_{k_1}\cdot (-\triangle)^{\frac{1}{2}}u_{<k_1-10}) \label{eq:locproj23}
\end{align} 

For \eqref{eq:locproj21}, we have
\begin{align*}
&2^{k(s-1)} \big\|P_k\big( \sum_{k_1\geq k+5}\Pi_{\tilde{u}_{\perp}}\big((-\triangle)^{\frac{1}{2}}u\big)(u_{k_1}\cdot (-\triangle)^{\frac{1}{2}}u_{<k_1-10})\big)\big\|_{L_t^1 L_x^2}\\
&\lesssim  T^{\frac{\gamma}{2(2+\gamma)}} 2^{k(s-1)} \sum_{k_1\geq k+5} \Big(\big\|u_{k_1}\big\|_{L_t^2L_x^\infty} \big\|(-\triangle)^{\frac{1}{2}}u_{<k_1-10}\big\|_{L_t^{2+\gamma}  L_x^\infty} \\
& \big\|P_{[k_1-5,k_1+5]}\big(\Pi_{\tilde{u}_{\perp}}\big((-\triangle)^{\frac{1}{2}}u\big)\big)\big\|_{L_t^\infty L_x^2} \Big)\\
&\lesssim T^{\frac{\gamma}{2(2+\gamma)}}  2^{k(s-1)} \sum_{\substack{k_1\geq k+5 \\ k_2<k_1-10}}2^{(\frac{n+1}{2}-2s)k_1} \ 2^{(\frac{n}{2}-s)k_2} 2^{k_2-\frac{k_2}{2+\gamma}} \big\|u_{k_1}\big\|_{S_{loc}}\big\|u_{k_2}\big\|_{S_{loc}} 
\\ & \big( \sum _{k_3\in Z}2^{-\sigma|k_1-k_3|}\big\|P_{k_3}u\big\|_{S_{loc}}(1+\big\|\tilde{u}\big\|_{S_{loc}})\big)\\
&\lesssim T^{\frac{\gamma}{2(2+\gamma)}} \big\|u\big\|_{S_{loc}}^2 \big(\sum_{k_3}2^{-\sigma|k_3|}\big\|P_{k_3}u\big\|_{S_{loc}}\big)(1+\big\|\tilde{u}\big\|_{S_{loc}}). 
\end{align*}
Similarly, for \eqref{eq:locproj22} involves intermediate $k_1$, we have 

\begin{align*}
&2^{k(s-1)} \big\|P_k\big[ \sum_{k_1\in[k-5,k+5]}\Pi_{\tilde{u}_{\perp}}\big((-\triangle)^{\frac{1}{2}}u\big)(u_{k_1}\cdot (-\triangle)^{\frac{1}{2}}u_{<k_1-10})\big]\big\|_{L_t^1 L_x^2}\\
&\lesssim T^{\frac{\gamma}{2(2+\gamma)}}2^{k(s-1)} \sum_{k_1\in[k-5,k+5]}\big\|u_{k_1}\big\|_{L_t^\infty L_x^2}\big\|(-\triangle)^{\frac{1}{2}}u_{<k_1-10}\big\|_{L_t^{2+\gamma} L_x^\infty} \\
& \big\|P_{<k+10}\Pi_{\tilde{u}_{\perp}}\big((-\triangle)^{\frac{1}{2}}u\big)\big\|_{L_t^2 L_x^\infty} \\
&\lesssim T^{\frac{\gamma}{2(2+\gamma)}} \   \|u\|_{S_{loc}}^2  
\big(\sum _{k_3\in Z}2^{-|k-k_3|}\big\|P_{k_3}u\big\|_{S_{loc}}(1+\big\|\tilde{u}\big\|_{S_{loc}}) \big)
\end{align*}

Lastly, for \eqref{eq:locproj23} when $k_1$ small, we get
\begin{align*}
    &2^{k(s-1)} \big\|P_k\big( \sum_{k_1\leq k-5}\Pi_{\tilde{u}_{\perp}}\big((-\triangle)^{\frac{1}{2}}u\big)(u_{k_1}\cdot (-\triangle)^{\frac{1}{2}}u_{<k_1-10})\big)\big\|_{L_t^1 L_x^2}\\
    &\lesssim  T^{\frac{\gamma}{2(2+\gamma)}} 2^{k(s-1)} \sum_{k_1\leq k-5}\big\|u_{k_1}\big\|_{L_t^2 L_x^\infty} \big\|(-\triangle)^{\frac{1}{2}}u_{<k_1-10}\big\|_{L_t^{2+\gamma} L_x^\infty}\\
    &\big\|P_{[k-20,k+20]}\big(\Pi_{\tilde{u}_{\perp}}\big((-\triangle)^{\frac{1}{2}}u\big)\big)\big\|_{L_t^\infty L_x^2}\\
    &\lesssim T^{\frac{\gamma}{2(2+\gamma)}} \sum_{\substack{k_1\leq k-5 \\ k_2<k_1-10}}2^{(\frac{n}{2}-s)(k_1+k_2)} 2^{-\frac{k_1}{2}} 2^{k_2-\frac{k_2}{2+\gamma}} \big\|u_{k_1}\big\|_{S_{loc}}\big\|u_{k_2}\big\|_{S_{loc}} \\
    & \big( \sum _{k_3\in Z}2^{-\sigma|k-k_3|}\big\|P_{k_3}u\big\|_{S_{loc}}(1+\big\|\tilde{u}\big\|_{S_{loc}})\big)\\
    &\lesssim T^{\frac{\gamma}{2(2+\gamma)}} \big\|u\big\|_{S_{loc}}^2 \big(\sum_{k_3}2^{-\sigma|k-k_3|}\big\|P_{k_3}u\big\|_{S_{loc}}\big)(1+\big\|\tilde{u}\big\|_{S_{loc}}). 
    \end{align*}

Finally, the third term \eqref{local2.3} is the most delicate one. The derivative $(-\triangle)^{\frac12}$ falls on the higher -frequency term $u_{k_2}$. We need to use the Lemma~\ref{lem:multilinestimates} to deal with it.

Firstly, we know the difference 
\begin{align*}
\Pi_{\tilde{u}_{\perp}}\big((-\triangle)^{\frac{1}{2}}u\big)(u_{<k_2-10}\cdot (-\triangle)^{\frac{1}{2}}u_{k_2}) -\Pi_{\tilde{u}_{\perp}}\big((-\triangle)^{\frac{1}{2}}u\big)(-\triangle)^{\frac{1}{2}}(u_{<k_2-10}\cdot u_{k_2})
\end{align*}

can be estimated like \eqref{local2.2} above. Thus it is enough to consider $\Pi_{\tilde{u}_{\perp}}\big((-\triangle)^{\frac{1}{2}}u\big)(-\triangle)^{\frac{1}{2}}(u_{<k_2-10}\cdot u_{k_2})$ instead. 

Using the geometric condition the condition $u\cdot u = 1$, we have
\begin{equation}\label{eq:locorthomicro}
0 = u\cdot u - Q\cdot Q = \sum_{|k_1-k_2|\leq 10} u_{k_1}u_{k_2} + 2\sum_{k_1}u_{k_1}\cdot u_{<k_1-10}
\end{equation}

This implies that 
\begin{equation}\label{eq:loccancel}
0 =  \sum_{|k_1-k_2|<10} (-\triangle)^{\frac{1}{2}}\big(u_{k_1}u_{k_2}\big) + 2\sum_{k_1}(-\triangle)^{\frac{1}{2}}\big(u_{k_1}\cdot u_{<k_1-10}\big)
\end{equation}

Hence we have 
\begin{align*}
&\sum_{k_2}\Pi_{\tilde{u}_{\perp}}\big((-\triangle)^{\frac{1}{2}}u\big)(-\triangle)^{\frac{1}{2}}(u_{<k_2-10}\cdot u_{k_2})\\& = -\sum_{|k_3 - k_4|<10}\frac12\Pi_{\tilde{u}_{\perp}}\big((-\triangle)^{\frac{1}{2}}u\big)(-\triangle)^{\frac{1}{2}}(u_{k_3}\cdot u_{k_4}).
\end{align*}

Thereafter, we can deduce
\begin{align*}
&2^{k(s-1)}\big\|P_k\big[\sum_{\substack{|k_3 - k_4|<10\\ k_3<k-20}}\frac12\Pi_{\tilde{u}_{\perp}}\big((-\triangle)^{\frac{1}{2}}u\big)(-\triangle)^{\frac{1}{2}}(u_{k_3}\cdot u_{k_4})\big]\big\|_{L_t^1 L_x^2}\\
&\lesssim 2^{k(s-1)} \sum_{\substack{|k_3 - k_4|<10\\ k_3<-20}}\big\|P_{[k-10,k+10]}\big[\Pi_{\tilde{u}_{\perp}}\big((-\triangle)^{\frac{1}{2}}u\big)\big]\big\|_{L_t^\infty L_x^2}\big\|(-\triangle)^{\frac{1}{2}}(u_{k_3}\cdot u_{k_4})\big\|_{L_t^1 L_x^\infty} \\
&\lesssim T^{\frac{\gamma}{2(2+\gamma)}}  \sum_{\substack{|k_3 - k_4|<10\\ k_3<k-20}}2^{k_3}\big\|u_{k_3}\big\|_{L_t^{2+\gamma} L_x^\infty}\big\|u_{k_4}\big\|_{L_t^2L_x^\infty}  \big( (1+\big\|\tilde{u}\big\|_{S_{loc}})\sum_{k_3}2^{-\sigma|k-k_3|}\big\|P_{k_3}u\big\|_{S_{loc}} \big) \\
&\lesssim T^{\frac{\gamma}{2(2+\gamma)}}  \sum_{\substack{|k_3 - k_4|<10\\ k_3<k-20}} 2^{(\frac{n}{2}-s)(k_3+k_4)} 2^{k_3-\frac{k_3}{2+\gamma}} 2^{-\frac{k_4}{2}} \|u_{k_3}\|_{S_{loc}}  \|u_{k_4}\|_{S_{loc}} \\
&\big( (1+\big\|\tilde{u}\big\|_{S_{loc}})\sum_{k_3}2^{-\sigma|k-k_3|}\big\|P_{k_3}u\big\|_{S_{loc}} \big) \\
&\lesssim  T^{\frac{\gamma}{2(2+\gamma)}} \|u\|_{S_{loc}}^2 \big( (1+\big\|\tilde{u}\big\|_{S_{loc}})\sum_{k_3}2^{-\sigma|k-k_3|}\big\|P_{k_3}u\big\|_{S_{loc}} \big)
\end{align*}

When $k_3>k-20$, we put both $u_{k_{3}}$ into $L_t^2 L_x^4$, and $u_{k_4}$ into $L_t^{2+\gamma} L_x^4$ to get desired result. So we conclude the estimate of \eqref{eq:locmultilin2_s2}. 

{\it{Proof of \eqref{eq:locmultilin3_s2}}}. \par
 We first get rid of the projection operator $\Pi_{\tilde{u}^{\perp}}$. As in the Lemma \ref{lem:locmultilin2}, we have
\[
\big\|P_k\big[\Pi_{\tilde{u}^{\perp}}F\big]\big\|_{L_t^1 \dot{H}_x^{s-1}}\lesssim (1+\big\|\tilde{u}\big\|_{S_{loc}})\sum_{k_3 \in \Z}2^{-\sigma|k-k_3|}\big\|P_{k_3}F\big\|_{L_t^1\dot{H}_x^{s-1}}.
\]

Therefore, it suffices to show that 
\begin{equation}
    \begin{split}
        &\|\sum_{k_1,k_2 \in \Z}\ P_k \big[ u \times \Lah \left(u_{k_1} \times \Lah u_{k_2} \right) - u \times \left(u_{k_1} \times (-\Delta) u_{k_2} \right) \big]\|_{L_t^1 \dot{H}^{s-1}} \\ 
        &\lesssim  T^{\frac{\gamma}{2(2+\gamma)}}  \|u\|_{S_{loc}}^3 
    \end{split}
\end{equation}

We consider different cases in the following. For $k_1>k+10$, we first split it into 
\begin{align*}
    &2^{k(s-1)}\big\|\sum_{\substack{k_1>k+10 \\ k_2 \in \Z}}P_k\big[u \times (-\triangle)^{\frac{1}{2}}(u_{k_1} \times (-\triangle)^{\frac{1}{2}}u_{k_2}) - u \times(u_{k_1} \times (-\triangle)u_{k_2})\big]\big\|_{L_t^1 L_x^2}\\
    &=2^{k(s-1)}\big\|\sum_{\substack{|k_1 - k_2|<5\\k_1>k+10}}P_k\big[u \times (-\triangle)^{\frac{1}{2}}(u_{k_1} \times (-\triangle)^{\frac{1}{2}}u_{k_2}) - u \times(u_{k_1} \times (-\triangle)u_{k_2})\big]\big\|_{L_t^1 L_x^2}\\
    &+2^{k(s-1)}\big\|\sum_{\substack{k_1>k+10 \\ k_2<k_1-5}  }P_k\big[u \times (-\triangle)^{\frac{1}{2}}(u_{k_1} \times (-\triangle)^{\frac{1}{2}}u_{k_2}) - u \times(u_{k_1} \times (-\triangle)u_{k_2})\big]\big\|_{L_t^1 L_x^2} \\
    &+2^{k(s-1)}\big\|\sum_{\substack{k_1>k+10 \\ k_2>k_1+5}  }P_k\big[u \times (-\triangle)^{\frac{1}{2}}(u_{k_1} \times (-\triangle)^{\frac{1}{2}}u_{k_2}) - u \times(u_{k_1} \times (-\triangle)u_{k_2})\big]\big\|_{L_t^1 L_x^2}
\end{align*}
For the first term, we need to use Lemma \ref{lem:multilinestimates} again, in particular, \eqref{singlefreq} and \eqref{singlefreq2}. We have
\begin{align*}
&2^{k(s-1)} \big\|\sum_{\substack{|k_1 - k_2|<5\\k_1>k+10}}P_k\big[u \times (-\triangle)^{\frac{1}{2}}(u_{k_1} \times (-\triangle)^{\frac{1}{2}}u_{k_2}) - u \times(u_{k_1} \times (-\triangle)u_{k_2})\big]\big\|_{L_t^1 L_x^2}\\
&\lesssim \|u\|_{L_{t,x}^\infty} \ 2^{k(s-1)} \ \sum_{\substack{|k_1 - k_2|<5\\k_1>k+10}} 2^{k_1+k_2}\big\|P_{k_1}u\big\|_{L_t^2 L_x^4}\big\|u_{k_2}\big\|_{L_t^2 L_x^4} \\
&\lesssim T^{\frac{\gamma}{2(2+\gamma)}} \|u\|_{S_{loc}} \sum_{\substack{|k_1 - k_2|<5\\k_1>k+10}}2^{k(s-1)} 2^{(\frac{n}{4}-s-\frac{1}{2+\gamma})k_1} 
2^{(\frac{n}{4}-s-\frac{1}{2})k_2} \ \big\|u_{k_1}\big\|_{S_{loc}} \big\|u_{k_2}\big\|_{S_{loc}}\\
&\lesssim  T^{\frac{\gamma}{2(2+\gamma)}} \|u\|_{S_{loc}}^3
\end{align*}

Then, for the high-low intersection second term, we derive
\begin{align*}
&2^{k(s-1)}\big\|\sum_{\substack{k_1>k+10 \\ k_2<k_1-5}  }P_k\big[u \times (-\triangle)^{\frac{1}{2}}(u_{k_1} \times (-\triangle)^{\frac{1}{2}}u_{k_2}) - u \times(u_{k_1} \times (-\triangle)u_{k_2})\big]\big\|_{L_t^1 L_x^2}\\
    &\lesssim 2^{k(s-1)} \sum_{\substack{k_1>k+10 \\ k_2<k_1-5}} \|u_{k_1}\|_{L_t^2L_x^\infty}\  2^{k_1+k_2} \|u_{k_1}\|_{L_t^\infty L_x^2} \ \|u_{k_2}\|_{L_t^{2+\gamma} L_x^\infty} \\
    &\lesssim  2^{k(s-1)} \sum_{k_1>k+10} 2^{(\frac{n}{2}-2s+1)k_1} \|u_{k_1}\|^2_{S_{loc}} \ \sum_{k_2<k_1-5} 2^{(\frac{n}{2}-s+\frac{1+\gamma}{2+\gamma})k_2} 2^{-\frac{k_1}{2}}\|u_{k_2}\|_{S_{loc}} \\
    &\lesssim T^{\frac{\gamma}{2(2+\gamma)}} \|u\|^3_{S_{loc}}
 \end{align*}
 
The leftover case, when $k_1>k+10, \ k_2>k_1+5$, can be estimated symmetrically as
\begin{align*}
    &2^{k(s-1)}\big\|\sum_{\substack{k_1>k+10 \\ k_2>k_1+5}} P_k\big[u \times (-\triangle)^{\frac{1}{2}}(u_{k_1} \times (-\triangle)^{\frac{1}{2}}u_{k_2}) - u \times(u_{k_1} \times (-\triangle)u_{k_2})\big]\big\|_{L_t^1 L_x^2}\\
        &\lesssim 2^{k(s-1)}\sum_{\substack{k_1>k+10 \\ k_2>k_1+5}} \|u_{k_2}\|_{L_t^2L_x^\infty}\  2^{k_1+k_2} \|u_{k_1}\|_{L_t^{2+\gamma} L_x^\infty} \ \|u_{k_2}\|_{L_t^\infty L_x^2} \\
         &\lesssim  2^{k(s-1)} \sum_{k_2>k+10} 2^{(\frac{n}{2}-2s+1)k_2} \|u_{k_2}\|^2_{S_{loc}} \ \sum_{k_1<k_2-5} 2^{(\frac{n}{2}-s+\frac{1+\gamma}{2+\gamma})k_1} 2^{-\frac{k_2}{2}}\|u_{k_1}\|_{S_{loc}} \\
        &\lesssim T^{\frac{\gamma}{2(2+\gamma)}} \|u\|^3_{S_{loc}}
     \end{align*}

Hence, we conclude that the case when $k_1>k+10$. Moreover, the above estimates can also be used to get the estimate for $k_2>k+10$. Therefore, we conclude the case when both frequencies are larger. \par 

Next, we consider the case for $k_1,k_2<k-10$. By Lemma~\ref{lem:multilinestimates}, with respect to the $N_{loc}$ norm, we have
\begin{align*}
&P_k\big[ u \times (-\triangle)^{\frac12}(u_{k_1} \times (-\triangle)^{\frac12}u_{k_2})\big] \simeq
P_k\big[ u \times ((-\triangle)^{\frac12}u_{k_1} \times (-\triangle)^{\frac12}u_{k_2})\big] +T^{\frac{\gamma}{2(2+\gamma)}} \|u\|^3_{S_{loc}}, \\
&P_k\big[ u \times \left(u_{k_1} \times (-\Delta) u_{k_2} \right)    \big] \simeq
P_k\big[ u \times ((-\triangle)^{\frac12}u_{k_1} \times (-\triangle)^{\frac12}u_{k_2})\big] +T^{\frac{\gamma}{2(2+\gamma)}} \|u\|^3_{S_{loc}}.
\end{align*}

Therefore, we consider the following instead.
\begin{align*}
    &2^{k(s-1)} \big\| \sum_{k_1,k_2<k-10} P_k\big[ P_{[k-5,k+5]}u \times \Lah u_{k_1} \times \Lah u_{k_2}\big] \big\|_{L_t^1 L_x^2}  \\
    &\lesssim 2^{k(s-1)} T^{\frac{\gamma}{2(2+\gamma)}} \sum_{k_1,k_2<k-10} \|P_{[k-5,k+5]}u\|_{L_t^{\infty} L_x^2} \|\Lah u_{k_1}\|_{L_t^{2+\gamma} L_x^\infty} \|\Lah u_{k_2}\|_{L_t^2 L_x^\infty} \\
    &\lesssim T^{\frac{\gamma}{2(2+\gamma)}} \sum_{k_1,k_2<k-10}\ 2^{-k} \|u_k\|_{S_{loc}}  2^{(1-\frac{1}{2+\gamma}+\frac{n}{2}-s)k_1} \|u_{k_1}\|_{S_{loc}} 2^{(\frac{1}{2}+\frac{n}{2}-s)k_1} \|u_{k_2}\|_{S_{loc}} \\
    &\lesssim T^{\frac{\gamma}{2(2+\gamma)}} \|u\|^3_{S_{loc}}.
\end{align*}

Finally, we consider the scenario when one frequency is intermediate in $[k-10,k+10]$ and the other is small. \par
{\it{(i): $k_1\in [k-10,k+10], k_2<k+10$.}} We split it into two cases:
\begin{align*}
    &2^{k(s-1)}\big\|\sum_{\substack{k_1 \in[k-10,k+10] \\ k_2<k+10}  }P_k\big[u\times (-\triangle)^{\frac{1}{2}}(u_{k_1} \times (-\triangle)^{\frac{1}{2}}u_{k_2}) - u \times(u_{k_1} \times (-\triangle)u_{k_2})\big]\big\|_{L_t^1 L_x^2}\\
    &=2^{k(s-1)}\big\|\sum_{\substack{k_1 \in[k-10,k+10] \\ k_2<k+10}  }P_k\big[u_{\geq k_2-10} \times (-\triangle)^{\frac{1}{2}}(u_{k_1} \times (-\triangle)^{\frac{1}{2}}u_{k_2}) \\
    &- u_{\geq k_2-10} \times(u_{k_1} \times (-\triangle)u_{k_2})\big]\big\|_{L_t^1 L_x^2}\\
    &+2^{k(s-1)}\big\|\sum_{\substack{k_1 \in[k-10,k+10] \\ k_2<k+10}  }P_k\big[u_{<k_2-10} \times (-\triangle)^{\frac{1}{2}}(u_{k_1} \times (-\triangle)^{\frac{1}{2}}u_{k_2}) \\
    &- u_{<k_2-10} \times(u_{k_1} \times (-\triangle)u_{k_2})\big]\big\|_{L_t^1 L_x^2}
\end{align*}

Firstly, we have

\begin{align*}
    &2^{k(s-1)}\big\|\sum_{\substack{k_1 \in[k-10,k+10] \\ k_2<k+10}  }P_k\big[u_{\geq k_2-10} \times (-\triangle)^{\frac{1}{2}}(u_{k_1} \times (-\triangle)^{\frac{1}{2}}u_{k_2}) - u_{\geq k_2-10} \times(u_{k_1} \times (-\triangle)u_{k_2})\big]\big\|_{L_t^1 L_x^2}\\
        &\lesssim 2^{k(s-1)} T^{\frac{\gamma}{2(2+\gamma)}} \sum_{\substack{k_1 \in[k-10,k+10] \\ k_2<k+10}} \|u_{\geq k_2-10}\|_{L_t^2L_x^\infty}\  2^{k_1+k_2} \|u_{k_1}\|_{L_t^\infty L_x^2} \ \|u_{k_2}\|_{L_t^{2+\gamma} L_x^\infty} \\
        &\lesssim T^{\frac{\gamma}{2(2+\gamma)}} \|u_{k}\|_{S_{loc}} \sum_{\substack{ k_2<k+10 \\ k_3 \geq k_2-10}} 2^{(-\frac{1}{2}+\frac{n}{2}-s)k_3} \ 2^{(1-\frac{k_2}{2+\gamma}+\frac{n}{2}-s)k_2}    \ \|u_{k_2}\|_{S_{loc}}\|u_{k_3}\|_{S_{loc}} \\
        &\lesssim T^{\frac{\gamma}{2(2+\gamma)}} \|u\|^3_{S_{loc}}
     \end{align*}

Move to the second part, using the Lemma~\ref{lem:multilinestimates} again. We know that 
\begin{align*}
    (-\triangle)^{\frac{1}{2}}(u_{k_1} \times (-\triangle)^{\frac{1}{2}}u_{k_2}) - (u_{k_1} \times (-\triangle)u_{k_2}) \simeq \Lah u_{k_1} \times \Lah u_{k_2}.
\end{align*}

By the cross product rule $$a \times (b \times c)= b(a\cdot c)-c(a\cdot b),$$ 

we have  
\begin{align*}
    &P_k\big[u_{<k_2-10} \times (-\triangle)^{\frac{1}{2}}(u_{k_1} \times (-\triangle)^{\frac{1}{2}}u_{k_2}) - u_{<k_2-10} \times(u_{k_1} \times (-\triangle)u_{k_2})\big] \\
   &\simeq P_k\big[ u_{<k_2-10} \times \Lah u_{k_1} \times \Lah u_{k_2}\big] \\
   &= P_k\big[\Lah u_{k_1} (u_{<k_2-10} \cdot \Lah u_{k_2}) \big] -P_k \big[\Lah u_{k_2} (u_{<k_2-10}  \cdot \Lah u_{k_1}) ]
\end{align*}

For the first term, we have 
\begin{align*}
    &2^{k(s-1)}\big\|\sum_{\substack{k_1 \in[k-10,k+10] \\ k_2<k+10}  }P_k\big[\Lah u_{k_1} (u_{<k_2-10} \cdot \Lah u_{k_2}) \big]\big\|_{L_t^1 L_x^2} \\
    &\lesssim 2^{k(s-1)} \sum_{k_1 \in[k-10,k+10]} \|\Lah u_{k_1}\|_{L_t^\infty L_x^2} \sum_{k_2<k+10} \|P_{k_2}(u_{<k_2-10} \cdot \Lah u_{k_2})\|_{L_t^1L_x^\infty} \\
    &\lesssim \|u\|_{S_{loc}} \ \sum_{k_2<k+10} \|P_{k_2}(u_{<k_2-10} \cdot \Lah u_{k_2})\|_{L_t^1L_x^\infty} \\
    &\lesssim \|u\|_{S_{loc}} \sum_{k_2<k+10} \|\Lah P_{k_2}(u_{<k_2-10} \cdot  u_{k_2})\|_{L_t^1L_x^\infty} + T^{\frac{\gamma}{2(2+\gamma)}} \|u\|_{S_{loc}}^2
\end{align*}

By the geometric property property \eqref{eq:locorthomicro}, we know that
\begin{equation}\label{eq:locorthomicrolocalized}
0=P_{k}(u \cdot u) = 2 P_{k} (u_{[k-10,k+10]}\cdot u_{<k-10}) + P_{k}(u_{>k+10} \cdot u_{>k+10}).
\end{equation}

After that, we can replace
\begin{equation}
    \label{locerror3_local}
    P_{k_2}(u_{<{k_2}-10} \cdot u_{[{k_2}-10,{k_2}+10]}) = -\frac{1}{2} P_{k_2}(u_{\geq{k_2}+10} \cdot u_{\geq {k_2}+10}).
\end{equation}

Hence, we finish this term by
\begin{align*}
    &\sum_{k_2<k+10} \|\Lah P_{k_2}(u_{<k_2-10} \cdot  u_{k_2})\|_{L_t^1L_x^\infty} \\
    &\simeq \sum_{k_2<k+10} \|\Lah P_{k_2}(u_{\geq k_2+10} \cdot  u_{\geq k_2+10})\|_{L_t^1L_x^\infty} \\
    &\lesssim T^{\frac{\gamma}{2(2+\gamma)}} \sum_{\substack{k_2<k+10 \\ k_3 \geq k_2+10}} 2^{k_2} 2^{-(\frac{1}{2}+\frac{1}{2+\gamma})k_3} 2^{(n-2s)k_3} \|u_{k_3}\|_{S_{loc}}^2 \\
    &\lesssim T^{\frac{\gamma}{2(2+\gamma)}} \|u\|_{S_{loc}}^2.
\end{align*}



The second part can be treated similarly by the Lemma~\ref{lem:multilinestimates} and the geometric property. Thus we conclude this case. \par

{\it{(ii)}}: $k_2\in [k-10,k+10], k_1<k+10$. As shown in the case {\it{(i)}}, we further split the case with respect to the frequency of the first term involving $u$. When $u_{\geq k_1-10}$, it can be estimate as above. Thus, we focus on the case for $u_{< k_1-10}$.


By the Lemma~\ref{lem:multilinestimates}, we have 
\begin{align*}
&2^{k(s-1)} \big\|\sum_{\substack{k_1<k+10 \\ k_2 \in [k-10,k+10]}} P_k\big[u_{<k_1-10} \times (-\triangle)^{\frac{1}{2}}(u_{k_1} \times (-\triangle)^{\frac{1}{2}}u_{k_2}) \\
& - (-\triangle)^{\frac{1}{2}}\big(u_{<k_1-10} \times (u_{k_1} \times (-\triangle)^{\frac{1}{2}}u_{k_2})\big)\big]\big\|_{L_t^1 L_x^2}\\
& \lesssim  2^{k(s-1)} T^{\frac{\gamma}{2(2+\gamma)}} \sum_{\substack{k_1<k+10 \\ k_2 \in [k-10,k+10]}} \big\|(-\triangle)^{\frac{1}{2}}u_{<k_1-10}\big\|_{L_t^2 L_x^\infty}\big\|u_{k_1}\big\|_{L_t^{2+\gamma} L_x^\infty}\big\|(-\triangle)^{\frac{1}{2}}u_{k_2}\big\|_{L_t^\infty L_x^2} \\
 &\lesssim T^{\frac{\gamma}{2(2+\gamma)}} \|u_k\|_{S_{loc}} \ \sum_{\substack{k_1<k+10 \\ k_3 <k_1-10}} 2^{(\frac{1}{2}+\frac{n}{2}-s)k_3} \|u_{k_3}\|_{S_{loc}} \ 2^{(-\frac{1}{2+\gamma}+\frac{n}{2}-s)k_1} \|u_{k_1}\|_{S_{loc}}   \\
 &\lesssim T^{\frac{\gamma}{2(2+\gamma)}} \|u\|_{S_{loc}}^3
\end{align*}

Therefore, we can consider $P_k \big[(-\triangle)^{\frac{1}{2}}\big(u_{<k_1-10} \times (u_{k_1} \times (-\triangle)^{\frac{1}{2}}u_{k_2})\big)\big]$ instead. Using the cross product rule, we reformulate the expression as 

\begin{equation}\label{eq:locthefourterms}\begin{split}
&P_k\big[(-\triangle)^{\frac{1}{2}}\big(u_{<k_1-10} \times (u_{k_1} \times (-\triangle)^{\frac{1}{2}}u_{k_2})\big) - u_{<k_1-10} \times(u_{k_1} \times (-\triangle)u_{k_2})\big]\\
& = P_k \big[ (-\triangle)^{\frac{1}{2}}\big( u_{k_1}(u_{<k_1-10}\cdot (-\triangle)^{\frac{1}{2}}u_{k_2}) -  (-\triangle)^{\frac{1}{2}}u_{k_2}(u_{<k_1-10}\cdot u_{k_1}) \big) \big]\\
& - P_k[\big(u_{k_1}( u_{<k_1-10}\cdot (-\triangle)u_{k_2}) - (-\triangle)u_{k_2}(u_{<k_1-10} \cdot u_{k_1})\big)] \\
&= P_k \big[ (-\triangle)^{\frac{1}{2}}\big( u_{k_1}(u_{<k_1-10}\cdot (-\triangle)^{\frac{1}{2}}u_{k_2}) -  \big(u_{k_1}( u_{<k_1-10}\cdot (-\triangle)u_{k_2} \big) \big]\\
&+P_k\big[ (-\triangle)u_{k_2}(u_{<k_1-10} \cdot u_{k_1})\big) -\Lah\big( \Lah u_{k_2} (u_{<k_1-10}\cdot u_{k_1}) \big) \big].
\end{split}\end{equation}

Furthermore, we also know the differences 
\begin{align*}
    \| u_{<k_1-10} \cdot \Lah u_{k_2} - \Lah(u_{<k_1-10} \cdot u_{k_2}) \|_{L_t^2 L_x^2}
 \end{align*}

 and

 \begin{align*}
     &\| u_{<k_1-10} \cdot (-\triangle) u_{k_2}-  (-\triangle) (u_{<k_1-10} \cdot u_{k_2}) \|_{L_t^2 L_x^2} 
 \end{align*}
are all desirable bounds so that we can replace them respectably.

Additionally, by putting $u_{k_1}$ into $L_t^2 L_x^\infty$, we also know that
\begin{align*}
&P_k\big[(-\triangle)^{\frac{1}{2}}\big(u_{k_1}(u_{<k_1-10}\cdot (-\triangle)^{\frac{1}{2}}u_{k_2})\big) \big]-  P_k\big[u_{k_1}( u_{<k_1-10}\cdot (-\triangle)u_{k_2})\big]\\
& \simeq P_k \big[(-\triangle)^{\frac{1}{2}}\big(u_{k_1}(-\triangle)^{\frac{1}{2}}(u_{<k_2-10}\cdot u_{k_2}) \big] - P_k \big[u_{k_1}(-\triangle)( u_{<k_2-10}\cdot u_{k_2})\big] 
\end{align*}

Then we can use the geometric property \eqref{eq:locorthomicro} to replace
$$
P_{k_2}(u_{<k_2-10}\cdot u_{k_2}) = -\frac{1}{2} P_{k_2} (u_{k_2+10} \cdot u_{k_2+10}).
$$

We conclude that

\begin{align*}
    &2^{k(s-1)} \big\| \sum_{\substack{k_1<k+10 \\k_2\in[k-10,k+10]}} P_k \big[(-\triangle)^{\frac{1}{2}}\big(u_{k_1}(-\triangle)^{\frac{1}{2}}(u_{<k_2-10}\cdot u_{k_2}) \big] \\ 
    &- P_k\big[u_{k_1}(-\triangle)( u_{<k_2-10}\cdot u_{k_2})\big]\big\|_{L_t^1 L_x^2} \\
&\lesssim 2^{k(s-1)} T^{\frac{\gamma}{2(2+\gamma)}} \sum_{\substack{k_1<k+10 \\k_2\in[k-10,k+10]}} \big\|(-\triangle)^{\frac{1}{2}}u_{k_1}\big\|_{L_t^2 L_x^\infty} \ 2^{k_2}\ \big\|u_{>k_2+10}\big\|_{L_t^{2+\gamma} L_x^\infty}\big\|u_{>k_2+10}\big\|_{L_t^\infty L_x^2} \\
&\lesssim \sum_{k_1<k+10} \ 2^{(\frac{n}{2}+\frac{1}{2}-s)k_1} \|u_{k_1}\|_{S_{loc}} \ \sum_{\substack{k_2\in[k-10,k+10] \\k_3>k_2+10}} \  2^{ks} 2^{(\frac{n}{2}-2s-\frac{1}{2+\gamma})k_3} \|u_{k_3}\|^2_{S_{loc}} \\
&\lesssim T^{\frac{\gamma}{2(2+\gamma)}} \|u\|^3_{S_{loc}}.
\end{align*}

It only remains to bound the difference 
\begin{equation*}
    P_k\big[ (-\triangle)u_{k_2}(u_{<k_1-10} \cdot u_{k_1})\big) -\Lah\big( \Lah u_{k_2} (u_{<k_1-10}\cdot u_{k_1}) \big) \big].
\end{equation*}

As before, we replace  
\begin{equation*}
    P_{k_1}(u_{<k_1-10} \cdot u_{k_1}) = -\frac{1}{2} P_{k_1}(u_{>k_1+10} \cdot u_{>k_1+10}).
\end{equation*}

Then we deduce
\begin{align*}
    & 2^{k(s-1)}  \| \sum_{\substack{k_1 <k+10 \\k_2\in [k-10,k+10]} } P_k\big[ (-\triangle)u_{k_2}P_{k_1}(u_{>k_1+10} \cdot u_{>k_1+10} ) \\
    &-\Lah \big( (-\triangle)^{\frac{1}{2}}u_{k_2} P_{k_1}(u_{>k_1+10} \cdot u_{>k_1+10} ) \big)\big]\|_{L_t^1 L_x^2} \\
    &\lesssim 2^{k(s-1)} \sum_{\substack{k_1 <k+10 \\k_2\in [k-10,k+10]} } 2^{k_1+k_2} \| P_{k_1}(u_{>k_1+10} \cdot u_{>k_1+10} ) \|_{L_t^1 L_x^\infty} \| u_{k_2}\|_{L_t^\infty L_x^2} \\
    &\lesssim T^{\frac{\gamma}{2(2+\gamma)}} \|u_{k_2}\|_{S_{loc}} \ \sum_{\substack{k_1 <k+10 \\k_2\in [k-10,k+10]} } 2^{k_1} 2^{(-\frac{1}{2+\gamma} -\frac{1}{2}+n-2s)k_3} \|u_{k_3}\|^2_{S_{loc}}\\
    &\lesssim T^{\frac{\gamma}{2(2+\gamma)}} \|u\|^3_{S_{loc}}.
    \end{align*}
    Hence, we conclude the multilinear estimates. 


\bibliographystyle{AIMS} 
\bibliography{paper}
\end{document}

%% file: notation.tex
\renewcommand{\epsilon}{\varepsilon}
\renewcommand{\hat}{\widehat}
\renewcommand{\tilde}{\widetilde}
\renewcommand{\bar}{\overline}
\renewcommand{\le}{\leqslant}
\renewcommand{\ge}{\geqslant}

\def\Laq{(-\Delta)^{\frac{1}{4}}}
\def\Lah{(-\Delta)^{\frac{1}{2}}}

\newcommand{\besov}{\dot{B}^{\frac{n}{2}}_{2,1}}

\newcommand{\R}{\mathbb{R}}
\newcommand{\Z}{\mathbb{Z}}
\newcommand{\Hy}{\mathbb{H}^2}
\newcommand{\be}{\begin{equation}}
\newcommand{\ee}{\end{equation}}